\documentclass[10pt]{article}
\usepackage[utf8]{inputenc}

\usepackage[left=1.4in,top=1.3in,right=1.3in,bottom=1.3in,footskip = 0.333in]{geometry}

\usepackage{authblk, amsmath, amsthm, amsfonts, amssymb, amsbsy, bbm, bigstrut, graphicx, enumerate,  upref, longtable, comment, comment,xcolor,siunitx,mathrsfs}
\usepackage{float}
\usepackage{tabularx, booktabs, makecell, caption}
\usepackage{siunitx}
\usepackage{algorithm}

\usepackage{subcaption}
\usepackage[breaklinks]{hyperref}

\usepackage[numbers, sort&compress]{natbib} 
\usepackage{comment}
\usepackage[export]{adjustbox}

\newtheorem{thm}{Theorem}
\newtheorem{lemma}{Lemma}

\newtheorem{prop}{Proposition}
\newtheorem{defn}{Definition}

\theoremstyle{definition}
\newtheorem{remark}{Remark}

\newtheorem{ass}{Assumption}

\numberwithin{thm}{section}
\numberwithin{remark}{section}
\numberwithin{ex}{section}
\numberwithin{ass}{section}
\numberwithin{defn}{section}
\numberwithin{lemma}{section}
\numberwithin{corollary}{section}
\numberwithin{prop}{section}

\newcommand{\argmax}{\operatorname{argmax}}

\newcommand{\bx}{\mathbf{x}}
\newcommand{\bA}{\mathbf{A}}
\newcommand{\by}{\mathbf{y}}

{\tiny {\tiny }}

\newcommand{\ee}{\mathbb{E}}

\newcommand{\rr}{\mathbb{R}}

\allowdisplaybreaks

\newcommand{\bE}{ \mathbb{E} }

\newcommand{\bv}{ { v} }

\newcommand{\bX}{ { X} }

\newcommand{\bP}{\mathbb{P}}

\def \bX{\mathbf{X}}
\def \bY{\mathbf{Y}}
\def \bu{\mathbf{u}}
\def \bv{\mathbf{v}}

\def \bt{\mathbf{t}}

\def \bT{\mathbf{T}}

\DeclareMathOperator{\sech}{sech}
\def \btheta {\boldsymbol{\theta}}

\newcommand{\independent}{\perp\!\!\!\!\perp}

\makeatletter
\def\thickhline{%
	\noalign{\ifnum0=`}\fi\hrule \@height \thickarrayrulewidth \futurelet
	\reserved@a\@xthickhline}
\def\@xthickhline{\ifx\reserved@a\thickhline
	\vskip\doublerulesep
	\vskip-\thickarrayrulewidth
	\fi
	\ifnum0=`{\fi}}
\makeatother

\newlength{\thickarrayrulewidth}
\usepackage{algorithm}

\numberwithin{equation}{section}

\renewcommand{\hat}{\widehat}
\renewcommand{\tilde}{\widetilde}

\begin{document}
\title{Causal effect estimation under network interference with mean-field methods}

\author[1]{Sohom Bhattacharya\thanks{bhattacharya.s@ufl.edu}}
\author[2]{Subhabrata Sen \thanks{subhabratasen@fas.harvard.edu}}
\affil[1]{Department of Statistics, University of Florida}
\affil[2]{Department of Statistics, Harvard University}
\makeatletter
\renewcommand\AB@affilsepx{\\ \protect\Affilfont}
\renewcommand\Authsep{, }
\renewcommand\Authand{, }
\renewcommand\Authands{, }
\makeatother

\maketitle 

\begin{abstract}
   We study causal effect estimation from observational data under interference. The interference pattern is captured by an observed network. We adopt the chain graph framework of \cite{tchetgen2021auto}, which allows (i) interaction among the outcomes of distinct study units connected along the graph and (ii) long range  interference, whereby the outcome of an unit may depend on the treatments assigned to distant units connected along the interference network. For ``mean-field" interaction networks, we develop a new scalable iterative algorithm to estimate the causal effects. For gaussian weighted networks, we introduce a novel causal effect estimation algorithm based on Approximate Message Passing (AMP). Our algorithms are provably consistent under a ``high-temperature" condition on the underlying model. 
   We estimate the (unknown) parameters of the model from data using maximum pseudo-likelihood and  
   establish $\sqrt{n}$-consistency of this estimator in all parameter regimes. 
   Finally, we prove that the downstream estimators obtained by plugging in estimated parameters into the aforementioned algorithms are consistent at high-temperature. Our methods can accommodate dense interactions among the study units---a setting beyond reach using existing techniques. Our algorithms originate from the study of variational inference approaches in high-dimensional statistics; overall, we demonstrate the usefulness of these ideas in the context of causal effect estimation under interference.

   
\end{abstract}

\section{Introduction}
Estimation of causal effects from observational data is of fundamental importance across the natural and social sciences. Estimation methods were traditionally developed under the Stable Unit Treatment Values Assumption (SUTVA) \cite{rubin1990formal}; formally, this assumes that the outcomes of an experimental unit are unaffected by the treatments assigned to the other study units. 

The SUTVA assumption is unrealistic in many modern applications e.g., in social networks, treatments assigned to a study unit are observed by the neighbors and this, in turn, directly influences their behavior. Similarly, in randomized clinical trials, treatments assigned to an individual often influence the outcomes of other individuals living in close geographic proximity. This has motivated an in-depth investigation into causal effect estimation from observational data under interference (see e.g. \cite{tchetgen2021auto,ogburn2024causal,emmenegger2022treatment,forastiere2021identification} and references therein). To ensure the success of the proposed methodology, the existing literature inevitably constrains the degree of interaction among the study units. Typically, this is formalized via strong constraints on the maximum degree of the interaction graph among the study units. Unfortunately, these assumptions are typically asymptotic, and thus impossible to verify in finite (but large) populations. To the best of our knowledge, consistent causal effect estimators are unknown under interference problems with dense interactions. 

This manuscript aims to address this critical gap in the existing literature. We develop causal effect estimators with rigorous guarantees under dense interference models, provided the overall interaction strengths are relatively \emph{weak}. We focus on the Neyman-Rubin potential outcomes framework, and consider a study with $n$ units. Throughout, we work with a binary treatment and denote the possible treatments as $\mathbf{t} := (t_1, \cdots, t_n) \in \{\pm 1\}^n$. We observe $\{(Y_i, T_i, \mathbf{X}_i) \in \mathbb{R} \times \{\pm 1\} \times [-1,1]^d : 1\leq i \leq n\}$, where $Y_i$, $T_i$ and $\mathbf{X}_i$ denote the response, assigned treatment and pre-treatment covariates of the $i^{th}$ study unit respectively. We set $\mathbf{Y} = (Y_1, \cdots, Y_n)$, $\mathbf{T} = (T_1, \cdots, T_n)$ and $\mathbf{X}^\top = (\mathbf{X}_1, \cdots, \mathbf{X}_n) \in \mathbb{R}^{d \times n}$. In our subsequent asymptotic analysis, we assume the number of study units $n \to \infty$ while the covariate dimension $d$ is fixed throughout. The potential outcomes are denoted as $\{\mathbf{Y}_i(\mathbf{t}) : \mathbf{t} \in \{\pm 1\}^n\}$. To accommodate possible interference, the potential outcomes $\mathbf{Y}_i(\cdot)$ are functions of the entire treatment profile $\mathbf{t}$, rather than the individual treatment $t_i$. Throughout, we will assume a network version of the traditional consistency assumption in causal inference i.e. $\mathbf{Y}(\mathbf{T}) = \mathbf{Y}$ a.e. For identifiability of the causal effects, we assume a network version of the conditional ignorability assumption 
\begin{align*}
    \mathbf{T} \independent \mathbf{Y}(\mathbf{t}) | \mathbf{X} \,\, \mathrm{for\, all } \,\, \mathbf{t} \in \{\pm 1\}^n. 
\end{align*}
Note that this assumption reduces to the standard no unmeasured confounding assumption in the case of no interference and i.i.d. study units. Finally, we will assume throughout that for all $n \geq 1$, there exists $\sigma_n >0$ such that $\mathbb{P}[\mathbf{T}= \mathbf{t} | \mathbf{X}] \geq \sigma_n >0$ for all $\mathbf{t} \in \{\pm 1\}^n$. This is the network version of the traditional positivity assumption in causal inference.

Next, we introduce the causal estimands of interest. Specifically, we study the \emph{direct} and the \emph{indirect/spillover} effect of the assigned treatments on the outcomes. 
To this end, we first introduce the average direct causal effect for unit $i$ upon
changing the unit’s treatment status from $t_i=-1$ to $t_i=1$:
\begin{equation}\label{eq:define_de_i}
    \text{DE}_{i}(\mathbf{t}_{-i}) := \mathbb{E}[\mathbf{Y}_i| \bt = (1, \bt_{-i})] - \mathbb{E}[\mathbf{Y}_i| \bt = (-1, \bt_{-i})]
\end{equation}
where $(c, \mathbf{t}_{-i})$, $c \in \{\pm 1\}$, denotes the binary vector where the $i^{th}$ entry is $c$ and the remaining entries are specified by $\mathbf{t}_{-i}$. 
Note that the direct effect $\mathrm{DE}_i(\cdot)$ is dependent on the treatment assignments of the other units $\mathbf{t}_i \in \{\pm 1\}^{n-1}$. Under the consistency, conditional ignorability and positivity assumptions introduced above, a network version of Robins's g-formula implies 
\begin{align*}
    \text{DE}_{i}(\mathbf{t}_{-i}) := \mathbb{E}_{\mathbf{X}} \big[\mathbb{E}[\mathbf{Y}_i| \mathbf{t}= (1,\mathbf{t}_{-i}),\mathbf{X}]-\mathbb{E}[\mathbf{Y}_i|\mathbf{t} = (-1,\mathbf{t}_{-i}),\mathbf{X}]\big],  
\end{align*}
and thus $\text{DE}_{i}(\mathbf{t}_{-i})$ can be expressed as a function of the observed data law. 
To define an averaged direct effect, following \cite{hudgens2008toward,tchetgen2012causal,tchetgen2021auto}, we average these effects over a hypothetical allocation probability measure $\pi$ on $\{\pm 1\}^{n-1}$:
\begin{equation}\label{eq:define_de}
    \text{DE}(\pi)=\frac{1}{n}\sum_{i=1}^{n}\sum_{\mathbf{t}_{-i}\in \{\pm 1\}^{n-1}} \pi(\mathbf{t}_{-i}) \text{DE}_i(\mathbf{t}_{-i}).
\end{equation}
Note that under the no interference setting i.e. if $\bY_i(\bt)= \bY_i(t_i)$, the direct effect $\mathrm{DE}(\pi)$ reduces to the traditional average treatment effect. 
Next, we define the average indirect or spillover causal effect experienced by unit $i$ if the unit’s treatment is set to be inactive, while changing the treatment of other units from inactive to $\mathbf{t}_{-i}$ :
\begin{equation}\label{eq:define_ie_i}
    \text{IE}_{i}(\mathbf{t}_{-i}) := \mathbb{E}[\bY_i| \bt = (-1, \bt_{-i})] - \mathbb{E}[\bY_i | \bt = -\mathbf{1}]. 
\end{equation}
Under the consistency, conditional ignorability and positivity assumptions, we have, 
\begin{align*}
    \text{IE}_{i}(\mathbf{t}_{-i}) := \mathbb{E}_\mathbf{X} \big[\mathbb{E}[\mathbf{Y}_i|\mathbf{t} = (-1, \mathbf{t}_{-i}), \mathbf{X}]-\mathbb{E}[\mathbf{Y}_i| \mathbf{t}= - \mathbf{1},\mathbf{X}]\big].  
\end{align*}
Thus the indirect effect $\text{IE}_{i}(\mathbf{t}_{-i})$ can also be expressed as a functional of the observed data law. 
Similar to direct effect, we average over the allocation $\pi$ to obtain  
\begin{equation}\label{eq:define_ie}
    \text{IE}(\pi)=\frac{1}{n}\sum_{i=1}^{n}\sum_{\mathbf{t}_{-i}\in \{\pm 1\}^{n-1}} \pi(\mathbf{t}_{-i}) \text{IE}_i(\mathbf{t}_{-i}).
\end{equation}
Observe that under no interference, i.e. if $\bY_i(\bt)= \bY_i(t_i)$, the indirect effect $\mathrm{IE}(\pi)=0$. 
In our subsequent discussion, we assume that the allocation measure $\pi$ appearing in \eqref{eq:define_de} and \eqref{eq:define_ie} is, in fact, the uniform distribution on $\{\pm 1\}^{n-1}$ i.e.,  $\pi(\mathbf{t}_{-i})=2^{-(n-1)}$ for all $\mathbf{t}_{-i} \in \{\pm 1\}^{n-1}$. Our arguments extend in a straightforward manner to any iid measure on $\{\pm 1\}^{n-1}$---we describe this extension in Remark~\ref{rmk:general_p}.   For notational simplicity, we suppress the dependence on $\pi$, and  write $\text{DE}$ and $\text{IE}$ in our subsequent discussion. Under the network versions of consistency, conditional ignorability and positivity introduced above, the network version of Robins's g-formula implies that the causal estimands $\mathrm{DE}$ and $\mathrm{IE}$ are functions of the observed data law (we refer the interested reader to 
\cite[Section 2.2]{tchetgen2021auto} for an in-depth discussion of this point). However, to ensure identifiability of these causal estimands, one needs additional structure on the observed data law.


Here we follow the Markov Random Field  (MRF) based framework introduced in \cite{tchetgen2021auto} and subsequently explored by \cite{bhattacharya2020causal,sherman2018identification}. Given covariates $\mathbf{x}^{\top} = (\mathbf{x}_1, \cdots, \mathbf{x}_n)$, $\mathbf{x}_i \in [-1,1]^d$ and a treatment assignment $\mathbf{t} \in \{\pm 1\}^n$, the potential outcomes are given by the joint density 
 \begin{equation}\label{eq:define_gibbs}
    f(\mathbf{y}|\mathbf{t},\mathbf{x}) = \frac{1}{Z_n(\mathbf{t},\mathbf{x})} \exp\Big(\frac{1}{2}\,\mathbf{y}^\top \mathbf{A}_n \mathbf{y}+ \mathbf{y}^\top (\tau_0 \mathbf{t}+ \mathbf{x} \btheta_0) \Big) \prod_{i=1}^{n}d\mu (y_i),
\end{equation}
where $\mu$ is a compactly supported probability measure on $[-1,1]$ and 
\begin{equation}\label{eq:define_Z}
    Z_n(\mathbf{t},\mathbf{x}) = \int _{[-1,1]^n}\exp\Big(\frac{1}{2}\,\mathbf{y}^\top \mathbf{A}_n \mathbf{y}+ \mathbf{y}^\top (\tau_0 \mathbf{t}+  \mathbf{x} \btheta_0) \Big) \prod_{i=1}^{n}d\mu (y_i)
\end{equation}
is the normalizing constant. We will assume throughout that the measure $\mu$ is non-degenerate i.e., $\mu$ has at least two points in its support. The matrix 
$\mathbf{A}_n = \bA_n^{\top} \in \mathbb{R}^{n \times n}$ captures the interaction among units which is assumed known throughout and $\tau_0\in \mathbb{R}$ , $\btheta_0 \in \rr^d$ represent unknown parameters.  In social network applications, the matrix $\mathbf{A}_n$ is usually a scaled version of the adjacency matrix of the observed network. Throughout, we make the following assumptions on the parameter space. 

\begin{ass}[Parameter space]\label{assn:parameter_space}
    $(\tau_0,\btheta_0)\in [-B,B]\times [-M,M]^d$ for some $B,M>0$.
\end{ass}

Given covariates $\mathbf{x} = (\mathbf{x}_1, \cdots, \mathbf{x}_n)$, 
 the treatment assignments $\mathbf{T} = (\mathbf{T}_1, \cdots, \mathbf{T}_n) \in \{ \pm 1\}^n$ follow a propensity score model 
\begin{equation}\label{eq:prop_score}
    \bP(\mathbf{T} = \mathbf{t} |\mathbf{x})= \frac{1}{Z_n'(\mathbf{x})} \exp \Big(\frac{1}{2}\mathbf{t}^\top \mathbf{M}_n \mathbf{t} + \sum_{i=1}^n t_i \mathbf{x}^\top_i \boldsymbol{\gamma}_0 \Big),
\end{equation}
with $\boldsymbol{\gamma_0} \in [-M,M]^d$ for $M>0$. $Z_n'(\mathbf{x})$ refers to the normalization constant in the above model. We assume that the interaction matrix $\mathbf{M}_n = \mathbf{M}_n^{\top}$ is known throughout, and the propensity score model is known up to the parameter $\boldsymbol{\gamma}_0$. Note that we do not necessarily assume that  $\mathbf{A}_n = \mathbf{M}_n$. 

Finally, we assume that the observed covariates $\bX_i \sim \mathbb{P}_X$ are i.i.d., where $\mathbb{P}_X$ is a probability distribution supported on $[-1,1]^d$. Assume that $\text{Var}(\bX_i)=\boldsymbol{\Sigma}$, where $\boldsymbol{\Sigma}$ is a $d \times d$ matrix with 
\begin{align}\label{eq:sigma_min_eval}
    C \ge \lambda_{\max}(\boldsymbol{\Sigma}) \ge \lambda_{\min}(\boldsymbol{\Sigma}) \ge c>0
\end{align}
for some $c,C>0.$

Before proceeding further, we collect some remarks about our modeling assumptions. First, if $\bA_n=0$ in \eqref{eq:define_gibbs} and $\mathbf{M}_n=0$ in \eqref{eq:prop_score}, we reduce back to the traditional no-interference setting in causal inference. 
Second, we note that in the special case $\mathbf{A}_n = \mathbf{M}_n$, our model is a special case of the general chain graph based formulation analyzed in \cite{tchetgen2021auto}. Chain graphs were introduced originally in the causal inference literature in \cite{lauritzen2002chain}. Although chain graphs are not directly compatible with DAG-based causal frameworks \cite{lauritzen2002chain}, they can be useful in specific situations \cite{ogburn2020causal}. In this work, we do not analyze the conceptual underpinnings of this MRF-based causal framework. Instead, we take this framework as given, and develop novel methodology for causal effect estimation within this scope. Third, we note that if $\mathbf{A}_n \neq \mathbf{M}_n$, this model is not necessarily compatible with the formulation of \cite{tchetgen2021auto}. However, our mathematical results are valid in this more general setup---to highlight this feature, we work under the more general setting throughout this paper. Fourth, we assume that the covariates $\mathbf{X}_i \in [-1,1]^d$ and $Y_i \in [-1,1]$---this is mainly for technical convenience. Our results generalize unchanged as long as the covariates and the outcomes are bounded. It should be possible to extend our results beyond the bounded setting with additional technical work; we defer this to future investigations. Finally, the outcome regression model \eqref{eq:define_gibbs} and the propensity score model \eqref{eq:prop_score} are arguably the simplest models with the framework of \cite{tchetgen2021auto}. We restrict to this simple setting as it already captures some of the central conceptual challenges. In addition, we believe that it should be possible to analyze more involved MRFs using the ideas introduced in this work---we defer a discussion of potential extensions to Sections~\ref{sec:discussion}.  

To see some of the challenges involved in estimating the direct effect $\mathrm{DE}$ and the indirect effect $\mathrm{IE}$ in this MRF-based framework \eqref{eq:define_gibbs}, \eqref{eq:prop_score}, we start with the following lemma.   

\begin{lemma}\label{lem:de_define}
Set $\pi(\mathbf{t}_{-i})=2^{-(n-1)}$ for all $\mathbf{t}_i \in \{\pm 1\}^{n-1}$. Using \eqref{eq:define_de}, we have,  
\begin{align}\label{eq:simplify_de}
    &\mathrm{DE}=\frac{2}{n}\sum_{i=1}^{n} \bE_{\bar \bT, \bar \bX} \bE( \bar T_i \bY_i)=: \frac{2}{n} \bE_{\bar \bT, \bar\bX} \Big[\sum_{i=1}^{n}\bar T_i \langle \bY_i \rangle \Big], \nonumber \\
    &\mathrm{IE}:= \frac{1}{n} \bE_{\bar \bT, \bar\bX} \Big[\sum_{i=1}^{n} \langle \bY_i \rangle \Big]- \frac{1}{n} \bE_{-\mathbf{1}, \bar\bX} \Big[\sum_{i=1}^{n} \langle \bY_i \rangle \Big]- \frac{1}{2}\mathrm{DE},
\end{align}
where $\langle \bY_i \rangle := \langle \mathbf{Y}_i \rangle_{\mathbf{t},\bx}= \bE(\mathbf{Y}_i|\bt,\bx)$ and the expectation is taken with respect to the density \eqref{eq:define_gibbs}. Note that in \eqref{eq:simplify_de} above, $(\bar \bT,\bar \bX)$ are independent, $\bar \bT \sim \mathrm{Unif}(\{\pm 1 \}^n)$ and $\bar \bX= (\bar \bX_1, \cdots, \bar \bX_n)$, $\bar \bX_i \sim \mathbb{P}_{X}$ are i.i.d. 

\end{lemma}

\begin{remark}{(general product measures)}\label{rmk:general_p}
    Although we have stated our results for uniform hypothetical allocation probability $\pi$, i.e., $\pi(\mathbf{t}_{-i})=2^{-(n-1)}$ for all $\mathbf{t}_i \in \{\pm 1\}^{n-1}$, they can be easily extended for general product measures on $\{\pm 1\}^{n-1}$ with $\bP(t_i=1)=p$. For example, one can rewrite  the direct effect $\text{DE}$ in \eqref{eq:simplify_de} as $\mathrm{DE}= \frac{1}{n}\sum_{i=1}^{n}\bE_{\bar \bT, \bar\bX} [(a+b \bar T_i)\langle \bY_i \rangle_{\bar \bT, \bar\bX}]$, where $a,b \in \rr$ satisfy $a+b= \frac{1}{p}$, $a-b= \frac{1}{1-p}$. This provides a natural Horvitz-Thompson type representation for the direct effect. The extension for $\mathrm{IE}$ is analogous. 
\end{remark}

At this point, we see the first challenge in evaluating the average direct effect---the conditional expectations $\langle \mathbf{Y}_i \rangle$ with respect to the density \eqref{eq:define_gibbs} are difficult to evaluate in general. In their seminal work, \citet{tchetgen2021auto} suggest evaluating these marginal expectations using Gibbs sampling. Unfortunately, mixing properties of Markov chains for MRFs of the form \eqref{eq:define_gibbs} are unknown in general. In fact, if $\tau_0=0$ and $\theta_0=0$, one expects slow mixing if the model \eqref{eq:define_gibbs} is in a low-temperature regime \cite{friedli2017statistical}. In addition, the parameters $\tau_0$, $\theta_0$ in \eqref{eq:define_gibbs} are unknown, and must be estimated from data. In earlier work, \citet{tchetgen2021auto} estimate the unknown parameters using the maximum pseudo-likelihood estimator or the coding estimator \cite{besag1975statistical}. To establish $\sqrt{n}$-consistency of these estimates, \cite{tchetgen2021auto} assumes the presence of an independent set of diverging size. This assumption can be easily violated e.g. if the underlying graph is a sparse Erd\H{o}s-R\'{e}nyi random graph on $n$ vertices with probability growing faster than $(\log n)/n$. Similar restrictions on the underlying graph are also common in alternative  DAG-based frameworks\cite{emmenegger2022treatment,ogburn2024causal}. Our paper takes a step in overcoming these limitations.

\noindent
\textbf{Contributions:} We now turn to our main contributions.  
\begin{itemize}
    \item[(i)] We consider mean-field matrices $\mathbf{A}_n$ such that $\sup_n \|\mathbf{A}_n\| < \infty$ and $\mathrm{Tr}(\mathbf{A}_n^2) = o(n)$ in the outcome regression model \eqref{eq:define_gibbs}. We show in Lemma \ref{lemma:mean_field_examples} that many networks naturally satisfy this mean-field condition. We introduce an iterative algorithm (Algorithm \ref{algo:de_method}) which yields provably consistent estimators of the $\mathrm{DE}$ and $\mathrm{IE}$ under an appropriate ``high-temperature" condition. Conceptually, our algorithm directly estimates the one-dimensional mean vectors $\langle \mathbf{Y}_i \rangle_{\mathbf{t},\mathbf{x}}$ arising in \eqref{eq:simplify_de}, and thus avoids Monte Carlo methods. Our algorithm is numerically efficient and highly scalable. In addition, our algorithm also facilitates direct uncertainty quantification for the treatment effects. Our approach has roots in Variational Inference (VI) methods for MRFs, which have attracted significant attention recently for their computational efficiency~\cite{blei2017variational,wainwright2008graphical}. 

    \item[(ii)] As a second example, we consider $\mathbf{A}_n$ such that $\{\mathbf{A}_n(i,j): i < j\}$ are i.i.d. $\mathcal{N}(0, \beta^2/n)$ for some $\beta>0$, $\mathbf{A}_n(i,i)=0$ for $1\leq i \leq n$ and $\mathbf{A}_n(i,j) = \mathbf{A}_n(j,i)$ if $j<i$. To interpret the corresponding outcome regression model \eqref{eq:define_gibbs}, assume that the study units interact over a complete graph on $n$ vertices. The edge weights $\mathbf{A}_n(i,j)$ represent the strength of these interactions. If $\mathbf{A}_n(i,j)>0$, the vertices have a positive affinity and vice versa if $\mathbf{A}_n(i,j)<0$. This model is well-studied in statistical physics and probability as the Sherrington-Kirkpatrick model~\cite{talagrand2010mean}. Classical results from random matrix theory imply that $\mathrm{Tr}(\mathbf{A}_n^2)=\Theta(n)$~\cite{bai2010spectral}; consequently, this model is outside the class of ``mean-field" matrices introduced above.  We introduce an iterative algorithm based on Approximate Message Passing (AMP) to estimate the treatment effects $\mathrm{DE}$ and $\mathrm{IE}$ (Algorithm \ref{alg:amp}). At sufficiently high temperature (i.e. $\beta>0$ sufficiently small), the algorithm is consistent for the treatment effect. Further, by repeating this algorithm, we can construct a confidence interval for the treatment effects. Formally, we establish that the AMP algorithm  approximates the one-dimensional mean vector $\langle \mathbf{Y}_i \rangle_{\mathbf{t},\mathbf{x}}$ at high-temperature in this model. 

    \item[(iii)] The algorithms introduced above assume oracle knowledge of the model parameters $\tau_0$ and $\btheta_0$ in \eqref{eq:define_gibbs}. In practice, these parameters have to be estimated from the observed data. We use the maximum pseudo-likelihood estimator, as proposed in \cite{tchetgen2021auto}. Note that the interaction matrices introduced above do not necessarily satisfy the assumptions in \cite{tchetgen2021auto}. Consequently, the consistency guarantees derived in prior work are not applicable to our setting. We leverage recent advances in the analysis of pseudo-likelihood estimators~\cite{daskalakis2019regression,ghosal2020joint} to establish $\sqrt{n}$-consistency in our setup. 

    \item[(iv)] Finally, we establish that once the estimated model parameters $\hat{\tau}_0$, $\hat{\btheta}_0$ are plugged into either Algorithm \ref{algo:de_method} or \ref{alg:amp} introduced above, the estimators obtained are consistent for the target causal estimands.  
\end{itemize}

Our results provide a principled approach to causal effect estimation in MRF based models. 

\subsection{Prior Work}

Classical causal inference methods rely on the SUTVA or no-interference  assumption---formally, this assumes that individuals are independently affected by treatments or exposures. However, there has been increasing interest lately in applications where treatments ‘spill
over’ from the treated individual to others~\cite{aronow2017estimating,athey2018exact,eckles2016design,hong2006evaluating,rosenbaum2007interference,sobel2006randomized,vanderweele2010direct}. The key challenge in causal effect estimation under interference stems from the intrinsic high-dimensionality of the problem---
if the study units can interfere with each another arbitrarily, one has to track a potential outcome for each possible treatment profile,  resulting in an exponential number of parameters. To overcome this barrier, one usually imposes some additional structure on the interference among the study units. 


Early works on interference posited that spillover happens through specific structural models~\cite{bramoulle2009identification,graham2008identifying,lee2007identification,manski1993identification} and are hence often criticized for their restrictive nature~\cite{angrist2014perils,goldsmith2013social}. Partial interference presents the next step in relaxing these rigid structural assumptions on the inference structure. Under partial interference, 
%
one assumes that interference is contained within known and disjoint groups of units~\cite{basse2019randomization,ferracci2014evidence,hayes2017cluster,hong2006evaluating,hudgens2008toward,kang2016peer,liu2014large,lundin2014estimation,park2023assumption,tchetgen2012causal}. This assumption has been largely relaxed in the past decade to setups  where interference is not restricted to specific groups. Here,  one typically assumes that the study units interact via general networks, and interference is specified through an ``exposure mapping"~\cite{aronow2017estimating,forastiere2021identification,jagadeesan2020designs,li2022random,manski2013identification,toulis2013estimation,ugander2013graph}. The related issue of adjusting for interference in experimental design has been a recent subject of in-depth study ~\cite{chin2019regression,eichhorn2024low,leung2022rate,viviano2023causal}. Finally, there have been several works where the interference structure is assumed to be unknown~\cite{cortez2023exploiting,savje2021average,yu2022estimating}. However, one typically assumes that interference is imposed through a ``sparse" network 
%
%
(e.g., bounded maximum degree). 
In contrast, our work develops computationally feasible estimators for causal estimands under relatively dense interference structures.

This paper is closest to the literature on treatment effect estimation from observational network data (see e.g. ~\cite{tchetgen2021auto} and references therein). In \cite{tchetgen2021auto}, the authors develop a network version of Robins g-formula~\cite{robins1986new}, assuming that the data can be represented by a specific family of graphical models known as a chain graph~\cite{lauritzen2002chain}. Follow up work investigates Causal chain graphs~\cite{bhattacharya2020causal,sherman2018identification,shpitser2017modeling}.  Our work contributes scalable algorithms (in contrast to MCMC methods proposed by \cite{tchetgen2021auto}) with theoretical guarantees to this literature. A related line of work studies interference via structural equation models~\cite{emmenegger2022treatment,ogburn2024causal,van2014causal,sofrygin2017semi}, which involves estimating causal parameters via targeted maximum likelihood (TMLE)~\cite{van2006targeted,van2014causal} or Augmented Inverse Probability Weighting (AIPW)~\cite{robins1992recovery,robins1994estimation}. However, TMLE requires density estimation which is challenging in high-dimension and existing works \cite{emmenegger2022treatment,ogburn2024causal,tchetgen2021auto} heavily constrain the degree of interference among the study units.

Our work considers interference arising from two different classes of networks. First, we consider ``mean-field" graphs~\cite{basak2017universality} which includes the complete graph, dense regular graphs and graphon models (See Lemma \ref{lemma:mean_field_examples}). 
We borrow tools from high-dimensional statistics~\cite{mukherjee2022variational}, statistical physics~\cite{bhattacharya2021sharp,mukherjee2021high,chen2021convergence,talagrand2010mean}, and nonlinear large deviations~\cite{chatterjee2016nonlinear, eldan2018gaussian, bhattacharya2022ldp,bhattacharya2023gibbs,yan2020nonlinear} to develop our estimator for direct and indirect effects. Moving beyond the mean-field setup, we consider gaussian interaction matrices, and introduce an AMP based algorithm for causal effect estimation. AMP based approaches have been used recently in causal inference under no interference~\cite{jiang2022new} or for multi-period experiments~\cite{shirani2023causal}. However, our work is the first to formulate message passing algorithms for undirected graphical models in the context of causal inference. Our algorithms are rooted in the study of variational inference \cite{blei2017variational,wainwright2008graphical,mezard2009information} for high-dimensional probability distributions. We demonstrate that these ideas are valuable in the context of causal effect estimation under interference.

\noindent
\textbf{Organization:} The rest of the paper is structured as follows: we present our results in Section~\ref{sec:results}. In Section~\ref{sec:numerics}, we explore the finite sample performance of our algorithms on simulated data. Finally we discuss our main assumptions, and collect some directions for future inquiry in Section~\ref{sec:discussion}. We prove our results in Section~\ref{sec:proofs}, and defer some technical proofs to the Appendix.

\noindent
\textbf{Acknowledgments:} SB and SS thank Justin Ko for references on the soft-spin SK model. 
SS thankfully acknowledges support from NSF (DMS CAREER 2239234), ONR (N00014-23-1-2489) and AFOSR (FA9950-23-1-0429).

\noindent
\textbf{Notation:} Given any $n \times n$ symmetric matrix $ \mathbf B_n$, we denote its operator norm by $\|\mathbf B_n\|$ and its trace by $\mathrm{Tr}(\mathbf B_n)$. Define its largest and smallest eigenvalues by $\lambda_{\max}(\mathbf B_n)$ and $\lambda_{\min}(\mathbf B_n)$ respectively. Denote by $\mathbf{I}_n$ the $n\times n$ identity matrix. For two matrices $\bA$ and $\mathbf B$, define $A \succeq B$ if $A-B$ is non-negative definite, and $A \succ B$ if $A-B$ is positive definite. A function $f: \rr^d \mapsto \rr$ is called $C^m$ if it is $m$-times continuously differentiable. For any $C^2$ function $f$, we denote by $\nabla f$ and $\nabla^2 f$ its gradient and Hessian respectively. For any random variable $Z \sim \mu$, we denote interchangeably by $\ee_Z$ or $\ee_\mu$ the expectation w.r.t $Z$. For any two sequences of random variables $Z_n, W_n$ and $W_n>0$, we define by $Z_n= O_P(W_n)$ if $Z_n/W_n$ is tight. Given $x \in \rr$, $\delta_x$ denotes Dirac delta measure at $x$. For two sequence of real numbers $a_n$ and $b_n$, $a_n=O(b_n)$ will denote that $\limsup_{n \rightarrow \infty} \ a_n / b_n = C$ for some $C \in [0,\infty)$ and $a_n=o(b_n)$ will denote that $\lim_{n \rightarrow \infty} a_n/b_n=0$. We call $a_n \gg b_n$ if $a_n/b_n \rightarrow \infty$ and $a_n \lesssim b_n$ if there exists $C>0$ such that $a_n \le C b_n$. For $\mathbf{a}, \mathbf{b} \in \rr^m$, we denote its inner product by $\langle \mathbf{a}, \mathbf{b} \rangle$. The $\ell^2$ and $\ell^\infty$ norm of $\mathbf{a}$ is denoted by $\|\mathbf{a}\|$ and $\|\mathbf{a}\|_\infty$ respectively.  Denote by $\mathbf{1}$ an $n$-length vector of all $1$s. For $n \in \mathbb{N}$, define $[n]= \{1,\ldots, n\}$. Throughout, we use $C,c>0$ to denote absolute constants which can change across displays.

\section{Results}
\label{sec:results}

We present our results in this section. In Section~\ref{sec:computation_estimands}, we present our algorithmic results on computing the causal estimands of interest. In Section~\ref{sec:parameter_estimation} we examine parameter estimation in these models, and the performance of the plug-in estimator. We can strengthen our results under an additional Log-Sobolev assumption on the base measure $\mu$ in \eqref{eq:define_gibbs}---we collect these results in Section~\ref{sec:lsi_statement}. Finally, we describe our main technical contributions in Section~\ref{sec:technical}.

\subsection{Efficient computation of the causal estimands} 
\label{sec:computation_estimands}

We first introduce the notion of quadratic exponential tilts, which will be valuable for our future discussion. 

\begin{defn}\label{def:exp_tilt}
    For any probability measure $\mu$ on $[-1,1]$ and $\lambda:= (\lambda_1, \lambda_2) \in \mathbb{R} \times (0,\infty)$, define the exponential tilt as 
    $$\frac{d\mu_\lambda}{d \mu}(x) := \exp\Big(\lambda_1 x + \frac{\lambda_2}{2}x^2 - \alpha (\lambda) \Big), \qquad \text{where } \alpha (\lambda):= \log \int e^{\lambda_1 x + \frac{\lambda_2}{2} x^2} d \mu(x).$$
For any $\lambda_2>0$, the function $\alpha(\cdot, \lambda_2)$ is infinitely differentiable, with $$\alpha'(\lambda_1, \lambda_2)=\mathbb{E}_{\mu_\lambda}(X),\quad \alpha''(\lambda_1,\lambda_2)=\mathrm{Var}_{\mu_\lambda}(X)>0.$$
Throughout the article, we will use the notation $\alpha'$ and $\alpha''$ to denote first and second derivatives of $\alpha$ w.r.t the first argument. Without loss of generality, we will assume throughout that $\mathrm{Supp}(\mu)= [-1,1]$. Fix $\lambda_2>0$. Consequently, for $m\in (-1,1)$, there exists $\lambda_1=\lambda_1(m)\in \mathbb{R}$ such that $\bE_{\mu_\lambda}(y)=m$. Define $I(m)=D(\mu_\lambda|\mu)$, where $D(\cdot|\cdot)$ denotes the Kullback-Leibler divergence. Finally, define $I(1) = D(\delta_1 | \mu)$ and $I(-1) = D(\delta_{-1}|\mu)$.
\end{defn}

\begin{ass}[Operator norm]
\label{ass:operator_norm}
    We assume throughout that the matrices $\bA_n$ and $\mathbf{M}_n$ satisfy $\sup_n \| \bA_n \| < \infty$, $\sup_n \| \mathbf{M}_n \| < \infty$. 
\end{ass}

\begin{algorithm}[!ht]
  \vspace{0.1in}
  \begin{flushleft}
   i. Generate $\bar \bT= (\bar T_1,\ldots, \bar T_n)$ from the uniform probability distribution on $\{\pm 1\}^n$. Generate $\bar \bX  = (\bar \bX_1,\cdots, \bar \bX_n) $ i.i.d. $\mathbb{P}_X$ independent of $\bar \bT$.\\

   ii. Fix $M \geq 1$. Initialize $\mathbf{u}^{(0)} =\mathbf{0}$ and run the iteration 
   $\mathbf{u}^{(\ell+1)} = \alpha'\Big( \mathbf{A}_n \mathbf{u}^{(\ell)} + \tau_0 \bar \bT + \bar \bX \btheta_0, 0 \Big)$. \\
  
  iii. Initialize $\tilde{\mathbf{u}}^{(0)}=\mathbf{0}$ and iterate $\tilde{\mathbf{u}}^{(\ell+1)} = \alpha'\Big( \mathbf{A}_n \tilde{\mathbf{u}}^{(\ell)} - \tau_0 \mathbf{1} + \bar \bX \btheta_0 , 0\Big)$. \\

  iv. Obtain the estimate of direct effect $\hat{\mathrm{DE}}_M= \frac{2}{n} \sum_{i=1}^{n} \bar{T}_i u_i^{(M)}$.\\
  v. Obtain the estimate of indirect effect $\widehat{\text{IE}}_M= \frac{1}{n}(\sum_{i=1}^{n}u_i^{(M)}- \sum_{i=1}^{n} \tilde u^{(M)}_i)-  \frac{1}{2}\hat{\mathrm{DE}}_M$.
  \end{flushleft}
\caption{}
\label{algo:de_method}
\end{algorithm} 

Our first result establishes statistical guarantees for Algorithm~\ref{algo:de_method} on a class of ``mean-field matrices" $\bA_n$.  
\begin{thm}\label{thm:de_mean_field}
Assume $\bA_n$ satisfies Assumption~\ref{ass:operator_norm}. 
    Assume, in addition, that $\mathrm{Tr}(\mathbf{A}_n^2)=o(n)$. Run Algorithm~\ref{algo:de_method} to compute estimates $\hat{\mathrm{DE}}_M$ and $\hat{\mathrm{IE}}_M$. Then there exists a universal constant $C_0>0$ such that if $\limsup_{n \to \infty}\|\mathbf{A}_n \| \leq C_0$ then  
    \begin{align}
        \lim_{M \to \infty} \lim_{n \to \infty}  \Big| \ee (\hat{\mathrm{DE}}_M) - \mathrm{DE} \Big| =0, \,\,\, \nonumber\\
        \lim_{M \to \infty} \lim_{n \to \infty}  \Big| \ee (\hat{\mathrm{IE}}_M) - \mathrm{IE} \Big|=0.\,\,\,  \nonumber 
    \end{align}
    In the above display, $\mathbb{E}[\cdot]$ denotes expectation with respect to $(\bar \bT, \bar \bX)$. 
\end{thm}

\begin{remark}
    We state Theorem~\ref{thm:de_mean_field} for deterministic interaction matrices $\bA_n$. For random interaction matrices $\bA_n$, our results go through unchanged as long as the matrices $\mathbf{A}_n$ satisfy the conditions of Theorem~\ref{thm:de_mean_field} almost surely.    
\end{remark}

\begin{remark}
\label{remark:mf_slln}
    It is particularly instructive to specialize Algorithm~\ref{algo:de_method} to the case of no interference i.e., $\bA_n=0$. In this case, the iterations for $\mathbf{u}^{(\ell)}$ and $\tilde{\mathbf{u}}^{(\ell)}$ converge in one step. At convergence, $\mathbf{u}_i^{(1)} = \mathbb{E}[\mathbf{Y}_i(t_i)|t_i = \bar{T}_i, \bar \bX_i]$ and $\tilde{\mathbf{u}}_i^{(1)} = \mathbb{E}[\mathbf{Y}_i(t_i)|t_i = - 1, \bar \bX_i]$. In this case, 
    \begin{align*}
        \widehat{\mathrm{DE}}_1 = \frac{2}{n} \sum_{i=1}^{n} \bar{T}_i \mathbb{E}[\mathbf{Y}_i(t_i)|t_i = \bar{T}_i, \bar \bX_i] \to \mathrm{DE} \,\,\, \mathrm{a.s.},
    \end{align*}
    where the last convergence follows by the Strong Law of Large Numbers. In the same vein, we have, 
    \begin{align*}
        \widehat{\mathrm{IE}}_1 = \frac{1}{n}\Big(\sum_{i=1}^{n}u_i^{(M)}- \sum_{i=1}^{n} \tilde u^{(M)}_i \Big) -\frac{1}{2}\hat{\mathrm{DE}}_1 \to 0 \,\,\, \mathrm{a.s.}
    \end{align*}
    We note that the conclusion in this special case is in fact stronger than  Theorem~\ref{thm:de_mean_field}, since we do not need to evaluate the expectation with respect to $(\bar \bT, \bar \bX)$. 
\end{remark}

Before proceeding further, we discuss the class of mean-field matrices arising in Theorem~\ref{thm:de_mean_field}. 
Assumption~\ref{ass:operator_norm} is a convenient normalization, which ensures that $\log Z_n(\bt , \bx) = O(n)$ for all $ \bt \in \{ \pm 1\}^n$, $\bx \in [-1,1]^n$. Under this normalization, $\bA_n \in \mathbb{R}^n$ is a symmetric matrix with $n$ real eigenvalues which are $O(1)$. In general, this implies $\mathrm{Tr}(\bA_n^2)= O(n)$. In this light, we see that the trace assumption in Theorem~\ref{thm:de_mean_field} imposes  that $\bA_n$ does not have too many non-trivial eigenvalues. This may be interpreted as an approximate ``low-rank" condition on the matrix $\bA_n$. We present some canonical examples of matrices $\bA_n$ which satisfy the assumptions of Theorem~\ref{thm:de_mean_field}.

\begin{lemma}
    \label{lemma:mean_field_examples}
    The following sequences $\{\bA_n : n \geq 1\}$ satisfy $\|\bA_n \| = O(1)$ and $\mathrm{Tr}(\bA_n^2) = o(n)$. 
    \begin{itemize}
        \item[(i)] Complete graph: For $\beta>0$, let $\bA_n = \frac{\beta}{n} (\mathbf{1}\mathbf{1}^{\top} - \mathbf{I}_n)$ be the scaled adjacency matrix of a complete graph on $n$-vertices.  
        \item[(ii)] Regular graphs: Let $\mathscr{G}_n$ be a sequence of $d_n$-regular graphs on $n$ vertices with $d_n \to \infty$. Let $\bA_n(i,j) = \frac{\beta}{d_n} \mathbf{1}(i \sim j)$, where $i \sim j$ if $i$ and $j$ are connected in $\mathscr{G}_n$. 
        \item[(iii)] Erd\H{o}s-R\'{e}nyi graphs: Let $\mathscr{G}(n,p_n)$ be the Erd\H{o}s-R\'{e}nyi random graph on $n$ vertices with edge probability $p_n \in [0,1]$. We assume that $n p_n \gg \sqrt{\frac{\log n}{\log \log n}}$. For $\beta>0$, set $\bA_n(i,j) = \frac{\beta}{n p_n} \mathbf{1}(i \sim j)$. 
        \item[(iv)] Graphon models:  
        Let $W: [0,1]^2 \mapsto [0,1]$ be a symmetric real-valued function. Given $W$, consider a graph $\mathscr{G}_n$ on $n$ vertices with edges $E_{ij}$ sampled as follows: Let $U_i \stackrel{\text{i.i.d.}}{\sim} U(0,1)$ and $E_{ij} \sim \text{Bern}(\rho_n W(U_i,U_j))$ independently for $i<j$. For $\beta>0$, set $\bA_n(i,j)= \frac{\beta}{n \rho_n} E_{ij}$ with $ n \rho_n \gg \sqrt{\frac{\log n}{\log \log n}}$.

    \end{itemize}
\end{lemma}

We note that the assumptions of Theorem~\ref{thm:de_mean_field} are satisfied for many matrices $\bA_n$ obtained from canonical random graph models. We emphasize that these matrices are distinct from those analyzed in prior work e.g. the complete graph is connected, and does not have the growing independent set assumed in \cite{tchetgen2021auto}. Similarly, for $p_n =O(1)$, the graph degrees are $O(n)$, which violates the assumptions in \cite{emmenegger2022treatment}. In these examples, the operator norm of $\bA_n$ is governed by the scalar $\beta>0$. The bound $\|\bA_n \| \leq C_0$ in Theorem~\ref{thm:de_mean_field} translates to a bound on $\beta$. In statistical physics, the parameter $\beta>0$ represents the inverse temperature. Thus a condition of the form $\beta < C_0'$ corresponds to a ``high-temperature" phase for physical models. We adopt this terminology in our subsequent discussions. 

We note that one run of Algorithm~\ref{algo:de_method} yields the estimate $\widehat{\text{DE}}_M$ and $\widehat{\text{IE}}_M$. To obtain consistent estimates for the causal estimands $\text{DE}$ and $\text{IE}$, we seek their expectations with respect to $(\bar \bT, \bar \bX)$. In practice, we use Monte Carlo sampling to approximate the expectation. Formally, 
\begin{itemize}
    \item[(i)] For $k \geq 1$, let $(\bar \bT^{(1)}, \bar \bX^{(1)}), \cdots,( \bar \bT^{(k)} , \bar \bX^{(k)})$ be i.i.d., and let the corresponding estimates be denoted as $(\widehat{\text{DE}}_M^{(1)}, \widehat{\text{IE}}_M^{(1)}), \cdots (\widehat{\text{DE}}_M^{(k)}, \widehat{\text{IE}}_M^{(k)})$. Form the final estimates 
    \begin{align*}
        {\widehat{\text{DE}}_M}^{\text{avg}} = \frac{1}{k} \sum_{j=1}^{k} \widehat{\text{DE}}_M^{(j)},\,\,\,\,\,\,\, {\widehat{\text{IE}}_M}^{\text{avg}} = \frac{1}{k} \sum_{j=1}^{k} \widehat{\text{IE}}_M^{(j)}. 
    \end{align*}

    \item[(ii)] To facilitate uncertainty quantification for the treatment effect estimates, we proceed as follows: let $\zeta \in [0,1]$ denote the desired confidence level. Let ${\widehat{\text{DE}}_M}^{\zeta/2}$ and ${\widehat{\text{DE}}_M}^{1-\zeta/2}$ denote the $\zeta/2$ and the $(1-\zeta/2)$-quantile of the set $\{\widehat{\text{DE}}_M^{(1)}, \cdots, \widehat{\text{DE}}_M^{(k)}\}$. Similarly, let ${\widehat{\text{IE}}_M}^{\zeta/2}$ and ${\widehat{\text{IE}}_M}^{1-\zeta/2}$ denote the $\zeta/2$ and the $(1-\zeta/2)$-quantile of the set $\{\widehat{\text{IE}}_M^{(1)}, \cdots, \widehat{\text{IE}}_M^{(k)}\}$. We will use the confidence sets 
\begin{equation}\label{eq:conf_set_define}
    \text{CI}_{\text{DE},\zeta}= \Big[{\widehat{\text{DE}}_M}^{\zeta/2},{\widehat{\text{DE}}_M}^{1-\zeta/2}\Big], \quad  \text{CI}_{\text{IE},\zeta}= \Big[{\widehat{\text{IE}}_M}^{\zeta/2},{\widehat{\text{IE}}_M}^{1-\zeta/2}\Big]
\end{equation}
for the causal effects $\text{DE}$ and $\text{IE}$ respectively.    
\end{itemize}

Note that the computation of the  replicated estimates $(\widehat{\text{DE}}_M^{(1)}, \widehat{\text{IE}}_M^{(1)}), \cdots (\widehat{\text{DE}}_M^{(k)}, \widehat{\text{IE}}_M^{(k)})$ can be completely parallelized, and does not affect the run-time of the overall procedure. Our next result establishes the fidelity of the point estimates  $ {\widehat{\text{DE}}_M}^{\text{avg}}$ and $ {\widehat{\text{DE}}_M}^{\text{avg}}$  introduced above. To provide theoretical guarantees for the uncertainty quantification procedure described above, we require an additional Log-Sobolev Inequality for the base measure $\mu$ in \eqref{eq:define_gibbs} (see Section \ref{sec:lsi_statement}). The precise statement is given by Lemma \ref{lem:MC_CI}

\begin{lemma}
\label{lemma:data_driven}
For any $\varepsilon>0$, 
    \begin{align*}
        &\lim_{k \to \infty} \lim_{M \to \infty} \lim_{n \to \infty} \mathbb{P}\Big[|{\widehat{\mathrm{DE}}_M}^{\mathrm{avg}} - \mathrm{DE}| > \varepsilon \Big] =0, \\
        &\lim_{k \to \infty} \lim_{M \to \infty} \lim_{n \to \infty} \mathbb{P}\Big[|{\widehat{\mathrm{IE}}_M}^{\mathrm{avg}} - \mathrm{IE}| > \varepsilon \Big]  =0. 
    \end{align*}

\end{lemma}

Lemma~\ref{lemma:data_driven} establishes that if the number of replicates $k$ is large enough, the estimates ${\widehat{\mathrm{DE}}_M}^{\mathrm{avg}}$ and ${\widehat{\mathrm{IE}}_M}^{\mathrm{avg}}$ are consistent for $\mathrm{DE}$ and $\mathrm{IE}$ respectively. We defer the proof of Theorem~\ref{thm:de_mean_field}, Lemma~\ref{lemma:mean_field_examples} and Lemma~\ref{lemma:data_driven} to Section~\ref{sec:proof_mf}.    


The previous result focuses on ``mean-field" interaction matrices $\mathbf{A}_n$ satisfying $\mathrm{Tr}(\mathbf{A}_n^2) = o(n)$. We next turn to gaussian interaction matrices $\bA_n$. Classical results from Random Matrix Theory imply that $\mathrm{Tr}(\bA_n^2)= O(n)$ \cite{bai2010spectral} and thus these matrices go beyond the class of interactions covered by Theorem~\ref{thm:de_mean_field}. To this end, we first introduce some objects which will be relevant for our subsequent discussion. 

\begin{lemma}
\label{lemma:uniqueness}
    Let $T \sim \mathrm{Unif}(\pm 1)$, $H \stackrel{d}{=}  \mathbf{X}_1^{\top} \theta_0 $, $\bX \sim \mathbb{P}_{X}$ and $G \sim \mathcal{N}(0,1)$ be independent. Then there exists $\beta_0>0$ such that for all $\beta< \beta_0$, the fixed point system  
    \begin{align}
        q_* &= \mathbb{E}[(\alpha'(\tau_0 \,T + H + \beta \sqrt{q_*}\,G, \beta^2 \sigma_{1,*}^2))^2], \nonumber \\
        \sigma_{1,*}^2 &= \mathbb{E}[\alpha''(\tau_0 T + H + \beta \sqrt{q_*}\, G, \beta^2 \sigma_{1,*}^2)]\label{eq:q_defn} 
    \end{align}
    has a unique solution. Similarly, for $\beta< \beta_0$, the fixed point system
    \begin{align}
        r_* = \mathbb{E}[(\alpha'(-\tau_0  + H + \beta \sqrt{r_*}\,G, \beta^2 \sigma_{2,*}^2))^2], 
        \nonumber \\
        \sigma_{2,*}^2 = \mathbb{E}[\alpha''(-\tau_0 + H + \beta \sqrt{r_*}\,G, \beta^2 \sigma_{2,*}^2)]        
        \label{eq:r_defn}
    \end{align}
    has a unique solution. 
\end{lemma} 

The constants $(q_*,\sigma_{1,*}^2)$, $(r_*, \sigma_{2,*}^2)$ are intrinsically related to the MRF $f$ \eqref{eq:define_gibbs}. Indeed, assume that $\bA_n$ in \eqref{eq:define_gibbs} is a symmetric gaussian matrix with $\mathcal{N}(0, \beta^2/n)$ i.i.d. entries above the diagonal, and let $t_i = \bar{T}_i \sim \mathrm{Unif}(\pm 1)$ i.i.d and $\bx_i = \bar \bX_i \sim \mathbb{P}_X$ i.i.d. be independent. Let $\by^1$ and $\by^2$ be two i.i.d. samples from the distribution \eqref{eq:define_gibbs}. We prove in Theorem~\ref{thm:overlap_conc} and \ref{thm:overlap_higher_moment} that at high temperature (i.e. $\beta>0$ sufficiently small), $\frac{1}{n} \| \by^1 \|^2 \approx \frac{1}{n} \| \by^2 \|^2 \approx (\sigma_{1,*}^2 + q_*)$, while $\frac{1}{n} (\by^1)^{\top} \by^2 \approx q_*$ with high probability under $f$. The constants $(r_*, \sigma_{2,*}^2)$ have similar interpretations when $t_i=-1$ and $\bx_i = \bar \bX_i \sim \mathbb{P}_X$ are i.i.d. The concentration of the norm and pair-wise inner product for the SK model at high-temperature was originally established in the seminal work of M. Talagrand \cite{talagrand2010mean}; here we extend these results to the outcome regression model $f$ in \eqref{eq:define_gibbs}. In spin glass theory, the distribution $f$ corresponds to a ``soft spin SK model with random external fields". Note that the density $f$ and the constant $q_*$ depend on the parameter $\beta>0$---we make this dependence explicit in the following definition. 

\begin{defn}[Local uniform overlap concentration]
For $\beta>0$, we say that the overlap concentrates uniformly at $\beta$ if there exists $\delta>0$ such that 
\begin{align}
\label{eq:unif_concentration}
    \lim_{n \to \infty} \sup_{\beta-\delta \leq \beta' \leq \beta } \mathbb{E}\Big[\mathbb{E}_{f,\beta'}\Big(\frac{1}{n} (\by^1)^{\top}\by^2 - q_*(\beta') \Big)^2 \Big] =0,  
\end{align}
  where $\mathbb{E}_{f,\beta'}[\cdot]$ denotes expectation with respect to the density $f$ with parameter $\beta'$ and $\mathbb{E}[\cdot]$ denotes expectation with respect to $(\bA_n, \bar \bT, \bar \bX)$.   
\end{defn}

\begin{remark}
    The concentration of the normalized inner product (or `overlap') is a canonical property of mean-field spin glass models at high-temperature. We refer to  Theorem~\ref{thm:overlap_conc} for a version of this property in our setting. We note that the uniform local convergence property \eqref{eq:unif_concentration} is stronger than the consequence of Theorem~\ref{thm:overlap_conc}. This stronger property is expected to hold for high-temperature spin glasses, and has been invoked in the prior literature \cite{chen2021convergence}. We will also assume this property for our theoretical results. 
\end{remark}

Before proceeding further, we note that if $\bA_n$ is random, the causal estimands $\text{DE}$ and $\text{IE}$ are themselves random, being functions of the matrix $\bA_n$. We will develop an Algorithm to compute these causal estimands. Our algorithms will work on ``typical" realizations of the interaction matrix $\bA_n$. To formalize this notion, we  introduce the following definition.

\begin{defn}
    Fix $n \geq 1$. Let $\{X_{n,M} : M \geq 1\}$ be a sequence of random variables measurable with respect to $\mathbf{A}_n$. We say that $X_{n,M} \stackrel{\mathscr{P}_{n,M}}{\longrightarrow} 0$ if there exists a deterministic sequence $\{ \varepsilon_{n,M} : M \geq 1\}$ satisfying  $\varepsilon_{n,M} \geq 0$, 
    \begin{align*}
        \lim_{M \to \infty} \lim_{n \to \infty} \varepsilon_{n,M} = 0 
    \end{align*}
    such that 
    \begin{align*}
        \lim_{M \to \infty} \lim_{n \to \infty} \mathbb{P}[|X_{n,M}| > \varepsilon_{n,M}] = 0. 
    \end{align*}
    In the display above, $\mathbb{P}[\cdot]$ refers to the randomness with respect to $\mathbf{A}_n$. 
\end{defn}



\begin{algorithm}[!ht]
\begin{flushleft}

    i. Generate $\bar \bT= (\bar T_1,\cdots, \bar T_n)$ from the uniform distribution on $\{\pm 1\}^n$. Generate $\bar \bX = (\bar \bX_1, \cdots, \bar \bX_n)$ i.i.d. from $\mathbb{P}_X$ independent of $\bar \bT$. \\

    ii. Set $\beta_0$ as in Lemma \ref{lemma:uniqueness}. For $\beta<\beta_0$, let $(q_*, \sigma_{1,*}^2)$ and $(r_*, \sigma_{2,*}^2)$ be the unique solutions to \eqref{eq:q_defn} and \eqref{eq:r_defn} respectively. 

    iii. \textbf{Initialization:} Set $\bu^{[0]}=0$, $\bu^{[1]}=0$, $u^{[2]}_i = \sum_{j=1}^{n} \mathbf{G}_n(i,j) \sqrt{q_*} $. 

    iv. \textbf{Iteration $k \geq 2$:} For $k \geq 2$, 
    \begin{align*}
        u^{[k+1]}_i &= \sum_{l=1}^{n} \mathbf{G}_n(i,l) \alpha'(\beta u_l^{[k]} + \tau_0 \bar{T}_l + \bar{\bX}_l^{\top} \boldsymbol{\theta}_0, \beta^2 \sigma_{1,*}^2) - d_k \alpha'(\beta u^{[k-1]}_i + \tau_0 \bar{T}_i + \bar{\bX}_i^{\top} \boldsymbol{\theta}_0 , \beta^2 \sigma_{1,*}^2) \nonumber, \\
        d_{k} &= \frac{\beta}{n} \sum_{j=1}^{n} \alpha''(\beta u_j^{[k]} + \tau_0 \bar{T}_j + \bar{\bX}_j^{\top} \boldsymbol{\theta}_0 , \beta^2 \sigma_{1,*}^2). 
    \end{align*}

    v. \textbf{Initialization:} Set $\tilde{\bu}^{[0]}=0$, $\tilde{\bu}^{[1]}=0$, $\tilde{\bu}^{[2]}_i =  \sum_{i=1}^{n} \mathbf{G}_n(i,j) \sqrt{r_*} $. 

    vi. \textbf{Iteration $k \geq 2$:} For $k \geq 2$, 
    \begin{align*}
        \tilde{u}^{[k+1]}_i &= \sum_{l=1}^{n} \mathbf{G}_n(i,l) \alpha'(\beta \tilde{u}_l^{[k]} - \tau_0  + \bar{\bX}_l^{\top} \boldsymbol{\theta}_0, \beta^2 \sigma_{2,*}^2) - \tilde{d}_k \alpha'(\beta \tilde{u}^{[k-1]}_i - \tau_0 + \bar{\bX}_i^{\top} \boldsymbol{\theta}_0 , \beta^2 \sigma_{2,*}^2) \nonumber, \\
        \tilde{d}_{k} &= \frac{\beta}{n} \sum_{j=1}^{n} \alpha''(\beta \tilde{u}_j^{[k]} - \tau_0  +\bar{\bX}_j^{\top} \boldsymbol{\theta}_0, \beta^2 \sigma_{2,*}^2). 
    \end{align*}

    vii. \textbf{Mean estimate:} Fix $M\geq 1$. Set 
    \begin{align*}
    \mathbf{m}^{[M]} &= \alpha'(\beta \mathbf{u}^{[M]} + \tau_0 \bar \bT + \bar \bX \boldsymbol{\theta}_0, \beta^2 \sigma_{1,*}^2), \\
    {\tilde{\mathbf{m}}}^{[M]} &= \alpha'(\beta \mathbf{u}^{[M]} - \tau_0 \mathbf{1} + \bar \bX \boldsymbol{\theta}_0, \beta^2 \sigma_{2,*}^2).
    \end{align*}


    vii. Obtain the estimate of the direct effect $\hat{\mathrm{DE}}_M= \frac{2}{n} \sum_{i=1}^{n} \bar{T}_i m_i^{[M]}$. 

    viii. Obtain the estimate of the indirect effect $\hat{\mathrm{IE}}_M = \frac{1}{n} ( \sum_{i=1}^{n} m_i^{[M]} - \sum_{i=1}^{n} \tilde{m}_i^{[M]})- \frac{1}{2}\hat{\text{DE}}_M$.

    \end{flushleft}
    \caption{}
    \label{alg:amp}
\end{algorithm}

\noindent 
Armed with this notion of convergence, we turn to our main result for gaussian interaction matrices $\bA_n$. 
\begin{thm}\label{thm:amp}
    Set $\mathbf{A}_n= \mathbf{A}_n^{\top}$, $\mathbf{A}_n(i,j) \sim \mathcal{N}(0, \beta^2/n)$ i.i.d. for $i<j$ for some $\beta >0$, $\mathbf{A}_n(i,i)=0$ for $i \in [n]$. Set $\mathbf{A}_n(i,j) = \beta \mathbf{G}_n(i,j)$. Recall $\beta_0$ from \eqref{lemma:uniqueness}. Run Algorithm~\ref{alg:amp} to compute estimates $\hat{\mathrm{DE}}_M$ and $\hat{\mathrm{IE}}_M$. There exists $0<\bar{\beta}_0 \leq \beta_0$ such that for $\beta < \bar{\beta}_0$, if the overlap concentrates uniformly at $\beta$, 
    %
    \begin{align*}
       \Big| \ee (\hat{\mathrm{DE}}_M) - \mathrm{DE} \Big| \stackrel{\mathscr{P}_{n,M}}{\longrightarrow} 0,\\
       \Big| \ee(\hat{\mathrm{IE}}_M) - \mathrm{IE} \Big|  \stackrel{\mathscr{P}_{n,M}}{\longrightarrow} 0. \\
    \end{align*}
    In the above display, $\mathbb{E}[\cdot]$ denotes expectation with respect to $(\bar \bT,\bar \bX)$.
\end{thm}

\begin{remark}
    Classical results from Random Matrix Theory imply that $\| \mathbf{G}_n \| \to 2$ a.s. \cite[Theorem 5.1]{bai2010spectral}.  Consequently, $\beta < \bar{\beta}_0$ implies a high-temperature condition $\| \bA_n \| \leq C_0$ for some $C_0>0$. This condition is comparable to the high-temperature condition in Theorem~\ref{thm:de_mean_field}.  
\end{remark}

\begin{remark}
\label{remark:amp_slln}
    It is again helpful to specialize to the case of no interference i.e. $\beta=0$. In this case, $q_*= \mathbb{E}[(\alpha'(\tau_0 T + H,0))^2]$, $\sigma_{1,*}^2 = \mathbb{E}[\alpha''(\tau_0 T + H,0)]$, where $H \stackrel{d}{=} \mathbf{X}_1^{\top} \boldsymbol{\theta}_0$. Similarly, $r_* = \mathbb{E}[(\alpha'(-\tau_0 + H,0))^2]$, $\sigma_{2,*}^2 = \mathbb{E}[\alpha''(-\tau_0 + H , 0)]$. Note that Algorithm~\ref{alg:amp} iterates converge for all $k \geq 3$ and $\mathbf{m}^{[k]} = \alpha'(\tau_0 \bar \bT + \bar \bX \boldsymbol{\theta}_0,0)$, $\tilde{\mathbf{m}}^{[k]} = \alpha'(- \tau_0 \mathbf{1} + \bar \bX \boldsymbol{\theta}_0,0)$. In particular, $m_i^{[k]} = \mathbb{E}[\mathbf{Y}_i(t_i)| t_i = \bar{T}_i, \bX_i]$ and $\tilde{m}_i^{[k]} = \mathbb{E}[\mathbf{Y}_i(t_i) | t_i = -1, \bX_i]$. We thus reduce back to the setting of Remark~\ref{remark:mf_slln}.   
\end{remark}
We note that similar to Theorem~\ref{thm:de_mean_field}, the estimates $\hat{\text{DE}}_M$ and $\hat{\text{IE}}_M$ are consistent for the corresponding estimands in expectation over $(\bar \bT, \bar \bX)$. In practice, we simulate several sets of treatment-covariate configurations $\{(\bar \bT^{(1)}, \bar \bX^{(1)}), \cdots, (\bar \bT^{(k)}, \bar \bX^{(k)})\}$, run parallel AMP algorithms on the distinct configurations, and finally average the estimates obtained from the individual runs. These estimates enjoy similar guarantees as derived in Lemma~\ref{lemma:data_driven}. We omit a formal statement to this end to avoid repetition. We prove Lemma~\ref{lemma:uniqueness} and  Theorem~\ref{thm:amp} in Section~\ref{sec:proof_amp}.


\subsection{Parameter Estimation}
\label{sec:parameter_estimation}
We have so far assumed oracle knowledge of the model parameters $\tau_0$ and $\boldsymbol{\theta}_0$. In practice, these parameters are unknown and must be estimated from data.  We now turn to the estimation of these unknown parameters. 
To this end, we utilize an estimator based on pseudo-likelihood \cite{besag1975statistical}. Formally, we define this as follows.

\begin{defn}[Maximum pseudo-likelihood estimator]
Given $(\bY,\bT,\bX)$, we define the maximum pseudo-likelihood estimator of the parameters $(\tau_0,\btheta_0)$ as
\begin{align}
    (\hat \tau_{\text{MPL}}, \hat {\boldsymbol{\theta}}_{\text{MPL}}) &=\argmax_{\tau,\boldsymbol{\theta}} \prod_{i=1}^{n} f(\bY_i| \bY_{-i},\bT,\bX). \label{eq:mple_outcome}  
\end{align}
as long as the maximizers in the above display are unique. 
The conditional probabilities in \eqref{eq:mple_outcome} are computed from the joint distribution \eqref{eq:define_gibbs}. 
\end{defn}

\begin{remark}
    Parameter estimation in the MRF \eqref{eq:define_gibbs} is challenging due to the intractable normalization constant $Z_n(\bt, \bx)$. The pseudo-likelihood is free from the  normalization constant, and thus computationally efficient. This has motivated the wide-spread use of  maximum pseudo-likelihood estimates in MRFs~\cite{besag1975statistical,daskalakis2019regression,mukherjee2021high}. 
\end{remark}

\begin{remark}[Coding estimators] 
    As an alternative to the maximum pseudo-likelihood estimator, \citet{tchetgen2021auto} also consider the ``coding estimator". Formally, if $\bA_n$ represents the scaled adjacency matrix of an underlying graph, the coding estimator for parameters $(\tau_0,\btheta_0)$ is defined as 
\[(\hat \tau_{\text{MPL}}, \hat {\boldsymbol{\theta}}_{\text{MPL}}) =\argmax_{\tau,\boldsymbol{\theta}} \prod_{i\in S} f(\mathbf{Y}_i| \mathbf{Y}_{S^c},\bt,\bx),\]
where $S$ denotes any maximal independent set of the underlying graph. In this work, we focus on relatively dense networks where the set $S$ is usually very small  e.g. if the underlying network is a complete graph, the set $S$ is a singleton. Therefore, we will only consider the maximum pseudo-likelihood estimator for the subsequent discussion.
\end{remark}


Our next result establishes the $\sqrt{n}$-consistency of the pseudo-likelihood estimator. 

\begin{thm}\label{thm:param_estimate}
    Assume $(\bA_n, \mathbf{M}_n)$ satisfy Assumption~\ref{ass:operator_norm}.  Assume that the parameters $(\tau_0, \btheta_0)$ satisfy Assumption \ref{assn:parameter_space} and belong to the interior of the parameter space. Then the pseudo-likelihood estimator $(\hat \tau_{\text{MPL}}, \hat {\boldsymbol{\theta}}_{\text{MPL}})$ exists and satisfies
    \begin{equation*}
        \sqrt{n} ((\hat \tau_{\text{MPL}}, \hat {\boldsymbol{\theta}}_{\text{MPL}})-(\tau_0,\boldsymbol{\theta}_0))= O_{P}(1).
    \end{equation*}
\end{thm}

Next, we turn to the implication of Theorem \ref{thm:param_estimate} to the estimation of the target causal estimands. Algorithms \ref{algo:de_method} and \ref{alg:amp} involve the true parameters $\tau_0,\btheta_0$ which are unknown in practice. We use the estimates $(\hat \tau_{\text{MPL}}, \hat {\boldsymbol{\theta}}_{\text{MPL}})$ to construct the natural plug-in estimators for the target estimands. The following result establishes the stability of our algorithms under this plug-in procedure. For mathematical clarity, we use $\hat{\mathrm{DE}}_M(\tau, \btheta)$ and $\hat{\mathrm{IE}}_M(\tau,\btheta)$ in the following result to explicitly track the dependence of these estimates on the underlying parameters.   

\begin{lemma}\label{lem:algo_stable}
    Assume that $(\bA_n, \mathbf{M}_n)$ satisfies Assumption~\ref{ass:operator_norm}. Let $(\widehat{\mathrm{DE}}(\tau_0, \btheta_0),\widehat{\mathrm{IE}}(\tau_0, \btheta_0))$ denote the causal effect estimates obtained from either Algorithm~\ref{algo:de_method} or \ref{alg:amp} using the true parameters $\tau_0, \btheta_0$. Similarly, let $(\widehat{\mathrm{DE}}(\hat{\tau}_{\text{MPL}}, \hat{\btheta}_{\text{MPL}}),\hat{\mathrm{IE}}(\hat{\tau}_{\text{MPL}}, \widehat{\btheta}_{\text{MPL}}))$ denote the corresponding plug-in estimators obtained from either Algorithm~\ref{algo:de_method} or \ref{alg:amp} using the estimators $\hat{\tau}_{\text{MPL}}, \hat{\btheta}_{\text{MPL}}$.    
    There exists $C_0>0$ such that, if $\limsup_n \|\bA_n\| \le C_0$, we have, for any $\varepsilon >0$,
    \begin{align*}
       & \lim\limits_{n \rightarrow \infty} \bP \left(\left|\hat{\mathrm{DE}}_M(\tau_0, \btheta_0) -\hat{\mathrm{DE}}_M(\hat \tau_{\text{MPL}}, \hat {\boldsymbol{\theta}}_{\text{MPL}})\right| > \varepsilon \right)=0, \\
        &\lim\limits_{n \rightarrow \infty} \bP \left(\left|\hat{\mathrm{IE}}_M(\tau_0, \btheta_0) -\hat{\mathrm{IE}}_M(\hat \tau_{\text{MPL}}, \hat {\boldsymbol{\theta}}_{\text{MPL}})\right| > \varepsilon \right)=0.
    \end{align*}
\end{lemma}

Our result justifies the use of the plug-in procedure in combination with Algorithms \ref{algo:de_method} and \ref{alg:amp} for the estimation of treatment effects. We prove Theorem~\ref{thm:param_estimate} and Lemma~\ref{lem:algo_stable} in Section~\ref{sec:proof_pseudolikelihood}.   


\subsection{Additional concentration under Log-Sobolev inequality}\label{sec:lsi_statement}
Recall that Algorithms ~\ref{algo:de_method} and Algorithm~\ref{alg:amp} yield one replicate of the estimates $\hat{\text{DE}}$ and $\hat{\text{IE}}$. Theorems \ref{thm:de_mean_field} and \ref{thm:amp} establish consistency for these estimates, once we evaluate the average over $(\bar \bT, \bar \bX)$.  As discussed in Lemma~\ref{lemma:data_driven}, in practice we sample multiple configurations of $(\bar \bT, \bar \bX)$, and average the separate estimators to obtain our final estimate. On the other hand, Remarks~\ref{remark:mf_slln} and \ref{remark:amp_slln} show that under the no-interference setting, this last Monte Carlo step is unnecessary, due to the Law of Large Numbers. One might wonder if this concentration-property persists, under the general interference model. Here we establish concentration of the individual estimates $\widehat{\text{DE}}$ and $\widehat{\text{IE}}$ around the respective target causal estimands if the base measure $\mu$ in \eqref{eq:define_gibbs} satisfies an additional property. In particular, we will require the exponential tilts of the base measure $\mu$ to satisfy a Log Sobolev Inequality (LSI)~\cite{bauerschmidt2019very}. This property will also allow us to derive rigorous guarantees for the CI introduced in \eqref{eq:conf_set_define}.


\begin{defn}[Log Sobolev Inequality]
    For any measure $\tilde \mu$ supported on $\rr^m$, define its entropy as $\mathrm{Ent}_{\tilde \mu}(f)= \ee f(X) \log f(X) - \ee f(X) \ee \log f(X)$, where $X \sim \tilde \mu$. A measure $\tilde \mu$ satisfies LSI for some constant $C>0$ if for any $C^1$ function $f: \mathbb R^m \mapsto \mathbb R$, we have
\begin{equation}\label{eq:define_lsi}    \mathrm{Ent}_{\tilde \mu}(f^2) \le C \int |\nabla f(x)|^2 d\tilde \mu(x).
\end{equation}
\end{defn}

\begin{remark}
    Let us mention two examples of $\tilde \mu$ which satisfy the LSI defined above. First, if $\tilde{\mu}= p \delta_{1}+ (1-p)\delta_{-1}$ is a two-point measure on $\{-1,1\}$, $p \in (0,1)$, then LSI holds for $\tilde \mu$~\cite[Remark 1 and Remark 2]{latala2000between}. This includes the uniform distribution on $\{\pm 1\}$, which corresponds to $p = 0.5$. Second, if $\tilde \mu$ is the uniform measure on a compact interval, it satisfies LSI~\cite[Theorem 2.2]{ghang2014sharp}.
\end{remark}

\begin{defn}\label{defn:tilted-LSI}
    A measure $\tilde \mu$ satisfies tilted-LSI property if for any $\lambda:= (\lambda_1, \lambda_2) \in \mathbb{R} \times (0,\infty)$, the measure $\tilde \mu_\lambda$ defined by Definition \ref{def:exp_tilt} satisfies LSI \eqref{eq:define_lsi}, with a uniformly bounded constant over $\lambda \in \mathbb{R} \times (0, \infty)$.
\end{defn}

As an example, note that if $\tilde \mu$ is supported on $\{-1,1\}$ then $\tilde \mu_\lambda$ is also  supported on $\{-1,1\}$ for any $\lambda:= (\lambda_1, \lambda_2) \in \mathbb{R} \times (0,\infty)$. Hence, by the discussion above, $\tilde \mu$ satisfies tilted-LSI property.
Now, we are in a position to state our main concentration result of the section.

\begin{prop}\label{prop:lsi_concentration}
    Suppose the base measure $\mu$ in \eqref{eq:define_gibbs} satisfies the tilted LSI  property. 
    \begin{itemize}
        \item[(i)] Assume that $\bA_n$ satisfies Assumption~\ref{ass:operator_norm} and $\mathrm{Tr}(\bA_n^2) = o(n)$. Run Algorithm~\ref{algo:de_method} to compute estimates $\hat{\mathrm{DE}}_M$ and $\hat{\mathrm{IE}}_M$. Then, there exists $C_0>0$ such that for if  $\limsup_n\|\bA_n\| < C_0$, we have 
%
     \begin{align}
        \lim_{M \to \infty} \lim_{n \to \infty}  \Big|\hat{\mathrm{DE}}_M - \mathrm{DE} \Big| =0 \,\,\, \mathrm{a.s.}, \nonumber\\
        \lim_{M \to \infty} \lim_{n \to \infty}  \Big| \hat{\mathrm{IE}}_M - \mathrm{IE} \Big|=0\,\,\,  \mathrm{a.s.} \nonumber 
    \end{align}
    The convergence above is almost sure with respect to the randomness of $(\bar \bT, \bar \bX)$.

    \item[(ii)] Assume $\mathbf{A}_n= \mathbf{A}_n^{\top}$, $\mathbf{A}_n(i,j) \sim \mathcal{N}(0, \beta^2/n)$ i.i.d. for $i<j$ for some $\beta >0$, $\mathbf{A}_n(i,i)=0$ for $i \in [n]$. Set $\mathbf{A}_n(i,j) = \beta \mathbf{G}_n(i,j)$. Recall $\beta_0$ from \eqref{lemma:uniqueness}. Run Algorithm~\ref{alg:amp} to compute estimates $\hat{\mathrm{DE}}_M$ and $\hat{\mathrm{IE}}_M$. There exists $0<\bar{\beta}_0 < \beta_0$ such that for all $\beta < \beta_0$ and any $\varepsilon>0$, 
    \begin{align}
        \lim_{M \to \infty} \lim_{n \to \infty}  \mathbb{P}\Big[\Big|\hat{\mathrm{DE}}_M - \mathrm{DE} \Big| > \varepsilon \Big]=0, \,\,\,  \nonumber\\
        \lim_{M \to \infty} \lim_{n \to \infty}  \mathbb{P}\Big[\Big| \hat{\mathrm{IE}}_M - \mathrm{IE} \Big|> \varepsilon \Big]= 0.\,\,\,   \nonumber 
    \end{align} 
    In the display above, $\mathbb{P}[\cdot]$ denotes the probability under the joint distribution of $(\bA_n, \bar \bT, \bar \bX)$. 
     \end{itemize}
     \end{prop}
 We prove Proposition~\ref{prop:lsi_concentration} in Section~\ref{sec:proof_conc_logconcavity}. 

 \begin{remark}
     We emphasize that the guarantees in Proposition~\ref{prop:lsi_concentration} are valid for any covariate distribution $\mathbb{P}_X$ supported on $[-1,1]^d$. The proofs of these assertions simplify significantly if one assumes additional properties on $\mathbb{P}_X$ (e.g. $\mathbb{P}_X$ satisfies a Log Sobolev Inequality). Our proof requires new ideas, and is one of our main technical contributions.  
 \end{remark}

Our next result provides theoretical guarantees for Monte-Carlo based confidence intervals $\text{CI}_{\text{DE},\zeta}$ and  $\text{CI}_{\text{IE},\zeta}$ defined by \eqref{eq:conf_set_define}.

\begin{lemma}\label{lem:MC_CI}
    Suppose the base measure $\mu$ in \eqref{eq:define_gibbs} satisfies the tilted LSI  property. For any $\varepsilon>0$, 
    \begin{align*}
         &\lim_{k \to \infty} \lim_{M \to \infty} \lim_{n \to \infty} \mathbb{P}\Big[ \mathrm{DE} \in [{\widehat{\mathrm{DE}}_M}^{\zeta/2} - \varepsilon,{\widehat{\mathrm{DE}}_M}^{1-\zeta/2} + \varepsilon] \Big] \geq 1 - \zeta, \\
        &\lim_{k \to \infty} \lim_{M \to \infty} \lim_{n \to \infty} \mathbb{P}\Big[ \mathrm{IE} \in [{\widehat{\mathrm{IE}}_M}^{\zeta/2} - \varepsilon,{\widehat{\mathrm{IE}}_M}^{1-\zeta/2} + \varepsilon] \Big] \geq 1 - \zeta. 
    \end{align*}
\end{lemma}

The result establishes that the proposed Confidence Intervals, after some fattening, have the desired coverage properties, provided the number of replicates $k$ is sufficiently large. We establish this result in Section~\ref{sec:proof_conc_logconcavity} and explore the non-asymptotic performance of these confidence intervals using numerical experiments in Section~\ref{sec:numerics}.


\subsection{Technical Contributions:}
\label{sec:technical}
In this section, we emphasize our technical contributions, concentrating on Theorems~\ref{thm:de_mean_field}, \ref{thm:amp} and \ref{thm:param_estimate}. In turn, this will  help us motivate Algorithms~\ref{algo:de_method} and \ref{alg:amp}. 

\begin{itemize}
    \item[(i)] We start with the proof of Theorem~\ref{thm:de_mean_field}. We focus on the direct effect $\text{DE}$; our discussions easily extend to the indirect effect $\text{IE}$. Recall the representation of the direct effect $\text{DE}$ from \eqref{eq:simplify_de}. The main computational challenge stems from the vector of one-dimensional means $\langle \bY \rangle$---a direct computation is intractable due to the inherent high-dimensionality in the problem as the graph size $n \to \infty$. The usual recourse is to use MCMC. Here, we adopt a different approach which is motivated from \emph{variational inference} \cite{wainwright2008graphical,blei2017variational}. We show first that at high-temperature, the vector $\langle \bY \rangle$ approximately satisfies the fixed point equation 
    \begin{align*}
        \langle \bY \rangle \approx \alpha'\Big( \bA_n \langle \bY \rangle + \tau_0 \bar \bT + \bar\bX \boldsymbol{\theta}_0,0 \Big).  
    \end{align*}
    Algorithm~\ref{algo:de_method} is a natural iterative scheme to solve for this fixed point. Finally, we establish that at high-temperature, the iteration is contractive, and converges to the right fixed point. To formalize this argument, we utilize recent progress in non-linear large deviations \cite{chatterjee2016nonlinear,yan2020nonlinear,eldan2018gaussian} and the study of Ising models \cite{basak2017universality,eldan2020taming,jain2019mean,jain2018mean} under this framework.

    \item[(ii)] The proof of Theorem~\ref{thm:amp} follows a similar template. In this case, ideas from spin-glass theory suggest that at high-temperature, the mean vector $\langle \bY \rangle$ satisfies the Thouless-Anderson-Palmer (TAP) equations 
    \begin{align*}
        \langle \bY \rangle \approx \alpha'\Big( \bA_n \langle \bY \rangle + \tau_0 \bar{\bT} + \bar{\bX}\boldsymbol{\theta}_0 - \beta \sigma_{1,*}^2 \langle \bY \rangle , \beta^2 \sigma_{1,*}^2 \Big)
    \end{align*}
    The AMP Algorithm in Algorithm~\ref{alg:amp} attempts to solve this fixed point system. Indeed, it is easy to see that if the iterations converge, the fixed point $\mathbf{m}^{[\infty]}$ will satisfy the TAP equations. We prove that these iterations indeed converge to the true mean vector $\langle \bY \rangle$ asymptotically. To this end, we extend the work of \cite{chen2021convergence}, which focuses exclusively on the Sherrington-Kirkpatrick (SK) model. In particular, the SK model is supported on the hypercube $\{\pm 1\}^n$, and does not involve the random external fields appearing in \eqref{eq:define_gibbs}. In the language of spin glass theory, we prove convergence of the AMP algorithm for the soft-spin SK model with i.i.d. external fields. Moving from Ising spin to soft spins introduces new challenges, as the norm of a sample (also called the ``self-overlap") is not constant, but must be tracked separately. In addition, due to the presence of the i.i.d. external fields, the coordinates of a sample are not exchangeable, but only exchangeable in expectation. We extend key results from \cite{talagrand2010mean} on the soft spin SK model to overcome these challenges. We consider this to be one of the main technical contributions of our work. 

    \item[(iii)] Finally, we turn to Theorem~\ref{thm:param_estimate}. Theorem \ref{thm:param_estimate} shows $\sqrt{n}$-consistency of pseudo-likelihood estimator as long as the operator norm of $\bA_n, \mathbf{M}_n$ are bounded. The main difficulty of the proof arises because the external field $\tau_0 T_i+ \bX^\top_i \btheta_0$ are not independent and further $\bT$ and $\bX$ are dependent. We overcome this difficulty by invoking Lemma \ref{Lem:grad_concentration}. Note that, our assumption is weaker than \cite[(1.2), (1.3)]{ghosal2020joint}, thus subsuming \cite[Proposition 1.6]{ghosal2020joint} which shows consistency under deterministic external field. Further, since our external field is not i.i.d., our results extends \cite[Theorem 3.1]{daskalakis2019regression}. Finally, we believe that for specific examples of $\bA_n,\mathbf{M}_n$, one can establish asymptotic normality of $(\hat \tau_{\text{MPL}}, \hat {\boldsymbol{\theta}}_{\text{MPL}})$. We leave such explorations for future work.
\end{itemize}

\section{Numerical Experiments}
\label{sec:numerics}
In this section, we study the finite sample behavior of our proposed method. As our first example, summarized in Table \ref{tab:curie_weiss}, we consider the interference induced by the complete graph (i.e., $\bA_n= \frac{\beta}{n}(\mathbf{1}\mathbf{1}^\top- I_n)$) on $n=200,400,800$ vertices respectively. The covariates are i.i.d. with $\bX_i \sim \text{Unif}(-1,1)$. The treatment $\bt$ is generated by \eqref{eq:prop_score} with $\mathbf{M}_n=\bA_n$ and the outcome $\by \in \{\pm 1\}^n$ given $\bt, \bx$ is simulated from \eqref{eq:define_gibbs}. Here, we choose 
\begin{align}\label{eq:sims_param}
    \beta=0.3, \quad \mu= \frac{1}{2} (\delta_1+\delta_{-1}), \quad (\tau_0,\theta_0,\gamma_0)=(0.5,2,0).
\end{align}
As stated in Lemma \ref{lemma:mean_field_examples}, $\bA_n$ satisfies $\mathrm{Tr}(\mathbf{A}_n^2)=O(1)= o(n)$. Here, we use Algorithm \ref{algo:de_method} to compute our causal estimands --- the direct and indirect effect. We use maximum pseudo-likelihood method~\ref{eq:mple_outcome} to compute $(\hat \tau,\hat \btheta)$. Then, we use Algorithm \ref{algo:de_method} with $M=500$ iterations to obtain our estimators $\widehat{\text{IE}}_M, \hat{\mathrm{DE}}_M$. We repeat this experiment $100$ times to compute a $95\%$ confidence interval for the direct and indirect effect. To understand the efficacy of the said interval estimator, we use the true $(\tau_0,\theta_0)$ to compute the oracle causal effect using $M=500$ iterations of Algorithm \ref{algo:de_method}. This is reported as ``truth" in Table \ref{tab:curie_weiss}. In all $3$ cases noted in Table \ref{tab:curie_weiss}, the oracle value is contained within the confidence interval. Finally, we report the run-time of Algorithm \ref{algo:de_method} which is $6.4$ seconds for a network of size $200$ and only $1.4$ min as we increase the number of vertices to $800$. 

Our second example exhibits how the density of the network affects our estimators. To this end, we consider Erd\H{o}s-R\'{e}nyi graph on $n=200$ vertices with $p=0.5,0.01,0.001$ with the parameter choices same as previous example. As shown by Lemma \ref{lemma:mean_field_examples}, such networks fall under the purview of Algorithm \ref{algo:de_method}. We applied Algorithm \ref{algo:de_method} with $M=500$ iterations and replicated the experiment $100$ times to construct  confidence intervals for the direct and indirect effects. The results are summarized in Table \ref{tab:erdos_renyi};  the reported  intervals contain the true direct and indirect effects in all three cases. As before, we note the run-time to compute our estimator, which is $\sim 6$ seconds.

Next, we consider a Gaussian interaction matrix $\bA_n(i,j) \sim N(0, \beta^2 /n)$, with parameters same as \eqref{eq:sims_param} and the covariates $\bX_i \sim \text{Unif}(-1,1)$ i.i.d.  As described by Theorem \ref{thm:amp}, we invoke Algorithm \ref{alg:amp} to compute the direct and indirect effect. We again use the maximum pseudo-likelihood method to compute $(\hat \tau,\hat \btheta)$. Next we solve the fixed point equation \ref{eq:fixed_pt} using Monte Carlo  with sample size $1000$. In Table \ref{tab:AMP}. we report our estimators of direct and indirect effect, $\widehat{\text{IE}}_M, \hat{\mathrm{DE}}_M$, based on $M=500$ iterations of Algorithm \ref{alg:amp} across sample sizes $n= 200, 400, 800$. We construct confidence intervals based on $100$ iterations of our algorithm. Our interval estimator contains the true direct and indirect effects. The run-time of the Algorithm is modest, with $~14$ seconds for network of size $200$ to $1.81$ minutes for a large network of size $800$.

Finally, we note that all run-times are based on Macbook Air $(2020)$ with $M1$ chip and $8$-core CPU.

	\begin{table}[ht]
 
 \begin{center}
  \caption { Estimated value of causal estimands (direct and indirect effect) for complete graph across different sample sizes along with the run-time required to compute those estimators.  The estimators are computed using Algorithm \ref{algo:de_method}.} \label{tab:curie_weiss}
	\begin{small}
		\begin{tabular}{ p{2cm} p{2cm} p{4cm}  p{3 cm} }
			\toprule
			Sample size& {\small Truth} & {\small Confidence Interval} & {\small Run-time (sec)}\\
			\midrule
200   &    &    & 6.4\\
   
			Direct   & 0.26   &   [0.21,0.36] & \\
   Indirect   & 0.17   &  [0.11,0.24]  &  \\

400   &    &    & 21.9\\
   
			Direct   & 0.2   &   [0.18,0.3] & \\
   Indirect   & 0.15   &   [0.11,0.21] & \\

   800   &    &    & 84\\
   
			Direct   & 0.24   &   [0.21,0.31] & \\
   Indirect   & 0.16   &   [0.14, 0.21] & \\

			\bottomrule
		\end{tabular}
	\end{small}
  \end{center}
\end{table}

\begin{table}[ht]
\caption { Estimated value of causal estimands for Erd\H{o}s-R\'{e}nyi graph on $n=200$ vertices across different network densities ($p=0.5,0.01,0.001$).} \label{tab:erdos_renyi}
 \begin{center}
	\begin{small}
		\begin{tabular}{ p{2cm} p{2cm} p{4cm}  p{3 cm} }
			\toprule
			p& {\small Truth} & {\small Confidence Interval} & {\small Run-time (sec)}\\
			\midrule
0.5   &    &    & 6.31\\
   
			Direct   & 0.20   &   [0.02,0.24] & \\
   Indirect   & 0.15   &  [0.04,0.17]  &  \\

0.01   &    &    & 6.5\\
   
			Direct   & 0.21   &   [0.15,0.35] & \\
   Indirect   & 0.15   &   [0.12,0.26] & \\

   0.001   &    &    & 6.33\\
   
			Direct   & 0.29   &   [0.09,0.32] & \\
   Indirect   & 0.09   &   [0.06,0.2] & \\

			\bottomrule
		\end{tabular}
	\end{small}
  \end{center}
 
\end{table}

\begin{table}[ht]
 
 \begin{center}
  \caption {Estimated value of causal estimands for Gaussian interaction matrix across different sample sizes. The estimators are computed using Algorithm \ref{alg:amp}.} \label{tab:AMP}
	\begin{small}
		\begin{tabular}{ p{2cm} p{2cm} p{4cm}  p{3 cm} }
			\toprule
			Sample size& {\small Truth} & {\small Confidence Interval} & {\small Run-time (sec)}\\
			\midrule
200   &    &    & 13.92\\
   
			Direct   & 0.28   &   [0.22,0.38] & \\
   Indirect   & 0.12   &  [0.08,0.21]  &  \\

400   &    &    & 34.3\\
   
			Direct   & 0.23   &   [0.16,0.28] & \\
   Indirect   & 0.1   &   [0.07,0.14] & \\

   800   &    &    & 108.6\\
   
			Direct   & 0.25   &   [0.15,0.26] & \\
   Indirect   & 0.12   &   [0.07,0.13] & \\

			\bottomrule
		\end{tabular}
	\end{small}
  \end{center}
\end{table}

\section{Discussion}
\label{sec:discussion}
    We discuss follow up questions arising from our results, and collect primary thoughts regarding their resolution.

    \begin{enumerate}
        \item[(i)] Beyond the ``high temperature" condition: Theorem~\ref{thm:de_mean_field} involves the condition $\|\bA_n \| < C_0$ while Theorem~\ref{thm:amp} invokes the condition $\beta < \bar{\beta}_0$. These conditions correspond to ``high-temperature" conditions in Ising Models. It is of natural interest to explore causal effect estimation under Ising models at low temperature. At it's heart, Algorithms~\ref{algo:de_method} and ~\ref{alg:amp} directly target the one-dimensional means $\langle \bY \rangle$. A recent line of work provides rigorous evidence that sampling from these models should be intractable at low temperature \cite{kunisky2024optimality,galanis2024sampling}. In light of these results, we believe that our strategy should also fail for low temperature Ising models. We emphasize that this does not suggest that the estimation of causal effects is necessarily intractable at low temperature. We leave a systematic exploration of this issue to future work. 
        
        Additionally, we note that we have not tried to optimize the high-temperature constants $C_0$ and $\bar \beta_0$ appearing in our proofs. While it is possible to track these constants explicitly from our arguments, we do not believe they are optimal. Consequently, we leave these parameters implicit in our discussion.


        \item[(ii)] Beyond quadratic interaction: A natural follow up direction concerns more general MRFs for the outcome regression model \eqref{eq:define_gibbs} and the propensity score model \eqref{eq:prop_score}. The tensor Ising model  \cite{akiyama2019phase,sasakura2014ising} provides a natural starting point in this direction. The main challenge in analyzing this model lies in consistent parameter estimation.  
        While parameter estimation in the tensor Ising model has gained some recent attention~\cite{liu2024tensor,mukherjee2022estimation}, it is relatively under-explored. It would be interesting to extend the existing  estimation guarantees to the setting of causal inference.


        \item[(iii)] General covariate distributions: We assume that the covariates $\bX_i \in \mathbb R^d$ are i.i.d. and compactly supported. Our arguments should extend to settings with MRF dependence on $\{\bX_i: 1\leq i \leq n\}$, provided this dependence is sufficiently `weak'. ``Correlation-decay" properties for MRFs might be crucial in this regard \cite{gamarnik2013correlation}. Our proofs are involved even under the stylized covariate distribution assumed herein. In this light, we defer generalizations in this direction to future works.

        Additionally, we assume throughout that the the covariate distribution $\mathbb{P}_X$ is known. This is not a strong assumption. If $\mathbb{P}_X$ is unknown, we can estimate it (non-parametrically) from the data. We can then use the estimated distribution $\hat{\mathbb{P}}_X$ for downstream causal effect estimation---our methods are stable under this perturbation, and the guarantees go through unchanged.

        \item[(iv)] Beyond bounded responses: We assume througout that the outcome density $f$ in \eqref{eq:define_gibbs} is compactly supported. We believe that under appropriate tail-decay conditions on $\mu$, the results in this paper should generalize. Going beyond the bounded support assumption requires extending the theory of parameter estimation (Theorem \ref{thm:param_estimate}) to probability measures on unbounded spaces. This is beyond the scope of the current paper. 

        \item[(v)] Universality: Algorithm~\ref{alg:amp} provides an algorithm based on Approximate Message Passing (AMP) to estimate the causal effects. There has been substantial recent progress regarding the \emph{universality} of AMP algorithms to the distribution of the matrix $\bA_n$ \cite{bayati2015universality,chen2021universality,dudeja2023universality,wang2022universality}. In this light, one might expect that the same results should go through as long as $\bA_n(i,j)$ are i.i.d mean $0$, variance $1/n$ with sufficiently light tails (e.g. i.i.d. sub-gaussian). While it is certainly true that the AMP algorithm (Algorithm~\ref{alg:amp}) has universal statistics over the distribution of $\bA_n$, here we establish a much stronger property: we establish that the algorithm asymptotically converges to the vector of one-dimensional means $\langle \bY \rangle$. We utilize gaussianity crucially in our proofs. One would expect this property to also be universal (particularly at high-temperature), but this remains beyond the reach of existing techniques. 

        \item[(vi)] DAG: We adopt the MRF based framework introduced in \cite{tchetgen2021auto} in this work. As discussed in the introduction, MRF based models are not necessarily compatible with the traditional Directed Acyclic Graph (DAG) based formalism in causal inference. An independent research thread extends DAG based ideas to study the interference problem (see e.g. \cite{emmenegger2022treatment,ogburn2024causal,van2014causal} and references therein). It would be interesting to explore if the ideas introduced in this work can also yield new insights under the DAG-based framework. We leave this for future work. 

        \item[(vii)] Other causal estimands e.g. GATE: We focus exclusively on the Direct and Indirect effects in this paper. Alternative causal estimands are also popular in the interference literature e.g. the Global Average Treatment Effect (GATE). Formally, the GATE is defined as $\frac{1}{n} \sum_{i=1}^{n} \mathbb{E}[\bY_i(\mathbf{1})] - \mathbb{E}[\bY_i(-\mathbf{1})]$. This estimand is particularly well-suited to the partial interference setting, since one can naturally deploy cluster randomized experiments under the partial interference assumption. This estimand is less natural under the MRF framework under consideration. However, we believe it would be interesting to explore estimation of other causal estimands within our framework. 
    \end{enumerate}

\section{Proofs}
\label{sec:proofs}

We present our proofs in this section. We start with the proof of Lemma~\ref{lem:de_define}. We prove Theorems~\ref{thm:de_mean_field} and \ref{thm:amp} in Sections~\ref{sec:proof_mf} and \ref{sec:proof_amp} respectively. We establish the consistency of pseudo-likelihood estimates in Section~\ref{sec:proof_pseudolikelihood}. Finally, we prove concentration under log-concavity in Section~\ref{sec:proof_conc_logconcavity}.

\begin{proof}[Proof of Lemma \ref{lem:de_define}]
    Recall that
    \begin{align*}
        \text{DE}_{i}(\mathbf{t}_{-i}) &= \mathbb{E}_{\bar \bX} \big[\mathbb{E}[\mathbf{Y}_i| T_i=1,\mathbf{T}_{-i}=\mathbf{t}_{-i},\bar \bX]-\mathbb{E}[\mathbf{Y}_i|T_i=-1,\mathbf{T}_{-i}=\mathbf{t}_{-i},\bar \bX]\big]\\
        &= \mathbb{E}_{\bar \bX} \big[\mathbb{E}[T_i \mathbf{Y}_i| T_i=1,\mathbf{T}_{-i}=\mathbf{t}_{-i},\bar \bX]+\mathbb{E}[T_i\mathbf{Y}_i|T_i=-1,\mathbf{T}_{-i}=\mathbf{t}_{-i},\bar \bX]\big]\\
        &= 2\mathbb{E}_{\bar \bX \bar T_i}[\bar T_i \mathbb{E}[\bY_i| \mathbf{T}=(\bar T_i, \bt_{-i}), \bar \bX]].
    \end{align*}
    Therefore we have, by definition of direct effect,
    \begin{align*}
        \text{DE}&=\frac{1}{n}\sum_{i=1}^{n}\sum_{\mathbf{t}_{-i}\in \{\pm 1\}^{n-1}} 2^{-(n-1)} \text{DE}_i(\mathbf{t}_{-i})\\
        &=\frac{2}{n} \sum_{i=1}^{n} \ee_{\bar \bT_{-i}} \ee_{\bar \bX, \bar T_i} \bar{T}_i \langle \mathbf{Y}_i \rangle\\
        &=\frac{2}{n} \sum_{i=1}^{n} \ee_{\bar \bT, \bar \bX} \bar{T}_i \langle \mathbf{Y}_i \rangle.
    \end{align*}
Next, recall from \eqref{eq:define_ie_i} that,
 $\text{IE}_{i}(\bt_{-i}):= \mathbb{E}_{\bar \bX} \big[\mathbb{E}[\bY_i|T_i=-1,\bT_{-i}=\bt_{-i},\bar\bX]-\mathbb{E}[\bY_i|T_i=-1,\bT_{-i}=-\mathbf{1},\bar\bX]\big]$ and hence,  
 \begin{align}\label{eq:ie_split}
     \text{IE}&=\frac{1}{n}\sum_{i=1}^{n}\sum_{t_{-i}\in \{\pm 1\}^{n-1}} 2^{-(n-1)}\text{IE}_i(\bt_{-i}) \nonumber\\
     &= \frac{1}{n}\sum_{i=1}^{n}\sum_{\bt_{-i}\in \{\pm 1\}^{n-1}} \frac{1}{2^{n-1}} \left(\mathbb{E}_{ \bar\bX} \big[\mathbb{E}[\bY_i|T_i=-1,\bT_{-i}=\bt_{-i},\bar\bX]-\mathbb{E}[\bY_i|T_i=-1,\bT_{-i}=-\mathbf{1}, \bar\bX]\big]\right)\nonumber\\
     &:= \mathcal{H}_n - \frac{1}{n}\sum_{i=1}^{n}\mathbb{E}_{ \bar\bX}\mathbb{E}[\bY_i|\bT=-\mathbf{1},\bar\bX]= \mathcal{H}_n - \mathbb{E}_{\bar \bX}\mathbb{E}[\bar \bY|\bT=-\mathbf{1},\bar\bX],
 \end{align}
where we denote the first summand by $\mathcal{H}_n$. Setting $\alpha_{in}= \mathbb{E}_{\bar \bX} \big[\mathbb{E}[\bY_i|T_i=1,\bT_{-i}=\bt_{-i},\bar \bX]\big]$ , $\gamma_{in}= \mathbb{E}_{ \bar \bX} \big[\mathbb{E}[\bY_i|T_i=-1,\bT_{-i}=\bt_{-i}, \bar \bX] \big]$, note that,
\begin{align}\label{eq:hn}
    \mathcal{H}_n &= \frac{1}{n}\sum_{i=1}^{n} \bE_{\bar \bT_{-i}} \gamma_{in}=  \frac{1}{2n}\sum_{i=1}^{n} \bE_{\bar \bT_{-i}} (\alpha_{in}+ \gamma_{in}) -  \frac{1}{2n}\sum_{i=1}^{n} \bE_{\bar \bT_{-i}} (\alpha_{in}- \gamma_{in}) \nonumber\\
    &=  \frac{1}{2n}\sum_{i=1}^{n} \bE_{\bar \bT_{-i}} (\alpha_{in}+ \gamma_{in}) - \frac{1}{2}\text{DE} \nonumber\\
    &= \bE_{\bar \bT, \bar \bX} \langle \bar \bY \rangle- \frac{1}{2}\text{DE}.
\end{align}
 This completes the proof.
\end{proof}

\subsection{Proof of Theorem~\ref{thm:de_mean_field} and associated results} 
\label{sec:proof_mf}

\begin{proof}[Proof of Theorem \ref{thm:de_mean_field}]
Assume that $\|\bA_n\|< \frac{1}{4}$ for all sufficiently large $n$. For any $\bv\in [-1,1]^n$, 
define the function 
\begin{equation}\label{eq:def_mn}
M_n(\bv):= \frac{1}{2}\bv^\top \bA_n \bv + \sum_{i=1}^{n} v_i(\tau_0 t_i+ \btheta^\top_0 \bx_i)- \sum_{i=1}^{n} I(v_i),   
\end{equation}
where $I(v_i) = D(\mu_{(\lambda_1(v_i),0)}|\mu)$ is defined by Definition \ref{def:exp_tilt}. We note that for $\bv \in (-1,1)^n$,  
\begin{equation*}
    \nabla M_n(\bv)= \bA_n \bv + (\tau_0 \bt+ \bX \btheta_0)- \lambda_1(\bv),
\end{equation*}
where $\alpha'(\lambda_1(\bv),0)= \bv$. 
Differentiating once more, we have,  
\begin{equation*}
    \nabla^2 M_n(\bv)= \bA_n- (\alpha''(\lambda_1(\bv), \mathbf{0}))^{-1}, \quad \text{where  } \alpha''(\lambda_1(\bv), \mathbf{0})=\text{diag}(\alpha''(\lambda_1(v_i), 0)).
\end{equation*}
Since $\alpha''(\lambda_1(v_i),0)$ is the variance of a random variable supported on $[-1,1]$, we have $\min_{\bv}\|(\alpha''(\lambda_1(\bv),0))^{-1}\| \ge 1$. Hence, if $\|\bA_n\|<1$, we have $- \nabla^2 M_n(\bv) \succ 0$ for all $\bv \in (-1,1)^n$, yielding the strong concavity  of $M_n$ on $(-1,1)^n$. 
In addition, noting that $\lambda_1(w) \to \pm \infty$ if $w \to \pm 1$ implies that any global maximizer of $M_n$ lies in $(-1,1)^n$.

This implies the following strong separation property of the global  maximizer $\bv^\star$ of the function $M_n$: for any $\varepsilon>0$, there exists $\delta=\delta(\varepsilon)>0$ such that,
\begin{equation}\label{eq:mn_concave}
    \frac{1}{n} \sup_{\bv: \frac{1}{n}\|\bv-\bv^\star\|^2 > \varepsilon} M_n(\bv) < \frac{1}{n} \sup_{\bv} M_n(\bv) -\delta.
\end{equation}
Moreover,  the global maximizer $\bv^\star$ satisfies $\nabla M_n(\bv^\star)=0$ which can be rewritten as 
\begin{equation*}
    \bv^\star= \alpha'(\bA_n \bv^\star+ \tau_0 \bt+ \bX \btheta_0,0).
\end{equation*}

Fix $\bt,\bX$ for the rest of the proof. Suppose $\by \sim f$ defined by \eqref{eq:define_gibbs}.  Define the vector of conditional probabilities 
\begin{equation}\label{eq:define_b}
    \mathbf{b}(\by):= \bE_f(\by_i| \by_{-i}).
\end{equation}
Note that by \cite[Theorem 1 (ii)]{mukherjee2022variational}, we have
\begin{equation*}
    \frac{1}{n} \Big(M_n(\mathbf{b}(\by))-M_n(\bv^\star)\Big) \rightarrow 0 \implies  \frac{1}{n} \|\mathbf{b}(\by)- \bv^\star\|^2 \rightarrow 0,
\end{equation*}
using \eqref{eq:mn_concave}. This further yields, using Jensen's inequality and DCT that,

\begin{equation*}
    \frac{1}{n}\|\langle \by \rangle -\bv^\star\|^2 \le \frac{1}{n} \ee_f[\|\mathbf{b}(\by)- \bv^\star\|^2] \rightarrow 0.
\end{equation*}
In addition,
\begin{align*}
    \varepsilon_n:=\frac{1}{n}\|\langle \by \rangle- \alpha'(\bA_n \langle \by \rangle+ \tau_0 \bt+ \bX \btheta_0,0)\|^2 \le \frac{2}{n}\|\langle \by \rangle -\bv^\star\|^2 + \frac{2}{n} \|\alpha''\|_{\infty} \|\bA_n\| \|\langle \by \rangle -\bv^\star\|^2 \rightarrow 0,
\end{align*}
since $\|\bA_n\| <1 $ and $\|\alpha''\|_{\infty} \le 1$. Finally, regarding the iterates $\mathbf{u}^{(l)}$, we have
\begin{align*}
    D_{n,l+1}:=\frac{1}{n}\|\langle \by \rangle- \mathbf{u}^{(l+1)}\|^2 &\le 3 \varepsilon_n + \frac{3}{n} \|\alpha''\|_{\infty} \|\bA_n\| \|\langle \by \rangle -\mathbf{u}^{(l)}\|^2\\
    &= 3 \varepsilon_n+  3\|\bA_n\| D_{n,l}
\end{align*}
 If $\|\bA_n\|< 1/4$, we obtain $D_{n,l+1} 
\le 3\varepsilon_n+ C D_{n,l}$ for some $C<1$. Repeating this iteration, we see that 
\[ \lim\limits_{l \rightarrow \infty}\lim\limits_{n \rightarrow \infty} D_{n,l}=0.\]
This implies, by Cauchy-Schwarz inequality,
\[ \lim\limits_{l \rightarrow \infty}\lim\limits_{n \rightarrow \infty} \frac{1}{n}\sum_{i=1}^{n}t_i (\langle y_i \rangle- u^{(l)}_i) \le \lim\limits_{l \rightarrow \infty}\lim\limits_{n \rightarrow \infty} \sqrt{ D_{n,l}} = 0.\]
Finally, we conclude the proof of the result taking expectation over $\bt,\bX$ and applying DCT. The proof for $\widehat{\text{IE}}$ is analogous, and thus omitted.  
\end{proof}

\begin{proof}[Proof of Lemma \ref{lemma:mean_field_examples}]
    Note that by \cite[Corollary 1.2]{basak2017universality}, we have $\text{Tr}(\bA^2_n)=o(n)$ if $\frac{n}{|E_n|}= o(1)$, where recall that, $|E_n|$ is the number of edges of the graph. This shows $\text{Tr}(\bA^2_n)=o(n)$ for complete and regular graphs. Further, $\max_{j} \sum_{i=1}^{n} |\bA_n(i,j)|= O(1)$ for both these cases, yielding $\|\bA_n\|=O(1)$.

    For Erd\H{o}s-R\'{e}nyi graphs, total number of edges, $|E_n| \sim \text{Bin} (\binom{n}{2}, p_n)$, implying $\frac{2|E_n|}{n^2 p_n} \rightarrow 1$, as long as $n^2p_n \rightarrow \infty$. This shows if $np_n \rightarrow \infty$, we have $\text{Tr}(\bA^2_n)=o(n)$. Also, \cite[Theorem 1.1 and Lemma 2.1]{krivelevich2003largest} implies $\|\bA_n\|= O(1)$ proving our claim for Erd\H{o}s-R\'{e}nyi graphs.

For the graphon model, note that, 
\begin{align*}
    \ee \text{Tr}(\bA^2_n)= \frac{1}{n^2 \rho^2_n} \sum_{i,j} \rho_n W(U_i,U_j) = (1+o(1)) \frac{1}{\rho_n} \iint_{[0,1]^2}W(x,y)dx dy = o(n),
\end{align*}
if $n \rho_n \rightarrow \infty$. One can similarly compute $ \text{Var} (\text{Tr}(\bA^2_n))$  to obtain $\text{Tr}(\bA^2_n)=o(n)$, with high probability. Finally, invoking \cite[Corollary 1.3]{bbk}, we obtain $\|\bA_n\|= O(1)$ with high probability, completing our proof. 
\end{proof}

\begin{proof}[Proof of Lemma \ref{lemma:data_driven}]
    Note that, since $(\bar \bT^{(1)}, \bar \bX^{(1)}), \ldots, (\bar \bT^{(k)}, \bar \bX^{(k)})$ are i.i.d., and $|\widehat{\text{DE}}^{(j)}| \leq 1$ for $1\leq j \leq k$, uniformly over $n$, $M$, we have, 
 $$ \left|{\widehat{\mathrm{DE}}_M}^{\mathrm{avg}} - \ee (\widehat{\mathrm{DE}}_M) \right| = O_P\Big(\frac{1}{k} \Big). $$
    This, along with Theorem \ref{thm:de_mean_field} yields,
    $\lim_{k \to \infty} \lim_{M \to \infty} \lim_{n \to \infty} \mathbb{P}\Big[|{\widehat{\mathrm{DE}}_M}^{\mathrm{avg}} - \mathrm{DE}| > \varepsilon \Big] =0$. The same argument holds for $\widehat{\mathrm{IE}}_M$. 
\end{proof}

\subsection{Proof of Theorem~\ref{thm:amp} and Lemma~\ref{lemma:uniqueness}} 
\label{sec:proof_amp}
\begin{proof}[Proof of Theorem \ref{thm:amp}]
The proof follows from the results derived in Appendix \ref{appendix: sk_model} and Appendix \ref{appendix: amp+cavity}. We provide a outline of the proof below.

The distribution $f$ given by \eqref{eq:define_gibbs} is equivalent to \eqref{eq:SK_model} with $h_i= \tau_0 \bar T_i+ \btheta^\top_0 \bar \bX_i$. Hence, step iv. of Algorithm \ref{alg:amp} is equivalent to the AMP algorithm described by \eqref{app_B: AMP_description}. Therefore, invoking Theorem \ref{thm:amp_good}, we obtain 

\begin{equation*}
        \lim\limits_{M \rightarrow \infty} \lim\limits_{n \rightarrow \infty} \ee \|\langle \by \rangle - \mathbf{m}^{[M]} \|^2=0,
    \end{equation*}
where $\mathbf{m}^{[M]}$ is defined via Algorithm \ref{alg:amp}. In the display above, $\mathbb{E}[\cdot]$ denotes the expectation with respect to $(\bA_n, \bar \bT, \bar \bX)$. So, we have 
\begin{align*}
   \mathbb{E}\Big| \hat{\mathrm{DE}}_M - \mathrm{DE} \Big| = \frac{1}{n} \ee \Big| \bar \bT^\top (\langle \by \rangle - \mathbf{m}^{[M]})\Big| \le \frac{1}{\sqrt{n}}\ee \|\langle \by \rangle - \mathbf{m}^{[M]} \| \rightarrow 0,
\end{align*}
as $n \rightarrow \infty$ followed by $M \rightarrow \infty$, where the last step follows from Cauchy-Schwarz inequality.

The proof for indirect effect follows the same way after setting $\bar T_i=-1$, $i \in [n]$. This completes the proof of the Theorem.
\end{proof}

\begin{proof}[Proof of Lemma \ref{lemma:uniqueness}]
    The proof of Lemma \ref{lemma:uniqueness} follows immediately from Lemma \ref{lem:fixed_point_unique}, and we omit the details for brevity.
\end{proof}

\subsection{Proof of Theorem~\ref{thm:param_estimate} and Lemma~\ref{lem:algo_stable}}
\label{sec:proof_pseudolikelihood}

\begin{proof}[Proof of Theorem \ref{thm:param_estimate}]

We will assume, w.l.o.g., that $ \|\mathbf{A}_n\| \le 1$, $\|\mathbf{M}_n\| \le 1$ for large $n$ by appropriate scaling. Also assume, $|\tau_0| \le 1$, i.e., $B=1$ in Assumption \ref{assn:parameter_space} by scaling.

    The proof has two components: (i) to show that the Hessian of the log-pseudo-likelihood has minimum eigenvalue bounded away from $0$, and (ii) to prove that the gradient of log-pseudo-likelihood concentrates at a rate $\frac{1}{\sqrt{n}}$. 

Defining $m_i(\bY)=\sum_{j=1}^{n}\mathbf{A}_{n}(i,j)Y_j$, the log-pseudo-likelihood is given by 
\begin{equation}\label{eq:define_log_pl}
    l(\tau,\boldsymbol{\theta})= \frac{1}{n}\sum_{i=1}^{n}[Y_i m_i(\bY)+ Y_i(\tau T_i + \boldsymbol{\theta}^\top \bX_i) - \log B_i],
\end{equation}
where $B_i= \alpha (m_i(\bY)+ \tau T_i + \boldsymbol{\theta}^\top \bX_i, 0)$. Define $\tilde \bX_i= (\bX_i, T_i)$. The gradient of pseudo-likelihood is given by 
\begin{equation}\label{eq:define_gradient}
    \nabla l(\tau, \boldsymbol{\theta})= \frac{1}{n}\sum_{i=1}^n \tilde \bX_i (Y_i - \alpha' (m_i(\bY)+ \tau T_i + \boldsymbol{\theta}^\top \bX_i,0)),
\end{equation}
and the Hessian has the form:
\begin{equation}\label{eq:define_hessian}
    -H_{\tau,\boldsymbol{\theta}}= \frac{1}{n} \sum_{i=1}^{n} \tilde \bX_i\tilde \bX^\top_i \alpha'' (m_i(\bY)+ \tau T_i + \boldsymbol{\theta}^\top \bX_i,0).
    \end{equation}

We first show tightness of the gradient \eqref{eq:define_gradient}. Lemma \ref{Lem:grad_concentration} shows for any $k=1,2, \ldots, d+1$
$$\bE(\nabla l(\tau_0,\boldsymbol{\theta}_0)^2_k) = \frac{1}{n^2} \bE \left( \sum_{i=1}^{n} \tilde \bX_{i,k}(Y_i - \alpha' (m_i(\bY)+ \tau_0 T_i + \boldsymbol{\theta}^\top_0 \bX_i,0))\right)^2 \lesssim\frac{1}{n},$$
which implies, there exist $c= c(d,M)>0$ such that for any $\eta>0$
$$\bP\Big(\|\nabla l(\tau_0,\boldsymbol{\theta}_0)\|^2 \le \frac{c}{n \eta} \Big) \ge 1- \eta.$$

Next, we turn to  proving that the minimum eigenvalue of $-H_{\tau_0,\boldsymbol{\theta}_0}$, henceforth denoted by $\lambda_{\min}(-H_{\tau_0,\boldsymbol{\theta}_0})$, is bounded away from zero with high probability.

Fix $C>0$ to be chosen later. Define $S_C:= \{i: |m_i(\bY)|> C\}$. Then, using $m_i(\bY)=\sum_{j=1}^{n}\mathbf{A}_{n}(i,j)Y_j$, 
\begin{equation}\label{eq:mi_bound}
     |S_C|= \Big|\{i: |m_i(\bY)| > C \}\Big| \le \frac{1}{C^2} \sum_{i=1}^n m^2_i(\bY) \le \frac{n}{C^2} \|\mathbf{A}^2_n\|\le \frac{n}{C^2}.
\end{equation}
If $i \notin S_C$, then $|m_i(\bY)+ \tau T_i + \boldsymbol{\theta}^\top \bX_i| \le C_1$ for some $C_1>0$ and for all $(\tau,\boldsymbol{\theta})\in [-1,1]\times [-M,M]^d$. This implies, since $\alpha''>0$ and continuous, that there exists $\delta>0$ such that, for all $i \notin S_C$, $\max_{(\tau,\boldsymbol{\theta})}\alpha'' (m_i(\bY)+ \tau T_i + \boldsymbol{\theta}^\top \bX_i,0) >\delta$.
Therefore, we have 
\begin{align*}
    -H_{\tau,\boldsymbol{\theta}} &\succeq \frac{\delta}{n} \sum_{i=1}^n \tilde \bX_i\tilde \bX^\top_i 1_{i \in S^c_{C}}\\
    &= \frac{\delta}{n} \Big( \sum_{i=1}^n \tilde \bX_i\tilde \bX^\top_i - \sum_{i \in S_C} \tilde \bX_i\tilde \bX^\top_i\Big),
\end{align*}
which implies
\begin{align*}
    \lambda_{\min}(-H_{\tau_0,\boldsymbol{\theta}_0}) &\ge \delta \left( \lambda_{\min}\Big(\frac{1}{n}\sum_{i=1}^n \tilde \bX_i\tilde \bX^\top_i\Big) - \frac{1}{n} \left\| \sum_{i \in S_C} \tilde \bX_i\tilde \bX^\top_i \right\| \right)\\
    & \ge  \delta \left( \lambda_{\min}\Big(\frac{1}{n}\sum_{i=1}^n \tilde \bX_i\tilde \bX^\top_i\Big) - \frac{1}{C^2} (dM+1) \right),
\end{align*}
where the last inequality follows from $\| \sum_{i \in S_C} \tilde \bX_i\tilde \bX^\top_i\| \le |S_C| \max_{i} \|\tilde \bX_i\|^2_2 \le \frac{n}{C^2} (dM+1)$.
The rest of the proof shows there exists $c>0$ such that 
\begin{equation}\label{eq:hess_mid}
    \lambda_{\min}\Big(\frac{1}{n}\sum_{i=1}^n \tilde \bX_i\tilde \bX^\top_i\Big)>c
\end{equation}
with probability $1-o(1)$. Afterwards, one can choose $C$ such that $c- \frac{1}{C^2} (dM+1) >0$, completing the proof. For notational convenience, we set $G_n:= \frac{1}{n}\sum_{i=1}^n \tilde \bX_i\tilde \bX^\top_i$ for the rest of the proof.

Turning to prove \eqref{eq:hess_mid}, note that $G_n$ can be rewritten as 
$$G_n= \frac{1}{n}\begin{bmatrix}
\sum_{i=1}^n \bX_i \bX^\top_i&  \sum_{i=1}^n T_i \bX_i \\
\sum_{i=1}^n T_i \bX^\top_i & n
\end{bmatrix}$$

Note that, $ \lambda_{\min} (\frac{1}{n}\sum_{i=1}^n \bX_i \bX^\top_i) \ge \frac{1}{2} \lambda_{\min}(\boldsymbol{\Sigma}) \ge C_1>0$ with probability $1-o(1)$. Hence, it is enough to show, 
\begin{align}\label{eq:schur}
    1- \frac{1}{n} \bT^\top \bX (\bX^\top \bX)^{-1} \bX^\top \bT > c >0
\end{align}
for some $c>0$, with probability $1-o(1)$. Since $\mathbf P := \bX (\bX^\top \bX)^{-1} \bX^\top$ is a projection matrix of rank at most $d$, we can have orthonormal vectors $\bv_1,\ldots, \bv_d$ such that $\mathbf P= \sum_{l=1}^d \bv_l \bv^\top_l$. Define $\phi_i(\bT)= \bE(T_i|\bT_{-i})= \tanh((\mathbf{M_n}\bT)_i+ \gamma^\top_0 \bX_i)$. Invoking Lemma \ref{Lem:grad_concentration}, we have $$\ee \Bigg(\sum_{i=1}^n (\bv_l)_i (T_i - \phi_i (\bT)) \Bigg)^2 \lesssim 1$$ for all $l \in [d]$, where $\phi(\bT)= (\phi_1(\bT), \ldots, \phi_n(\bT))$. Therefore, by Markov's inequality, $\bv^\top_l(\bT- \phi(\bT)) = O_{P}(1)$. Hence, using an union bound, we have 
\begin{align*}
    \frac{1}{n}\left(\bT^\top \mathbf P \bT - \phi(\bT)^\top \mathbf P \phi(\bT) \right) & =\frac{1}{n} \sum_{l=1}^{d}  ( \bv^\top_l (\bT- \phi(\bT))( \bv^\top_l (\bT+ \phi(\bT))\\
    & =O_P \Big(\frac{d}{\sqrt{n}}\Big)= o_P(1).
\end{align*}
Further, since $\|\mathbf{P}\|= 1$, we have $\phi(\bT)^\top \mathbf P \phi(\bT) \le  \|\phi(\bT)\|^2$. This implies that, with probability $1-o(1)$, 
\begin{align*}
     1- \frac{1}{n} \bT^\top \bX (\bX^\top \bX)^{-1} \bX^\top \ge 1- \frac{1}{n} \sum_{i=1}^n \phi^2_i(\bT) -\varepsilon
     \end{align*}
for some small $\varepsilon>0$. Finally, there exists $C_2>0$
\begin{align*}
    1- \frac{1}{n} \sum_{i=1}^n \phi^2_i (\bT)= \frac{1}{n} \sum_{i=1}^n\sech^2((\mathbf{M_n}\bT)_i+ \gamma^\top_0 \bX_i) \ge C_2,
\end{align*}
by the same argument as \eqref{eq:mi_bound}. This completes the proof of \eqref{eq:schur}, which in turn implies $\lambda_{\min}(-H_{\tau,\boldsymbol{\theta}}) \ge c$ for some $c>0$.

This shows, via \cite[(47)]{daskalakis2019regression}, that
\begin{align}
    \|(\hat \tau_{\text{MPL}}, \hat{ \boldsymbol{\theta}}_{\text{MPL}}) -(\tau_0,\boldsymbol{\theta}_0)\|_2 \le \frac{2\|\nabla l(\tau_0,\boldsymbol{\theta}_0)\|_2}{\min_{(\tau,\boldsymbol{\theta})\in \mathcal{A}}\lambda_n(-H_{\tau,\boldsymbol{\theta}})} =O_{\bP}\left(\frac{1}{\sqrt{n}}\right).
\end{align}
This concludes the proof.
\end{proof}

\begin{lemma}\label{Lem:grad_concentration}
    Suppose $\psi:[-1,1]^{d+1} \mapsto [-M,M]$ is a bounded measurable function, and $\|\mathbf{A}_n\| \le 1$. Then we have 
    \begin{align*}
        \bE\Big[\big(\sum_{i=1}^n \psi(\bX_i,T_i) (Y_i -\bE(Y_i| \bY_{-i},\bT,\bX)\big)^2|\bX,\mathbf{T}\Big] \le C_1 n,
    \end{align*}
    for some $C_1>0$, which depends on $M$ but not on the choice of $\psi$. Moreover, if $\mathbf c(\bX):= (c_1(\bX),\ldots, c_n(\bX))$ satisfies $\|c(\bX)\|_2 \le 1$, then we have

    \begin{equation*}
         \bE\Big[\big(\sum_{i=1}^n c_i(\bX) (T_i -\bE(T_i| \bT_{-i},\bX)\big)^2|\bX\Big] \le C_2,
    \end{equation*}
    where $C_2$ does not depend on the choice of $\mathbf{c}(\bX)$.
    
\end{lemma}
\begin{proof}
   Without loss of generality, set $M=1$. The proof follows the same path as \cite[Lemma 8]{mukherjee2022variational}, which shows, 
    \begin{align*}
        &\bE\Big[\big(\sum_{i=1}^n \psi(\bX_i,T_i) (Y_i -\bE(Y_i| \bY_{-i})\big)^2|\bX,\mathbf{T}\Big] \nonumber \\
        & \lesssim n +  \sup_{\mathbf{u} \in [-1,1]^n} \sum_{i=1}^n\Big|\sum_{j=1}^{n}\mathbf{A}_{n}(i,j)u_j \Big| + \sum_{i,j=1}^n \mathbf{A}^2_{n}(i,j)\\ &\le n+ 
        \sup_{\mathbf{u} \in [-1,1]^n} \sqrt{n} \|\mathbf{A}_n \bu\|_2 +\text{Tr}(\mathbf{A}^2_n)\\
        & \le C_1 n,
    \end{align*}
    where the second inequality uses Cauchy-Schwarz inequality, along with the fact that $\|\mathbf{A}_n\|\le 1$. Regarding the second inequality, we again invoke \cite[Lemma 8]{mukherjee2022variational} to obtain 
    \begin{align*}
        &\bE\Big[\big(\sum_{i=1}^n c_i(\bX) (T_i -\bE(T_i| \bT_{-i},\bX)\big)^2|\bX\Big] \nonumber \\
        &\lesssim  \| \mathbf{c}\|^2 + \left| \sum_{i \neq j} c_i c_j \bA_{n}(i,j) T_i (T_i- \bE(T_i| \bT_{-i},\bX)\right|+ \left|\sum_{i,j=1}^n c_i c_j \bA^2_{n}(i,j) T_i\right|\\
        & \le \|\mathbf{c}\|^2 +  \|\bA_n\| \|\mathbf{c}\|^2_2+ \|\bA^2_n\| \|\mathbf{c}\|^2_2\\
        &\lesssim 1.
    \end{align*}
This completes the proof.  
\end{proof}

\begin{proof}[Proof of Lemma \ref{lem:algo_stable}]
     To show the limit for Algorithm \ref{algo:de_method}, we will use induction hypothesis on $M$. Denote the iterates corresponding to $(\hat \tau_{\text{MPL}}, \hat {\boldsymbol{\theta}}_{\text{MPL}})$ by $\widehat{\mathbf{u}^{(M)}}$. By definition,
    \begin{align*}
        \mathbf{u}^{(M+1)} - \widehat{\mathbf{u}^{(M+1)}}= \alpha'\Big( \mathbf{A}_n \mathbf{u}^{(M)} + \tau_0 \bar\bT + \bar \bX \mathbf{\theta}_0, 0 \Big)- \alpha'\Big( \mathbf{A}_n \mathbf{u}^{(M)} + \hat \tau_{\text{MPL}} \bar\bT + \mathbf{X} \hat {\boldsymbol{\theta}}_{\text{MPL}}, 0 \Big),
    \end{align*}
    which implies, since $\alpha'$ is Lipschitz, and $\bar \bT, \bar \bX$ are bounded,
    \begin{align*}
        \frac{1}{n} \|\mathbf{u}^{(M+1)} - \widehat{\mathbf{u}^{(M+1)}}\|^2 \lesssim \frac{1}{n} \|\mathbf{u}^{(M)} - \widehat{\mathbf{u}^{(M)}}\|^2+ (\tau_0-\hat \tau_{\text{MPL}})^2+ \|\btheta_0-\hat {\boldsymbol{\theta}}_{\text{MPL}}\|^2,
    \end{align*}
    which converges to $0$ in probability using induction.

    Therefore, we obtain, 

    \begin{align*}
        \left|\hat{\mathrm{DE}}_M(\tau_0, \btheta_0) -\hat{\mathrm{DE}}_M(\hat \tau_{\text{MPL}}, \hat {\boldsymbol{\theta}}_{\text{MPL}})\right|& = \frac{2}{n} |\bar \bT^\top (\mathbf{u}^{(M)} - \widehat{\mathbf{u}^{(M)}})|\\
        & \lesssim \frac{1}{\sqrt{n}} \|\mathbf{u}^{(M)} - \widehat{\mathbf{u}^{(M)}})\| \rightarrow 0.
    \end{align*}
This concludes the proof for $\text{DE}$. The proof for $\text{IE}$ is analogous.  Finally, the conclusion holds for Algorithm \ref{alg:amp} using Lemma \ref{lem:fixed_point_h_stability}.
\end{proof}

\subsection{Proof of Proposition~\ref{prop:lsi_concentration} and Lemma \ref{lem:MC_CI}}
\label{sec:proof_conc_logconcavity}
 
\begin{proof}[Proof of Proposition \ref{prop:lsi_concentration}]
We start with the proof of part (i) and focus on the direct effect. The proof of indirect effect follows  the same steps and hence is omitted for brevity. Recall from the proof of Theorem \ref{thm:de_mean_field}, given $\bar \bT=(\bar T_1, \ldots, \bar T_n)$ and $\bar \bX= (\bar \bX_1,\ldots, \bar \bX_n)$, we have $\lim\limits_{M \rightarrow \infty}\lim\limits_{n \rightarrow \infty} \frac{1}{n}\| \bu^{(M)} -\langle \by \rangle \|^2=0$. Hence, it is enough to show, 
\begin{equation}
    \lim\limits_{n \rightarrow \infty} \left |\frac{2}{n} \sum_{i=1}^{n} \bar T_i \langle y_i \rangle -  \text{DE} \right| = 0
\end{equation}

To this end, note that, the averaged direct effect can be written in terms of the normalizing constant $Z_n$ defined in  \eqref{eq:define_Z}. Setting $F_n (\tau) := F_n(\tau, \btheta_0)= \frac{1}{n} \log Z_n (\bar \bT, \bar \bX)$, we have

\begin{equation}\label{eq:fn_derivative}
\frac{1}{n} \sum_{i=1}^{n} \bar{T}_i \langle y_i \rangle = F_n'(\tau_0), \,\,\,\,    \text{DE}=2 \bE_{\bar \bT, \bar \bX} F'_n(\tau_0).
\end{equation}

Further, since $F''_n(\tau_0)= \text{Var}_{f}(\frac{1}{n}\sum_{i=1}^{n} \bar T_i y_i) \ge 0$, $F_n$ is a convex function in the first coordinate. Hence, for any fixed $h>0$, we obtain
\begin{equation}\label{eq:fn_convex}
    \frac{F_n(\tau_0) - F_n(\tau_0-h)}{h} \le F'_n(\tau_0) \le \frac{F_n(\tau_0+h) - F_n(\tau_0)}{h}
\end{equation}

Next, we claim that \begin{equation}\label{eq:zn_concentrate}
    \text{Var}_{\bar \bT,\bar \bX}(F_n(\tau)) \le \frac{c}{n},
\end{equation} 
for some $c>0$ independent of $n$. We defer the proof of this claim to the end of the section. Given \eqref{eq:zn_concentrate}, we obtain, using \eqref{eq:fn_convex} that, with probability $1- o(1)$,
\begin{align*}
    \frac{\ee_{\bar\bT,\bar \bX} F_n(\tau_0) - \ee_{\bar\bT,\bar \bX} F_n(\tau_0-h)- \frac{1}{n^{1/4}}}{h}  \le F'_n(\tau_0) \le  \frac{\ee_{\bar \bT,\bar \bX} F_n(\tau_0+h) - \ee_{\bar \bT, \bar \bX} F_n(\tau_0)+ \frac{1}{n^{1/4}} }{h}
\end{align*}
Since $F'_n$ is bounded, the above bounds also hold for $\ee_{ \bar \bT, \bar \bX} F'_n(\tau_0)$. Hence we obtain, with probability $1-o(1)$, for some $\tilde \tau \in [\tau_0-h,\tau_0+h]$,

\begin{align*}
    F'_n(\tau_0) - \ee_{\bar \bT, \bar \bX} F'_n(\tau_0) &\le \frac{ \ee_{\bar \bT, \bar \bX} F_n(\tau_0+h) -2 \ee_{\bar \bT, \bar \bX} F_n(\tau_0)+ \ee_{\bar \bT, \bar \bX} F_n(\tau_0-h) + \frac{2}{n^{1/4}} }{h}\\
    &=  \frac{ \frac{h^2}{2} \ee_{\bar \bT, \bar \bX} F''_n(\tilde \tau) + \frac{2}{n^{1/4}} }{h} = \frac{h}{2} \ee_{\bar \bT,\bar \bX}\text{Var}\left(\frac{1}{n}\sum_{i=1}^{n} \bar T_i y_i \right) + \frac{2}{n^{1/4}h}\\
    &\le \frac{h}{2n^2}  \ee_{\bar \bT,  \bar \bX}[\|\text{Cov}(\by)\| \| \bar \bT\|^2]+ \frac{2}{n^{1/4} h}  \lesssim \frac{h}{n}  + \frac{1}{n^{1/4}h} \rightarrow 0,
\end{align*}
as $n \to \infty$. The first equality follows by mean-value theorem, the second inequality is implied by the definition of operator norm, and the Log-Sobolev inequality \cite[Theorem 1 and Remark (ii)]{bauerschmidt2019very} yields that $\|\text{Cov}(\by)\|= O(1)$ uniformly over $\bt,\bx$. Note that \cite[Theorem 1]{bauerschmidt2019very} applies here since $\mu$ satisfies the  tilted-LSI property.

Similarly, a lower bound can be obtained for $F'_n(\tau_0) - \ee F'_n(\tau_0)$, implying 

\begin{equation*}
    2F'_n(\tau_0)-\text{DE}=  2F'_n(\tau_0)-2 \ee_{\bar \bT, \bar \bX}F'_n(\tau_0) \xrightarrow{\mathbb P}0,
\end{equation*}
which concludes the proof once we have shown \eqref{eq:zn_concentrate}.

Turning to the proof of \eqref{eq:zn_concentrate}, using Efron-Stein inequality \cite{10.1093/acprof:oso/9780199535255.001.0001}, 
\begin{align}
    &\mathrm{Var}_{\bar \bT,\bar \bX}(F_n(\tau)) \nonumber \\ 
    &\leq \frac{1}{2n^2} \mathbb{E}_{\bar \bT, \bar \bX} \Big[ \sum_{i=1}^{n} (\log Z_n (\bar \bT, \bar \bX) - \log Z_n(\bar \bT^{(i)}, \bar \bX))^2 + \sum_{i=1}^{n} (\log Z_n (\bar \bT, \bar \bX) - \log Z_n(\bar \bT, \bar \bX^{(i)}))^2 \Big], \nonumber  
\end{align}
where $\bar \bT^{(i)} = (\bar T_1, \cdots, \bar T_{i-1}, \bar T_i', \bar T_{i+1}, \cdots, \bar T_n)$, and $\bar \bT' = (\bar T_1', \cdots, \bar T_n')$ are i.i.d. $\mathrm{Unif}(\{\pm 1\} )$ independent of $\bt$. Similarly, $\bar \bX^{(i)}=(\bar \bX_1, \cdots, \bar \bX_{i-1}, \bar \bX_i', \bar \bX_{i+1}, \cdots, \bar \bX_n)$, $\bar \bX'= (\bar \bX_1', \cdots, \bar \bX_n')$ are i.i.d. $\mathbb{P}_{X}$ independent of $\bar \bX$. The claim \eqref{eq:zn_concentrate} follows once we establish that each term in the display above is $O(n)$. Without loss of generality, we work with the first term. The bound for the second term is similar, and thus omitted. 

The proof is by interpolation. Fix $1\leq i \leq n$. For $\upsilon \in [0,1]$, set 
\begin{align}
    H(\upsilon) = \log \int_{[-1,1]^{n}} \exp\Big( \mathbf{y}^{\top} \mathbf{A}_n \mathbf{y} + \tau_0 \sum_{j\neq i} y_j \bar T_j + \tau_0 y_i ((1-\upsilon) \bar{T}_i + \upsilon \bar T_i') + \mathbf{y}^{\top} \bar \bX\mathbf{\theta_0} \Big) \prod_{j=1}^{n} \mathrm{d} \mu(y_j). \nonumber   
\end{align}

This implies, 
\begin{align}
    \log Z_n (\bar \bT,\bar \bX) - \log Z_n(\bar \bT^{(i)}, \bar \bX) = \int_{0}^{1} \frac{\partial}{\partial \upsilon} H(\upsilon) \mathrm{d}\upsilon. \nonumber   
\end{align}
By direct computation, we obtain that for any $\upsilon \in [0,1]$, $|\frac{\partial}{\partial \upsilon} H(\upsilon)| \leq \tau_0$. In turn, this directly implies 
\begin{align}
    \sum_{i=1}^{n} \mathbb{E}_{\bar\bT, \bar \bX}(\log Z_n (\bar \bT, \bar \bX) - \log Z_n(\bar \bT^{(i)} , \bar \bX))^2 \leq n \tau_0^2. 
\end{align}
This yields the conclusion for Algorithm \ref{algo:de_method}.
The proof for Part (ii) follows along similar lines---we only highlight the differences here. The proof of Theorem~\ref{sec:proof_amp} establishes that $\frac{1}{n} \mathbb{E}[\| \langle \by \rangle - m^{[M]} \|^2] \to 0$ as $n \to \infty$, where $\mathbb{E}[\cdot]$ denotes the expectation with respect to the joint distribution $(\bA_n, \bar \bT, \bar \bX)$. Thus it suffices to prove, for any $\varepsilon >0$, 
\begin{align*}
    \mathbb{P}\Big[  \Big| \frac{2}{n}\sum_{i=1}^{n} \bar T_i \langle y_i \rangle - \mathrm{DE} \Big| > \varepsilon  \Big] \to 0 
\end{align*}
as $n \to \infty$. The rest of the proof follows the same argument as Part (i) and uses that $\|\text{Cov}(\by)\|= O(1)$ uniformly over $\bt,\bX$ for $\bA_n= \beta \mathbf{G}_n$ by \cite[Theorem 1]{bauerschmidt2019very}. This completes the proof.




\end{proof}

\begin{proof}[Proof of Lemma \ref{lem:MC_CI}]
    $\mu$ satisfies the tilted LSI property; this implies,  by Proposition \ref{prop:lsi_concentration} that for all $j\in [k]$, $  \hat{\mathrm{DE}}^{(j)}_M - \mathrm{DE}  = o_{P}(1)$ as $M,n \rightarrow \infty$. This implies that a fattened version of the confidence interval contains the true treatment effect with high-probability. The proof for $\text{IE}$ follows by the same argument.  
\end{proof}

\bibliography{ref}

\begin{thebibliography}{92}
\providecommand{\natexlab}[1]{#1}
\providecommand{\url}[1]{\texttt{#1}}
\expandafter\ifx\csname urlstyle\endcsname\relax
  \providecommand{\doi}[1]{doi: #1}\else
  \providecommand{\doi}{doi: \begingroup \urlstyle{rm}\Url}\fi

\bibitem[Akiyama et~al.(2019)Akiyama, Kuramashi, Yamashita, and
  Yoshimura]{akiyama2019phase}
Shinichiro Akiyama, Yoshinobu Kuramashi, Takumi Yamashita, and Yusuke
  Yoshimura.
\newblock Phase transition of four-dimensional ising model with higher-order
  tensor renormalization group.
\newblock \emph{Physical review D}, 100\penalty0 (5):\penalty0 054510, 2019.

\bibitem[Angrist(2014)]{angrist2014perils}
Joshua~D Angrist.
\newblock The perils of peer effects.
\newblock \emph{Labour Economics}, 30:\penalty0 98--108, 2014.

\bibitem[Aronow and Samii(2017)]{aronow2017estimating}
Peter~M. Aronow and Cyrus Samii.
\newblock Estimating average causal effects under general interference, with
  application to a social network experiment.
\newblock \emph{The Annals of Applied Statistics}, 11\penalty0 (4), 2017.

\bibitem[Athey et~al.(2018)Athey, Eckles, and Imbens]{athey2018exact}
Susan Athey, Dean Eckles, and Guido~W Imbens.
\newblock Exact p-values for network interference.
\newblock \emph{Journal of the American Statistical Association}, 113\penalty0
  (521):\penalty0 230--240, 2018.

\bibitem[Bai and Silverstein(2010)]{bai2010spectral}
Zhidong Bai and Jack~W Silverstein.
\newblock \emph{Spectral analysis of large dimensional random matrices},
  volume~20.
\newblock Springer, 2010.

\bibitem[Basak and Mukherjee(2017)]{basak2017universality}
Anirban Basak and Sumit Mukherjee.
\newblock Universality of the mean-field for the potts model.
\newblock \emph{Probability Theory and Related Fields}, 168:\penalty0 557--600,
  2017.

\bibitem[Basse et~al.(2019)Basse, Feller, and Toulis]{basse2019randomization}
Guillaume~W Basse, Avi Feller, and Panos Toulis.
\newblock Randomization tests of causal effects under interference.
\newblock \emph{Biometrika}, 106\penalty0 (2):\penalty0 487--494, 2019.

\bibitem[Bauerschmidt and Bodineau(2019)]{bauerschmidt2019very}
Roland Bauerschmidt and Thierry Bodineau.
\newblock A very simple proof of the lsi for high temperature spin systems.
\newblock \emph{Journal of Functional Analysis}, 276\penalty0 (8):\penalty0
  2582--2588, 2019.

\bibitem[Bayati et~al.(2015)Bayati, Lelarge, and
  Montanari]{bayati2015universality}
Mohsen Bayati, Marc Lelarge, and Andrea Montanari.
\newblock Universality in polytope phase transitions and message passing
  algorithms.
\newblock \emph{Annals of applied probability}, 25\penalty0 (2):\penalty0
  753--822, 2015.

\bibitem[Benaych-Georges et~al.(2019)Benaych-Georges, Bordenave, and
  Knowles]{bbk}
Florent Benaych-Georges, Charles Bordenave, and Antti Knowles.
\newblock Largest eigenvalues of sparse inhomogeneous erd{\H o}s--r{\'e}nyi
  graphs.
\newblock \emph{Annals of Probability}, 47\penalty0 (3):\penalty0 1653--1676,
  05 2019.
\newblock \doi{10.1214/18-AOP1293}.
\newblock URL \url{https://doi.org/10.1214/18-AOP1293}.

\bibitem[Besag(1975)]{besag1975statistical}
Julian Besag.
\newblock Statistical analysis of non-lattice data.
\newblock \emph{Journal of the Royal Statistical Society Series D: The
  Statistician}, 24\penalty0 (3):\penalty0 179--195, 1975.

\bibitem[Bhattacharya et~al.(2020)Bhattacharya, Malinsky, and
  Shpitser]{bhattacharya2020causal}
Rohit Bhattacharya, Daniel Malinsky, and Ilya Shpitser.
\newblock Causal inference under interference and network uncertainty.
\newblock In \emph{Uncertainty in Artificial Intelligence}, pages 1028--1038.
  PMLR, 2020.

\bibitem[Bhattacharya et~al.(2021)Bhattacharya, Mukherjee, and
  Ray]{bhattacharya2021sharp}
Sohom Bhattacharya, Rajarshi Mukherjee, and Gourab Ray.
\newblock Sharp signal detection under ferromagnetic ising models.
\newblock \emph{arXiv preprint arXiv:2110.02949}, 2021.

\bibitem[Bhattacharya et~al.(2022)Bhattacharya, Deb, and
  Mukherjee]{bhattacharya2022ldp}
Sohom Bhattacharya, Nabarun Deb, and Sumit Mukherjee.
\newblock Ldp for inhomogeneous u-statistics.
\newblock \emph{arXiv preprint arXiv:2212.03944}, 2022.

\bibitem[Bhattacharya et~al.(2023)Bhattacharya, Deb, and
  Mukherjee]{bhattacharya2023gibbs}
Sohom Bhattacharya, Nabarun Deb, and Sumit Mukherjee.
\newblock Gibbs measures with multilinear forms.
\newblock \emph{arXiv preprint arXiv:2307.14600}, 2023.

\bibitem[Blei et~al.(2017)Blei, Kucukelbir, and McAuliffe]{blei2017variational}
David~M Blei, Alp Kucukelbir, and Jon~D McAuliffe.
\newblock Variational inference: A review for statisticians.
\newblock \emph{Journal of the American statistical Association}, 112\penalty0
  (518):\penalty0 859--877, 2017.

\bibitem[Boucheron et~al.(2013)Boucheron, Lugosi, and
  Massart]{10.1093/acprof:oso/9780199535255.001.0001}
Stéphane Boucheron, Gábor Lugosi, and Pascal Massart.
\newblock \emph{{Concentration Inequalities: A Nonasymptotic Theory of
  Independence}}.
\newblock Oxford University Press, 02 2013.
\newblock ISBN 9780199535255.
\newblock \doi{10.1093/acprof:oso/9780199535255.001.0001}.
\newblock URL \url{https://doi.org/10.1093/acprof:oso/9780199535255.001.0001}.

\bibitem[Bramoull{\'e} et~al.(2009)Bramoull{\'e}, Djebbari, and
  Fortin]{bramoulle2009identification}
Yann Bramoull{\'e}, Habiba Djebbari, and Bernard Fortin.
\newblock Identification of peer effects through social networks.
\newblock \emph{Journal of econometrics}, 150\penalty0 (1):\penalty0 41--55,
  2009.

\bibitem[Chatterjee and Dembo(2016)]{chatterjee2016nonlinear}
Sourav Chatterjee and Amir Dembo.
\newblock Nonlinear large deviations.
\newblock \emph{Advances in Mathematics}, 299:\penalty0 396--450, 2016.

\bibitem[Chen and Lam(2021)]{chen2021universality}
Wei~Kuo Chen and Wai-Kit Lam.
\newblock Universality of approximate message passing algorithms.
\newblock \emph{Electronic Journal of Probability}, 26:\penalty0 36, 2021.

\bibitem[Chen and Tang(2021)]{chen2021convergence}
Wei-Kuo Chen and Si~Tang.
\newblock On convergence of the cavity and bolthausen’s tap iterations to the
  local magnetization.
\newblock \emph{Communications in Mathematical Physics}, 386\penalty0
  (2):\penalty0 1209--1242, 2021.

\bibitem[Chin(2019)]{chin2019regression}
Alex Chin.
\newblock Regression adjustments for estimating the global treatment effect in
  experiments with interference.
\newblock \emph{Journal of Causal Inference}, 7\penalty0 (2):\penalty0
  20180026, 2019.

\bibitem[Cortez-Rodriguez et~al.(2023)Cortez-Rodriguez, Eichhorn, and
  Yu]{cortez2023exploiting}
Mayleen Cortez-Rodriguez, Matthew Eichhorn, and Christina~Lee Yu.
\newblock Exploiting neighborhood interference with low-order interactions
  under unit randomized design.
\newblock \emph{Journal of Causal Inference}, 11\penalty0 (1):\penalty0
  20220051, 2023.

\bibitem[Daskalakis et~al.(2019)Daskalakis, Dikkala, and
  Panageas]{daskalakis2019regression}
Constantinos Daskalakis, Nishanth Dikkala, and Ioannis Panageas.
\newblock Regression from dependent observations.
\newblock In \emph{Proceedings of the 51st Annual ACM SIGACT Symposium on
  Theory of Computing}, pages 881--889, 2019.

\bibitem[Dudeja et~al.(2023)Dudeja, M.~Lu, and Sen]{dudeja2023universality}
Rishabh Dudeja, Yue M.~Lu, and Subhabrata Sen.
\newblock Universality of approximate message passing with semirandom matrices.
\newblock \emph{The Annals of Probability}, 51\penalty0 (5):\penalty0
  1616--1683, 2023.

\bibitem[Eckles et~al.(2016)Eckles, Karrer, and Ugander]{eckles2016design}
Dean Eckles, Brian Karrer, and Johan Ugander.
\newblock Design and analysis of experiments in networks: Reducing bias from
  interference.
\newblock \emph{Journal of Causal Inference}, 5\penalty0 (1):\penalty0
  20150021, 2016.

\bibitem[Eichhorn et~al.(2024)Eichhorn, Khan, Ugander, and Yu]{eichhorn2024low}
Matthew Eichhorn, Samir Khan, Johan Ugander, and Christina~Lee Yu.
\newblock Low-order outcomes and clustered designs: combining design and
  analysis for causal inference under network interference.
\newblock \emph{arXiv preprint arXiv:2405.07979}, 2024.

\bibitem[Eldan(2018)]{eldan2018gaussian}
Ronen Eldan.
\newblock Gaussian-width gradient complexity, reverse log-sobolev inequalities
  and nonlinear large deviations.
\newblock \emph{Geometric and Functional Analysis}, 28\penalty0 (6):\penalty0
  1548--1596, 2018.

\bibitem[Eldan(2020)]{eldan2020taming}
Ronen Eldan.
\newblock Taming correlations through entropy-efficient measure decompositions
  with applications to mean-field approximation.
\newblock \emph{Probability Theory and Related Fields}, 176\penalty0
  (3):\penalty0 737--755, 2020.

\bibitem[Emmenegger et~al.(2022)Emmenegger, Spohn, Elmer, and
  B{\"u}hlmann]{emmenegger2022treatment}
Corinne Emmenegger, Meta-Lina Spohn, Timon Elmer, and Peter B{\"u}hlmann.
\newblock Treatment effect estimation from observational network data using
  augmented inverse probability weighting and machine learning.
\newblock \emph{arXiv preprint arXiv:2206.14591}, 2022.

\bibitem[Ferracci et~al.(2014)Ferracci, Jolivet, and van~den
  Berg]{ferracci2014evidence}
Marc Ferracci, Gr{\'e}gory Jolivet, and Gerard~J van~den Berg.
\newblock Evidence of treatment spillovers within markets.
\newblock \emph{Review of Economics and Statistics}, 96\penalty0 (5):\penalty0
  812--823, 2014.

\bibitem[Forastiere et~al.(2021)Forastiere, Airoldi, and
  Mealli]{forastiere2021identification}
Laura Forastiere, Edoardo~M Airoldi, and Fabrizia Mealli.
\newblock Identification and estimation of treatment and interference effects
  in observational studies on networks.
\newblock \emph{Journal of the American Statistical Association}, 116\penalty0
  (534):\penalty0 901--918, 2021.

\bibitem[Friedli and Velenik(2017)]{friedli2017statistical}
Sacha Friedli and Yvan Velenik.
\newblock \emph{Statistical mechanics of lattice systems: a concrete
  mathematical introduction}.
\newblock Cambridge University Press, 2017.

\bibitem[Galanis et~al.(2024)Galanis, Kalavasis, and
  Kandiros]{galanis2024sampling}
Andreas Galanis, Alkis Kalavasis, and Anthimos~Vardis Kandiros.
\newblock On sampling from ising models with spectral constraints.
\newblock \emph{arXiv preprint arXiv:2407.07645}, 2024.

\bibitem[Gamarnik(2013)]{gamarnik2013correlation}
David Gamarnik.
\newblock Correlation decay method for decision, optimization, and inference in
  large-scale networks.
\newblock In \emph{Theory Driven by Influential Applications}, pages 108--121.
  INFORMS, 2013.

\bibitem[Ghang et~al.(2014)Ghang, Martin, and Waruhiu]{ghang2014sharp}
Whan Ghang, Zane Martin, and Steven Waruhiu.
\newblock The sharp log-sobolev inequality on a compact interval.
\newblock \emph{Involve}, 7:\penalty0 181--186, 2014.

\bibitem[Ghosal and Mukherjee(2020)]{ghosal2020joint}
Promit Ghosal and Sumit Mukherjee.
\newblock Joint estimation of parameters in ising model.
\newblock \emph{The Annals of Statistics}, 48\penalty0 (2), 2020.

\bibitem[Goldsmith-Pinkham and Imbens(2013)]{goldsmith2013social}
Paul Goldsmith-Pinkham and Guido~W Imbens.
\newblock Social networks and the identification of peer effects.
\newblock \emph{Journal of Business \& Economic Statistics}, 31\penalty0
  (3):\penalty0 253--264, 2013.

\bibitem[Graham(2008)]{graham2008identifying}
Bryan~S Graham.
\newblock Identifying social interactions through conditional variance
  restrictions.
\newblock \emph{Econometrica}, 76\penalty0 (3):\penalty0 643--660, 2008.

\bibitem[Hayes and Moulton(2017)]{hayes2017cluster}
Richard~J Hayes and Lawrence~H Moulton.
\newblock \emph{Cluster randomised trials}.
\newblock Chapman and Hall/CRC, 2017.

\bibitem[Hong and Raudenbush(2006)]{hong2006evaluating}
Guanglei Hong and Stephen~W Raudenbush.
\newblock Evaluating kindergarten retention policy: A case study of causal
  inference for multilevel observational data.
\newblock \emph{Journal of the American Statistical Association}, 101\penalty0
  (475):\penalty0 901--910, 2006.

\bibitem[Hopkins and Steurer(2017)]{hopkins2017efficient}
Samuel~B Hopkins and David Steurer.
\newblock Efficient bayesian estimation from few samples: community detection
  and related problems.
\newblock In \emph{2017 IEEE 58th Annual Symposium on Foundations of Computer
  Science (FOCS)}, pages 379--390. IEEE, 2017.

\bibitem[Hudgens and Halloran(2008)]{hudgens2008toward}
Michael~G Hudgens and M~Elizabeth Halloran.
\newblock Toward causal inference with interference.
\newblock \emph{Journal of the American Statistical Association}, 103\penalty0
  (482):\penalty0 832--842, 2008.

\bibitem[Jagadeesan et~al.(2020)Jagadeesan, Pillai, and
  Volfovsky]{jagadeesan2020designs}
Ravi Jagadeesan, Natesh~S Pillai, and Alexander Volfovsky.
\newblock Designs for estimating the treatment effect in networks with
  interference.
\newblock \emph{The Annals of Statistics}, 2020.

\bibitem[Jain et~al.(2018)Jain, Koehler, and Mossel]{jain2018mean}
Vishesh Jain, Frederic Koehler, and Elchanan Mossel.
\newblock The mean-field approximation: Information inequalities, algorithms,
  and complexity.
\newblock In \emph{Conference On Learning Theory}, pages 1326--1347. PMLR,
  2018.

\bibitem[Jain et~al.(2019)Jain, Koehler, and Risteski]{jain2019mean}
Vishesh Jain, Frederic Koehler, and Andrej Risteski.
\newblock Mean-field approximation, convex hierarchies, and the optimality of
  correlation rounding: a unified perspective.
\newblock In \emph{Proceedings of the 51st Annual ACM SIGACT Symposium on
  Theory of Computing}, pages 1226--1236, 2019.

\bibitem[Jiang et~al.(2022)Jiang, Mukherjee, Sen, and Sur]{jiang2022new}
Kuanhao Jiang, Rajarshi Mukherjee, Subhabrata Sen, and Pragya Sur.
\newblock A new central limit theorem for the augmented ipw estimator: Variance
  inflation, cross-fit covariance and beyond.
\newblock \emph{arXiv preprint arXiv:2205.10198}, 2022.

\bibitem[Kang and Imbens(2016)]{kang2016peer}
Hyunseung Kang and Guido Imbens.
\newblock Peer encouragement designs in causal inference with partial
  interference and identification of local average network effects.
\newblock \emph{arXiv preprint arXiv:1609.04464}, 2016.

\bibitem[Krivelevich and Sudakov(2003)]{krivelevich2003largest}
Michael Krivelevich and Benny Sudakov.
\newblock The largest eigenvalue of sparse random graphs.
\newblock \emph{Combinatorics, Probability and Computing}, 12\penalty0
  (1):\penalty0 61--72, 2003.

\bibitem[Kunisky(2024)]{kunisky2024optimality}
Dmitriy Kunisky.
\newblock Optimality of glauber dynamics for general-purpose ising model
  sampling and free energy approximation.
\newblock In \emph{Proceedings of the 2024 Annual ACM-SIAM Symposium on
  Discrete Algorithms (SODA)}, pages 5013--5028. SIAM, 2024.

\bibitem[Lata{\l}a and Oleszkiewicz(2000)]{latala2000between}
Rafa{\l} Lata{\l}a and Krzysztof Oleszkiewicz.
\newblock Between sobolev and poincar{\'e}.
\newblock In \emph{Geometric Aspects of Functional Analysis: Israel Seminar
  1996--2000}, pages 147--168. Springer, 2000.

\bibitem[Lauritzen and Richardson(2002)]{lauritzen2002chain}
Steffen~L Lauritzen and Thomas~S Richardson.
\newblock Chain graph models and their causal interpretations.
\newblock \emph{Journal of the Royal Statistical Society Series B: Statistical
  Methodology}, 64\penalty0 (3):\penalty0 321--348, 2002.

\bibitem[Lee(2007)]{lee2007identification}
Lung-Fei Lee.
\newblock Identification and estimation of econometric models with group
  interactions, contextual factors and fixed effects.
\newblock \emph{Journal of econometrics}, 140\penalty0 (2):\penalty0 333--374,
  2007.

\bibitem[Leung(2022)]{leung2022rate}
Michael~P Leung.
\newblock Rate-optimal cluster-randomized designs for spatial interference.
\newblock \emph{The Annals of Statistics}, 50\penalty0 (5):\penalty0
  3064--3087, 2022.

\bibitem[Li and Wager(2022)]{li2022random}
Shuangning Li and Stefan Wager.
\newblock Random graph asymptotics for treatment effect estimation under
  network interference.
\newblock \emph{The Annals of Statistics}, 50\penalty0 (4):\penalty0
  2334--2358, 2022.

\bibitem[Liu and Hudgens(2014)]{liu2014large}
Lan Liu and Michael~G Hudgens.
\newblock Large sample randomization inference of causal effects in the
  presence of interference.
\newblock \emph{Journal of the american statistical association}, 109\penalty0
  (505):\penalty0 288--301, 2014.

\bibitem[Liu et~al.(2024)Liu, Mukherjee, and Biswas]{liu2024tensor}
Tianyu Liu, Somabha Mukherjee, and Rahul Biswas.
\newblock Tensor recovery in high-dimensional ising models.
\newblock \emph{Journal of Multivariate Analysis}, page 105335, 2024.

\bibitem[Lundin and Karlsson(2014)]{lundin2014estimation}
Mathias Lundin and Maria Karlsson.
\newblock Estimation of causal effects in observational studies with
  interference between units.
\newblock \emph{Statistical Methods \& Applications}, 23:\penalty0 417--433,
  2014.

\bibitem[Manski(1993)]{manski1993identification}
Charles~F Manski.
\newblock Identification of endogenous social effects: The reflection problem.
\newblock \emph{The review of economic studies}, 60\penalty0 (3):\penalty0
  531--542, 1993.

\bibitem[Manski(2013)]{manski2013identification}
Charles~F Manski.
\newblock Identification of treatment response with social interactions.
\newblock \emph{The Econometrics Journal}, 16\penalty0 (1):\penalty0 S1--S23,
  2013.

\bibitem[Mezard and Montanari(2009)]{mezard2009information}
Marc Mezard and Andrea Montanari.
\newblock \emph{Information, physics, and computation}.
\newblock Oxford University Press, 2009.

\bibitem[Mukherjee et~al.(2021)Mukherjee, Niu, Halder, Bhattacharya, and
  Michailidis]{mukherjee2021high}
Somabha Mukherjee, Ziang Niu, Sagnik Halder, Bhaswar~B Bhattacharya, and George
  Michailidis.
\newblock High dimensional logistic regression under network dependence.
\newblock \emph{arXiv preprint arXiv:2110.03200}, 2021.

\bibitem[Mukherjee et~al.(2022)Mukherjee, Son, and
  Bhattacharya]{mukherjee2022estimation}
Somabha Mukherjee, Jaesung Son, and Bhaswar~B Bhattacharya.
\newblock Estimation in tensor ising models.
\newblock \emph{Information and Inference: A Journal of the IMA}, 11\penalty0
  (4):\penalty0 1457--1500, 2022.

\bibitem[Mukherjee and Sen(2022)]{mukherjee2022variational}
Sumit Mukherjee and Subhabrata Sen.
\newblock Variational inference in high-dimensional linear regression.
\newblock \emph{The Journal of Machine Learning Research}, 23\penalty0
  (1):\penalty0 13703--13758, 2022.

\bibitem[Ogburn et~al.(2020)Ogburn, Shpitser, and Lee]{ogburn2020causal}
Elizabeth~L Ogburn, Ilya Shpitser, and Youjin Lee.
\newblock Causal inference, social networks and chain graphs.
\newblock \emph{Journal of the Royal Statistical Society Series A: Statistics
  in Society}, 183\penalty0 (4):\penalty0 1659--1676, 2020.

\bibitem[Ogburn et~al.(2024)Ogburn, Sofrygin, Diaz, and Van~der
  Laan]{ogburn2024causal}
Elizabeth~L Ogburn, Oleg Sofrygin, Ivan Diaz, and Mark~J Van~der Laan.
\newblock Causal inference for social network data.
\newblock \emph{Journal of the American Statistical Association}, 119\penalty0
  (545):\penalty0 597--611, 2024.

\bibitem[Park and Kang(2023)]{park2023assumption}
Chan Park and Hyunseung Kang.
\newblock Assumption-lean analysis of cluster randomized trials in infectious
  diseases for intent-to-treat effects and network effects.
\newblock \emph{Journal of the American Statistical Association}, 118\penalty0
  (542):\penalty0 1195--1206, 2023.

\bibitem[Robins(1986)]{robins1986new}
James Robins.
\newblock A new approach to causal inference in mortality studies with a
  sustained exposure period—application to control of the healthy worker
  survivor effect.
\newblock \emph{Mathematical modelling}, 7\penalty0 (9-12):\penalty0
  1393--1512, 1986.

\bibitem[Robins and Rotnitzky(1992)]{robins1992recovery}
James~M Robins and Andrea Rotnitzky.
\newblock Recovery of information and adjustment for dependent censoring using
  surrogate markers.
\newblock In \emph{AIDS epidemiology: methodological issues}, pages 297--331.
  Springer, 1992.

\bibitem[Robins et~al.(1994)Robins, Rotnitzky, and Zhao]{robins1994estimation}
James~M Robins, Andrea Rotnitzky, and Lue~Ping Zhao.
\newblock Estimation of regression coefficients when some regressors are not
  always observed.
\newblock \emph{Journal of the American statistical Association}, 89\penalty0
  (427):\penalty0 846--866, 1994.

\bibitem[Rosenbaum(2007)]{rosenbaum2007interference}
Paul~R Rosenbaum.
\newblock Interference between units in randomized experiments.
\newblock \emph{Journal of the american statistical association}, 102\penalty0
  (477):\penalty0 191--200, 2007.

\bibitem[Rubin(1990)]{rubin1990formal}
Donald~B Rubin.
\newblock Formal mode of statistical inference for causal effects.
\newblock \emph{Journal of statistical planning and inference}, 25\penalty0
  (3):\penalty0 279--292, 1990.

\bibitem[Sasakura and Sato(2014)]{sasakura2014ising}
Naoki Sasakura and Yuki Sato.
\newblock Ising model on random networks and the canonical tensor model.
\newblock \emph{Progress of Theoretical and Experimental Physics},
  2014\penalty0 (5):\penalty0 053B03, 2014.

\bibitem[S{\"a}vje et~al.(2021)S{\"a}vje, Aronow, and
  Hudgens]{savje2021average}
Fredrik S{\"a}vje, Peter Aronow, and Michael Hudgens.
\newblock Average treatment effects in the presence of unknown interference.
\newblock \emph{Annals of statistics}, 49\penalty0 (2):\penalty0 673, 2021.

\bibitem[Sherman and Shpitser(2018)]{sherman2018identification}
Eli Sherman and Ilya Shpitser.
\newblock Identification and estimation of causal effects from dependent data.
\newblock \emph{Advances in neural information processing systems}, 31, 2018.

\bibitem[Shirani and Bayati(2023)]{shirani2023causal}
Sadegh Shirani and Mohsen Bayati.
\newblock Causal message passing: A method for experiments with unknown and
  general network interference.
\newblock \emph{arXiv preprint arXiv:2311.08340}, 2023.

\bibitem[Shpitser et~al.(2017)Shpitser, Tchetgen, and
  Andrews]{shpitser2017modeling}
Ilya Shpitser, Eric~Tchetgen Tchetgen, and Ryan Andrews.
\newblock Modeling interference via symmetric treatment decomposition.
\newblock \emph{arXiv preprint arXiv:1709.01050}, 2017.

\bibitem[Sobel(2006)]{sobel2006randomized}
Michael~E Sobel.
\newblock What do randomized studies of housing mobility demonstrate? causal
  inference in the face of interference.
\newblock \emph{Journal of the American Statistical Association}, 101\penalty0
  (476):\penalty0 1398--1407, 2006.

\bibitem[Sofrygin and van~der Laan(2017)]{sofrygin2017semi}
Oleg Sofrygin and Mark~J van~der Laan.
\newblock Semi-parametric estimation and inference for the mean outcome of the
  single time-point intervention in a causally connected population.
\newblock \emph{Journal of causal inference}, 5\penalty0 (1):\penalty0
  20160003, 2017.

\bibitem[Talagrand(2010)]{talagrand2010mean}
Michel Talagrand.
\newblock \emph{Mean field models for spin glasses: Volume I: Basic examples},
  volume~54.
\newblock Springer Science \& Business Media, 2010.

\bibitem[Tchetgen and VanderWeele(2012)]{tchetgen2012causal}
Eric J~Tchetgen Tchetgen and Tyler~J VanderWeele.
\newblock On causal inference in the presence of interference.
\newblock \emph{Statistical methods in medical research}, 21\penalty0
  (1):\penalty0 55--75, 2012.

\bibitem[Tchetgen~Tchetgen et~al.(2021)Tchetgen~Tchetgen, Fulcher, and
  Shpitser]{tchetgen2021auto}
Eric~J Tchetgen~Tchetgen, Isabel~R Fulcher, and Ilya Shpitser.
\newblock Auto-g-computation of causal effects on a network.
\newblock \emph{Journal of the American Statistical Association}, 116\penalty0
  (534):\penalty0 833--844, 2021.

\bibitem[Toulis and Kao(2013)]{toulis2013estimation}
Panos Toulis and Edward Kao.
\newblock Estimation of causal peer influence effects.
\newblock In \emph{International conference on machine learning}, pages
  1489--1497. PMLR, 2013.

\bibitem[Ugander et~al.(2013)Ugander, Karrer, Backstrom, and
  Kleinberg]{ugander2013graph}
Johan Ugander, Brian Karrer, Lars Backstrom, and Jon Kleinberg.
\newblock Graph cluster randomization: Network exposure to multiple universes.
\newblock In \emph{Proceedings of the 19th ACM SIGKDD international conference
  on Knowledge discovery and data mining}, pages 329--337, 2013.

\bibitem[Van~der Laan(2014)]{van2014causal}
Mark~J Van~der Laan.
\newblock Causal inference for a population of causally connected units.
\newblock \emph{Journal of Causal Inference}, 2\penalty0 (1):\penalty0 13--74,
  2014.

\bibitem[Van Der~Laan and Rubin(2006)]{van2006targeted}
Mark~J Van Der~Laan and Daniel Rubin.
\newblock Targeted maximum likelihood learning.
\newblock \emph{The international journal of biostatistics}, 2\penalty0 (1),
  2006.

\bibitem[VanderWeele(2010)]{vanderweele2010direct}
Tyler~J VanderWeele.
\newblock Direct and indirect effects for neighborhood-based clustered and
  longitudinal data.
\newblock \emph{Sociological methods \& research}, 38\penalty0 (4):\penalty0
  515--544, 2010.

\bibitem[Viviano et~al.(2023)Viviano, Lei, Imbens, Karrer, Schrijvers, and
  Shi]{viviano2023causal}
Davide Viviano, Lihua Lei, Guido Imbens, Brian Karrer, Okke Schrijvers, and
  Liang Shi.
\newblock Causal clustering: design of cluster experiments under network
  interference.
\newblock \emph{arXiv preprint arXiv:2310.14983}, 2023.

\bibitem[Wainwright et~al.(2008)Wainwright, Jordan,
  et~al.]{wainwright2008graphical}
Martin~J Wainwright, Michael~I Jordan, et~al.
\newblock Graphical models, exponential families, and variational inference.
\newblock \emph{Foundations and Trends{\textregistered} in Machine Learning},
  1\penalty0 (1--2):\penalty0 1--305, 2008.

\bibitem[Wang et~al.(2022)Wang, Zhong, and Fan]{wang2022universality}
Tianhao Wang, Xinyi Zhong, and Zhou Fan.
\newblock Universality of approximate message passing algorithms and tensor
  networks.
\newblock \emph{arXiv preprint arXiv:2206.13037}, 2022.

\bibitem[Yan(2020)]{yan2020nonlinear}
Jun Yan.
\newblock {Nonlinear large deviations: Beyond the hypercube}.
\newblock \emph{The Annals of Applied Probability}, 30\penalty0 (2):\penalty0
  812 -- 846, 2020.

\bibitem[Yu et~al.(2022)Yu, Airoldi, Borgs, and Chayes]{yu2022estimating}
Christina~Lee Yu, Edoardo~M Airoldi, Christian Borgs, and Jennifer~T Chayes.
\newblock Estimating the total treatment effect in randomized experiments with
  unknown network structure.
\newblock \emph{Proceedings of the National Academy of Sciences}, 119\penalty0
  (44):\penalty0 e2208975119, 2022.

\end{thebibliography}

\appendix

\section{Results on SK model on soft spins with random external fields} \label{appendix: sk_model}

We collect here some results on the Sherrington-Kirkpatrick (SK) model with soft spins and random external fields. Formally, we fix a probability measure $\mu$ on $[-1,1]$ and consider the probability distribution 
\begin{align}
    \frac{\mathrm{d}\nu}{\mathrm{d}\mu^{\otimes n}}(\mathbf{y}) = \frac{1}{Z} \exp\Big( \frac{\beta}{\sqrt{n}}\sum_{i<j\leq n} g_{ij} y_i y_j + \sum_{i=1}^{n} h_i y_i \Big), \label{eq:SK_model} 
\end{align}
where $\{g_{ij}: i <j \} \sim \mathcal{N}(0,1)$ are i.i.d., $h_i$ are random external fields independent of $\{g_{ij}: i <j \}$ and $Z$ denotes the normalizing constant. Throughout, we use $h$ to denote the distribution of the external fields $h_i$. The traditional SK model is supported on $\{\pm 1\}^n$; in contrast, the spins $y_i \in [-1,1]$ in the model above. This setting is typically referred to as the ``soft spin" setting in spin-glass literature. In this appendix, we will extend some well-known results for the SK model at high-temperature to the SK model with soft spins and random external fields. We refer the interested reader to \cite[Chapter 1]{talagrand2010mean} for the original results on the SK model at high-temperature. 

To this end, we first introduce the following fixed-point system, which will be critical for the subsequent discussion. Recall the log-moment generating function $\alpha(\cdot)$ from Definition~\ref{def:exp_tilt}. 

\begin{defn}[Fixed points]
\label{defn:fixed_points_sk}
Let $(q, \sigma^2) \in [0,1] \times [0, \infty)$ be defined via the following fixed point system: 
\begin{align}
    \mathbb{E}[\alpha''(h + \beta \sqrt{q}Z, \beta^2 \sigma^2)] = \sigma^2,  \nonumber \\
    \mathbb{E}[(\alpha'(h+ \beta \sqrt{q} Z, \beta^2 \sigma^2))^2] = q.  \label{eq:fixed_pt}
\end{align}
The random variable $Z \sim \mathcal{N}(0,1)$ in the display above.    
\end{defn}

\begin{lemma}\label{lem:fixed_point_unique}
    There exists $\beta_0>0$ such that for $\beta < \beta_0$, the fixed point system \eqref{eq:fixed_pt} has a unique solution $(\sigma_*^2, q_*)$. 
\end{lemma}

\begin{proof}
   We will prove the following stronger claim: Define the function $$\psi:(\sigma^2,q) \mapsto (\mathbb{E}[\alpha''(h + \beta \sqrt{q}Z, \beta^2 \sigma^2)],\mathbb{E}[(\alpha'(h+ \beta \sqrt{q} Z, \beta^2 \sigma^2))^2].$$ There exists $\beta_0>0$ such that for $\beta < \beta_0$, there exists $C<1$ such that for any two pairs $(\sigma^2_0,q_0)$ and $(\sigma^2_1,q_1)$, we have 
\begin{equation}\label{eq:psi_lipschitz}
        \|\psi(\sigma^2_0,q_0)-\psi(\sigma^2_1,q_1)\|^2 \le C \|(\sigma^2_0,q_0)- (\sigma^2_1,q_1)\|^2.
    \end{equation}
The proof of existence and  uniqueness of the fixed point follows immediately from \eqref{eq:psi_lipschitz}. To show \eqref{eq:psi_lipschitz}, we introduce
\begin{align*}
    &\sigma^2(t)= t \sigma^2_1+ (1-t) \sigma^2_0; \quad q(t)= t q_1+ (1-t) q_0; \\
    &H(t):= \psi(\sigma^2(t),q(t))=:(\psi_1(\sigma^2(t),q(t)),\psi_2(\sigma^2(t),q(t))). 
\end{align*}
Note that $\sigma^2(i)= \sigma^2_i$, $q(i)=q_i$ for $i=0,1$. Since
\begin{equation*}
\|H(1)-H(0)\|^2 \le (\|\|\nabla \psi_1\|^2\|_{\infty} + \|\|\nabla \psi_2\|^2\|_{\infty})\|(\sigma^2_0,q_0)- (\sigma^2_1,q_1)\|^2,
\end{equation*}
by definition, we only need to provide upper bounds on $\|\|\nabla \psi_i\|^2\|_{\infty}$, $i=1,2$. To this end, note that
\begin{align*}
    \frac{\partial}{\partial q} \psi_1(\sigma^2,q) &= \frac{\beta}{2 \sqrt{q}} \ee  \Bigg[ Z  \alpha^{(3)}(h+\beta \sqrt{q}Z, \beta^2 \sigma^2) \Bigg]\\
    &= \frac{\beta^2}{2} \ee \Bigg[\alpha ^{(4)}(h+\beta \sqrt{q}Z, \beta^2 \sigma^2)\Bigg],
\end{align*}
where the second equality is by Gaussian integration by parts. Also,  
\begin{align*}
    \frac{\partial}{\partial \sigma^2} \psi_1(\sigma^2,q)= \beta^2 \ee \left[  \frac{\partial}{\partial y} \alpha''(x,y)\big|_{h+\beta \sqrt{q}Z, \beta^2 \sigma^2}\right]
\end{align*}
Since by the definition of exponential family, all the partial derivatives of $\alpha$ are bounded, we have by the above two displays that $\|\|\nabla \psi_1\|^2\|_{\infty} \lesssim \beta^2$. Similarly, one obtains that, 
\begin{align*}
    &\frac{\partial}{\partial q} \psi_2(\sigma^2,q) = \frac{\beta}{\sqrt{q}} \ee  \Bigg[ Z  \alpha'(h+\beta \sqrt{q}Z, \beta^2 \sigma^2)\alpha^{(2)}(h+\beta \sqrt{q}Z, \beta^2 \sigma^2) \Bigg]\\
    \implies &\left|\frac{\partial}{\partial q} \psi_2(\sigma^2,q)\right| \le \beta^2 \ee \Bigg[\|(\alpha^{(2)})^2\|_{\infty}+ \|\alpha' \alpha^{(3)}\|_\infty\Bigg]
\end{align*}
using Gaussian integration by parts. Also,
\begin{align*}
    \left|\frac{\partial}{\partial \sigma^2} \psi_2(\sigma^2,q)\right| \le 2 \beta^2 \ee \Bigg[ \|\alpha'\|_{\infty} \|\frac{\partial}{\partial y} \alpha'(x,y) \|_\infty \Bigg],
\end{align*}
implying again $\|\|\nabla \psi_2\|^2_2\|_{\infty} \lesssim \beta^2$. Hence there exists $\beta_0$ such that for $\beta< \beta_0$, one obtains \eqref{eq:psi_lipschitz}, completing the proof.
\end{proof}

The following result shows stability of the fixed points \eqref{eq:fixed_pt} w.r.to the parameter $h$.

\begin{lemma}\label{lem:fixed_point_h_stability}
    Let $(\sigma^2_i, q_i)$, $i=0,1$ be the unique solutions of the fixed point equations \eqref{eq:fixed_pt} for $h=h_0$ and $h_1$ respectively, where $|h_0-h_1| <\varepsilon$ for some deterministic $\varepsilon>0$. Then there exists $\beta_0$ such that for any $\beta < \beta_0$,
    \begin{equation*}
        \|(\sigma^2_1, q_1)- (\sigma^2_0, q_0)\|^2 \lesssim \varepsilon^2.
    \end{equation*}
\end{lemma}

\begin{proof}
Define $(\sigma^2_{i,0}, q_{i,0})=(1,0)$ for $i=0,1$ and $\psi^{i}=(\psi^{i}_1,\psi^{i}_2)$, where $$\psi^{i}:(\sigma^2,q) \mapsto (\mathbb{E}[\alpha''(h_i + \beta \sqrt{q}Z, \beta^2 \sigma^2)],\mathbb{E}[(\alpha'(h_i+ \beta \sqrt{q} Z, \beta^2 \sigma^2))^2], \quad i=0,1$$ 

Define the iterates $(\sigma^2_{i,k+1}, q_{i,k+1})= \psi^i (\sigma^2_{i,k}, q_{i,k})$, $i=0,1$ which converges to the fixed points $(\sigma^2_0, q_0)$, $(\sigma^2_1, q_1)$ respectively. Our objective is to bound $\Delta_{k+1}:=\|(\sigma^2_{1,k+1}, q_{1,k+1}) - (\sigma^2_{0,k+1}, q_{0,k+1})\|^2$ recursively. To this end, by direct differentiation, we  note 

\begin{align*}
     |\psi^1_{1}(\sigma^2_{1,k}, q_{1,k})- \psi^0_{1}(\sigma^2_{0,k}, q_{0,k})|^2 &\lesssim \ee (h_0-h_1)^2 \\
    &+ \beta^2 \Bigg((\sigma^2_{1,k}- \sigma^2_{0,k})^2+ (q_{1,k}- q_{0,k})^2\bigg)\\
    &\leq \varepsilon^2 + \beta^2 \Delta_k,
\end{align*}
and a similar bound also  applies for $\psi^1_{2}- \psi^{0}_2$. Therefore we obtain, for some $C>0$,

\begin{align*}
    \Delta_{k+1} & \le C \varepsilon^2+ C \beta^2 \Delta_{k} \le C \varepsilon^2+ C \beta^2 ( C \varepsilon^2+ C \beta^2 \Delta_{k-1}) \leq  \ldots \\
    & \lesssim  \frac{\varepsilon^2}{1-\beta^2} \lesssim \varepsilon^2,
\end{align*}
if $\beta<1$. Since $\Delta_{k+1}$ converges to $\|(\sigma^2_1, q_1)- (\sigma^2_0, q_0)\|^2 $ as $k \rightarrow \infty$, this completes the proof.
\end{proof}

Our subsequent discussion will focus exclusively on a sufficiently high-temperature regime, and we will assume the existence and uniqueness of the parameters $(\sigma_*^2, q_*)$. For any sample $\mathbf{y} \sim \nu$, we use $R = \frac{1}{n} \sum_{i=1}^{n} y_i^2$ to denote the ``self-overlap". Similarly, for $y^{1}, y^{2}$ i.i.d. samples from $\nu$, we set $R_{1,2} = \frac{1}{n} \sum_{i=1}^{n} y^1_i y^2_i$ to denote the ``overlap". Our first lemma establishes the concentration of the overlap parameters at high-temperature. For any function $f: ([-1,1]^n)^m \to \mathbb{R}$, we use $\nu(f) := \int f(\mathbf{y}^1, \cdots, \mathbf{y}^m) \mathrm{d}\nu(\mathbf{y}^1) \cdots \mathrm{d}\nu  (\mathbf{y}^m)$. 

\begin{thm}
    \label{thm:overlap_conc}
    There exists $\beta_0>0$ such that for $\beta<\beta_0$, 
    \begin{align*}
        \mathbb{E}[\nu((R-(\sigma_*^2 + q_*))^2)] \lesssim \frac{1}{n}, \,\,\,\, \mathbb{E}[\nu((R_{1,2} - q_*)^2)] \lesssim \frac{1}{n}.  
    \end{align*}
    In the display above, $\mathbb{E}[\cdot]$ denotes the expectation with respect to the gaussian variables $\{g_{ij}: i <j\}$ and the random external fields $\{h_i: 1\leq i \leq n\}$. 
\end{thm}

\begin{proof}
    The proof is a straightforward adaptation of \cite[Theorem 1.12.1]{talagrand2010mean}. We sketch the proof, and refer the reader to \cite{talagrand2010mean} for detailed arguments. Let $\mathbf{\rho} = (y_1, \cdots, y_{n-1})$ denote the ``leave-one-out" system which removes the $n^{th}$ spin from the system. Define 
    \begin{align*}
        g(\mathbf{\rho}) = \frac{\beta}{\sqrt{n}} \sum_{i=1}^{n-1} g_{in} y_i, \,\,\,\, g_t(\mathbf{\rho}) = \sqrt{t} g(\mathbf{\rho}) + \sqrt{1-t} Y, 
    \end{align*}
    where $Y \sim \mathcal{N}(0, \beta^2 q_*)$ independent of $\{g_{ij}: i <j\}$ and $\{h_i : 1\leq i \leq n\}$. For $t \in [0,1]$, define the interpolating Hamiltonian 
    \begin{align*}
        H_t(\mathbf{y}) = \frac{\beta}{\sqrt{n}} \sum_{i<j \leq n-1} g_{ij} y_i y_j + y_n g_t(\mathbf{\rho}) + \frac{\beta^2 \sigma_*^2}{2} (1-t) y_n^2. 
    \end{align*}
For $t \in [0,1]$ let $\nu_t(\cdot)$ denote the probability distribution 
\begin{align}
    \frac{\mathrm{d}\nu_t}{\mathrm{d}\mu^{\otimes n}}(\mathbf{y}) = \frac{1}{Z_t} \exp\Big(  H_t(\mathbf{y}) + \sum_{i=1}^{n} h_i y_i\Big). 
\end{align}
Let $\mathbf{y}^{1}, \cdots, \mathbf{y}^{m}$ be i.i.d. replicas from $\nu_t(\cdot)$. For $\ell \neq \ell'$, set 
$R_{\ell, \ell'}^- = \frac{1}{n} \sum_{i=1}^{n-1} y^{\ell}_i y^{\ell'}_i$. Finally, set 
\begin{align}
    q(\ell, \ell') = \begin{cases}
        q_* &\textrm{if}\,\, \ell \neq \ell' , \\
        \sigma_*^2 + q_* &\textrm{o.w.} 
    \end{cases}\nonumber 
\end{align}
Let $f: ([-1,1]^n)^m \to \mathbb{R}$ be a real-valued function of $m$-replicas $\mathbf{y}^1, \cdots, \mathbf{y}^m$. 
Armed with this notation, we note that 
\begin{align}
    \frac{\mathrm{d}}{\mathrm{d}t} \mathbb{E}[\nu_t(f)] =& \frac{\beta^2}{2} \sum_{\ell, \ell' \leq m} \mathbb{E}[\nu_t(f y_n^{\ell} y_n^{\ell'} (R^-_{\ell,\ell'} - q(\ell, \ell') )] \nonumber \\
    &-m\beta^2 \sum_{\ell \leq m} \mathbb{E}[\nu_t(f y^{\ell}_n y^{m+1}_n (R^-_{\ell,m+1} - q(\ell, m+1)))] \nonumber \\
    & - \frac{m \beta^2}{2} \mathbb{E}[\nu_t(f(y_n^{m+1})^2 (R^-_{m+1, m+1} - q(m+1, m+1) )] \nonumber \\
    &+ \frac{m(m+1)}{2} \beta^2 \mathbb{E}[\nu_t(fy_n^{m+1} y_n^{m+2} (R^-_{m+1, m+2} - q(m+1, m+2) )]. \label{eq:derivative} 
\end{align}
The expression for the derivative follows by direct differentiation and a subsequent application of  Gaussian integration by parts. 

We now specialize to the case $m=2$. Note that by Cauchy-Schwarz inequality, 
\begin{align*}
    |\mathbb{E}[\nu_t(fy_n^{\ell} y_n^{\ell'} (R^-_{\ell,\ell'} - q(\ell, \ell'))]| \leq \sqrt{\mathbb{E}[\nu_t(f^2)] \mathbb{E}[\nu_t ((R^-_{\ell,\ell'}- q (\ell,\ell'))^2)]}.  
\end{align*}
Plugging this estimate back into \eqref{eq:derivative}, we have, for some universal constant $L>0$,
\begin{align}
    | \frac{\mathrm{d}}{\mathrm{d}t} \mathbb{E}[\nu_t(f)] | \leq L \beta^2 \sqrt{\mathbb{E}[\nu_t(f^2)]} \Big[ \sqrt{\mathbb{E}[\nu_t((R^-_{1,1} - (\sigma_*^2+q_*))^2)] } + \sqrt{\mathbb{E}[\nu_t((R^-_{1,2} - q_*)^2)]}\Big]. \label{eq:derivative_bound_1}
\end{align}
In addition, using Jensen inequality on \eqref{eq:derivative}, we also have, for some universal constant $L>0$, 
\begin{align}
    |\frac{\mathrm{d}}{\mathrm{d}t} \mathbb{E}[\nu_t(f)] | \leq L \beta^2 \mathbb{E}[\nu_t(|f|)]. \label{eq:derivative_bound_2}
\end{align}

We now turn to the concentration for the overlap and the self-overlap. Define 
\begin{align*}
    A = \mathbb{E}[\nu((R_{1,2}-q_*)^2)],\,\,\, B= \mathbb{E}[\nu((R-(\sigma_*^2 + q_*))^2)].  
\end{align*}
Recall that in our interpolation scheme, $\nu_1 = \nu$. We have, by site symmetry,  
\begin{align*}
    A = \mathbb{E}[\nu(f)], \,\,\, f= (y_n^1 y_n^2 - q_*) (R_{1,2} - q_*). 
\end{align*}
Thus $f^2 \leq 4 (R_{1,2} - q_*)^2$. Crucially, note that 
\begin{align*}
    \mathbb{E}[\nu_0( (y_n^1y_n^2 - q_*) ( R_{1,2}^- - q_*)) ] = 0.  
\end{align*}
In particular, this implies, $|\mathbb{E}[\nu_0(f)]| \lesssim \frac{1}{n}$. Now, we bound the derivative, 
\begin{align}
    \mathbb{E}[\nu(f)] \leq \mathbb{E}[\nu_0(f)] + \sup_{0\leq t \leq 1} | \frac{\mathrm{d}}{\mathrm{d}t} \mathbb{E}[\nu_t(f)] | \nonumber 
\end{align}
which yields the bound 
\begin{align}
    A \leq \frac{K}{n} + L \beta^2 \sqrt{A} (\sqrt{A} + \sqrt{B}) \nonumber 
\end{align}
for some universal constant $K>0$. Swapping the roles of $A$ and $B$ we obtain 
\begin{align}
    B \leq \frac{K}{n} + L \beta^2 \sqrt{B} (\sqrt{A} + \sqrt{B}). \nonumber
\end{align}
Combining, we obtain, 
\begin{align}
    (A+B) \leq \frac{K}{n} + L \beta^2 (A+B). \nonumber 
\end{align}
For $\beta$ small enough, we can invert this inequality to derive the required upper bound. 
\end{proof}

Our next result strengthens Theorem~\ref{thm:overlap_conc} by controlling the higher moments of the overlap and the self-overlap. 

\begin{thm}
\label{thm:overlap_higher_moment}
    There exists $\beta_0>0$ and a universal constant $C_0>0$ such that for all $\beta< \beta_0$ and for all $k \geq 1$, 
    \begin{align}
        \mathbb{E}[\nu((R_{1,2} - q_*)^{2k})] \leq \frac{(C_0 k)^k}{n^k},\,\,\,\, \mathbb{E}[\nu((R- (\sigma^2_* + q_*))^{2k})] \leq \frac{(C_0 k)^k}{n^k}. \nonumber 
    \end{align}
\end{thm}

\begin{proof}
    We will build on the interpolation ideas from the proof of Theorem \ref{thm:overlap_conc}, and thus re-use the notation introduced therein. We proceed by induction, following the proof of \cite[Proposition 1.6.7]{talagrand2010mean}. To this end, define 
    \begin{align}
        A_m = \frac{1}{n} \sum_{i=m}^{n} (y_i^1 y_i^2 - q_*), \,\,\,\, B_m = \frac{1}{n} \sum_{i=m}^{n} \Big( (y_i)^2 - (\sigma_*^2 + q_*) \Big). \nonumber 
    \end{align}
    Note that $A_1 = R_{1,2}$ and $B_1= R$.

As induction hypothesis, we assume that for all $m \leq n$, 
\begin{align}
    \mathbb{E}[\nu(A_m^{2k})] + \mathbb{E}[\nu(B_m^{2k})] \leq  \frac{(C_0 k)^k}{n^k}. \nonumber 
\end{align}
The base case $k=1$ follows from the proof of Theorem \ref{thm:overlap_conc}, once we observe that the same proof applies for all $1 \leq m \leq n$. 

We next sketch the induction step. Note that for $m=n$, $\max\{|A_n|, |B_n|\} \leq 2/n$, and thus the induction hypothesis holds trivially for $C_0$ large enough. Thus we assume that $m<n$ in the subsequent argument. First, we observe that 
\begin{align}
    \mathbb{E}[\nu(A_m^{2k+2})] = \frac{1}{n} \sum_{i=m}^{n} \mathbb{E}[\nu((y^1_i y^2_i - q_*) A_m^{2k+1})] = \frac{n-m+1}{n} \mathbb{E}[\nu((y_n^1 y_n^2 - q_*) A_m^{2k+1})], \nonumber 
\end{align}
where the last display uses site symmetry. Setting $f= (y^1_n y^2_n - q_*) A_m^{2k+1}$, we have, $\mathbb{E}[\nu(A_m^{2k+2})] \leq |\mathbb{E}[\nu(f)]|$. To control, $\mathbb{E}[\nu(f)]$, we again use an interpolation argument. To this end, define 
\begin{align*}
    A'_{m} = \frac{1}{n} \sum_{i=m}^{n-1} (y_i^1 y_i^2 - q_*). \nonumber 
\end{align*}
In turn, this implies 
\begin{align}
    \mathbb{E}[\nu_0((y_n^1y_n^2 - q_*) (A'_m)^{2k+1})]=0. \nonumber 
\end{align}
This immediately implies 
\begin{align*}
    |\mathbb{E}[\nu_0(f)]| = |\mathbb{E}[\nu_0(f)] - \mathbb{E}[\nu_0((y_n^1y_n^2 - q_*) (A'_m)^{2k+1})]| \leq 2 \mathbb{E}[\nu_0(|A_m^{2k+1} - (A_m')^{2k+1}|)]. 
\end{align*}
We now observe that $|A_m - A_m'| \leq 2/n $ and thus there exists a universal constant $C>0$ such that 
\begin{align}
    |\mathbb{E}[\nu_0(f)]| \leq \frac{Ck}{n} [\mathbb{E}[\nu_0(A_m^{2k})] + \mathbb{E}[\nu_0((A'_m)^{2k})]]. \nonumber 
\end{align}
For $\beta < \beta_0$, using \eqref{eq:derivative_bound_2} we further have 
\begin{align*}
   |\mathbb{E}[\nu_0(f)]| \leq \frac{Ck}{n} [\mathbb{E}[\nu(A_m^{2k})] + \mathbb{E}[\nu((A'_m)^{2k})]],   
\end{align*}
where $C>0$ denotes a (possibly larger) universal constant. Further, by site symmetry, $\mathbb{E}[\nu((A'_m)^{2k})]= \mathbb{E}[\nu(A_{m+1}^{2k})]$. Thus invoking the induction hypothesis, 
\begin{align}
    |\mathbb{E}[\nu_0(f)]|  \leq \frac{2Ck}{n} . \frac{(C_0 k)^k}{n^k}. \label{eq:derivative_bound_int1}  
\end{align}
Next, we use the interpolation bound $|\mathbb{E}[\nu(f)]| \leq |\mathbb{E}[\nu_0(f)]| + \sup_{0\leq t \leq 1} | \mathbb{E}[\mathrm{d}\nu_t(f)/\mathrm{d}t]|$. To control the derivative, we use \eqref{eq:derivative} for the case of two replicas. Upon using Holder's inequality with $\tau_1 = (2k+2)/(2k+1)$ and $\tau_2 = (2k+2)$, we have, for some universal constant $L>0$, 
\begin{align}
    &|\frac{\mathrm{d}}{\mathrm{d}t} \mathbb{E}[\nu_t(f)]|  \nonumber \\
    &\leq L\beta^2 \Big(\mathbb{E}[\nu_t(A_m^{2k+2})]\Big)^{\frac{2k+1}{2k+2}} \Big[ \Big(\mathbb{E}[\nu_t((R^-_{1,2} - q_*)^{2k+2})]\Big)^{\frac{1}{2k+2}} + \Big(\mathbb{E}[\nu_t((R^- - ( \sigma_*^2+  q_*))^{2k+2})]\Big)^{\frac{1}{2k+2}}  \Big]. \nonumber 
\end{align}
Further, using \eqref{eq:derivative_bound_2}, for $\beta < \beta_0$ we have, for a larger universal constant $L>0$, 
\begin{align}
    &|\frac{\mathrm{d}}{\mathrm{d}t} \mathbb{E}[\nu_t(f)]|  \nonumber \\
    &\leq L\beta^2 \Big(\mathbb{E}[\nu(A_m^{2k+2})]\Big)^{\frac{2k+1}{2k+2}} \Big[ \Big(\mathbb{E}[\nu((R^-_{1,2} - q_*)^{2k+2})]\Big)^{\frac{1}{2k+2}} + \Big(\mathbb{E}[\nu((R^- - ( \sigma_*^2+  q_*))^{2k+2})]\Big)^{\frac{1}{2k+2}}  \Big]. \nonumber
\end{align}
This implies, 
\begin{align*}
    &\mathbb{E}[\nu(A_m^{2k+2})]\leq |\mathbb{E}[\nu_0(f)]| + \sup_{0\leq t \leq 1} |\frac{\mathrm{d}}{\mathrm{dt}} \mathbb{E}[\nu_t(f)]| \leq \frac{2Ck}{n} . \frac{(C_0 k)^k}{n^k} + \nonumber \\
    &L\beta^2 \Big(\mathbb{E}[\nu(A_m^{2k+2})]\Big)^{\frac{2k+1}{2k+2}} \Big[ \Big(\mathbb{E}[\nu((R^-_{1,2} - q_*)^{2k+2})]\Big)^{\frac{1}{2k+2}} + \Big(\mathbb{E}[\nu((R^- - ( \sigma_*^2+  q_*))^{2k+2})]\Big)^{\frac{1}{2k+2}}  \Big].
\end{align*}
Repeating the same argument for $B_m$, we obtain an analogous bound 
\begin{align*}
    &\mathbb{E}[\nu(B_m^{2k+2})] \leq  \frac{2Ck}{n} . \frac{(C_0 k)^k}{n^k} + \nonumber \\
    &L\beta^2 \Big(\mathbb{E}[\nu(B_m^{2k+2})]\Big)^{\frac{2k+1}{2k+2}} \Big[ \Big(\mathbb{E}[\nu((R^-_{1,2} - q_*)^{2k+2})]\Big)^{\frac{1}{2k+2}} + \Big(\mathbb{E}[\nu((R^- - ( \sigma_*^2+  q_*))^{2k+2})]\Big)^{\frac{1}{2k+2}}  \Big].
\end{align*}
Adding these bounds and using $xy \leq x^p + y^q$ for $1/p + 1/q = 1$, we obtain 
\begin{align}
    \Big( \mathbb{E}[\nu(A_m^{2k+2}] + \mathbb{E}[\nu(B_m^{2k+2})]  \Big) &\leq \frac{4Ck}{n}. \frac{(C_0 k)^k}{n^k} + 2L\beta^2 \Big( \mathbb{E}[\nu(A_m^{2k+2}] + \mathbb{E}[\nu(B_m^{2k+2})]  \Big) + \nonumber \\
    & 2L\beta^2 \Big(\mathbb{E}[\nu((R^-_{1,2} - q_*)^{2k+2})] + \mathbb{E}[\nu((R^- - ( \sigma_*^2+  q_*))^{2k+2})]  \Big). \nonumber  
\end{align}
Using site symmetry, $\mathbb{E}[\nu((R^-_{1,2} - q_*)^{2k+2})] = \mathbb{E}[\nu(A_2^{2k+2})]$ and $\mathbb{E}[\nu((R^- - ( \sigma_*^2+  q_*))^{2k+2})] = \mathbb{E}[\nu(B_2^{2k+2})]$. Thus for $m=2$, we obtain the desired induction bound upon choosing $C_0$ sufficiently large and $\beta_0$ small enough. In addition, the induction step follows for other $m$ once we plug in the bound for $m=2$ in the display above. This concludes the proof.  

\end{proof}

Throughout the remainder of the paper, we denote by $\alpha^{(m)}(x,y)$ the $m$-th partial derivative of $\alpha$ w.r.to $x$. We use $\alpha'$ and $\alpha''$ to denote $\alpha^{(m)}$, $m=1,2$ respectively. 

We also need the fact that the average magnetization,  i.e., $\frac{1}{n}\sum_{i=1}^n \langle y_i \rangle$ converges. This has been established for SK-model in \cite[Lemma 4]{chen2021convergence} and can be extended to our setup using Theorem \ref{thm:overlap_conc}. We state the result below and omit the proof for conciseness. 
\begin{lemma}\label{lem:avg_magnetization}
    Under the assumptions of Theorem \ref{thm:overlap_conc}, we have,
    \begin{align*}
        \lim\limits_{n \rightarrow \infty} \frac{1}{n} \sum_{i=1}^{n} \langle y_i \rangle = \ee \alpha'(h+ \beta \sqrt{q_*} Z, \beta^2 \sigma^2_*)
    \end{align*}
    in probability. 
\end{lemma}

The next result establishes gaussianity for the centered cavity field at high-temperature. 

\begin{thm}
\label{thm:cavity_field_normality}
    Let $U$ be an infinitely differentiable function on $\mathbb{R}$. For all $l,k\geq 1$, assume that $\mathbb{E}|U^{(l)}(Z)|^k < \infty$, where $U^{(l)}(\cdot)$ denotes the $l^{th}$ derivative of $U(\cdot)$ and $Z\sim \mathcal{N}(0,1)$. Let $G_i$ and $\xi$ be i.i.d $\mathcal{N}(0,1)$ random variables independent of $\nu(\cdot)$. Set $\dot{y}_i = y_i - \nu(y_i)$. There exists $\beta_0>0$ such that for all $\beta < \beta_0$ and for all $k \geq 1$, 
    \begin{align}
        \mathbb{E}\Big[ \Big( \nu\Big( U\Big( \frac{1}{\sqrt{n}} \sum_{i=1}^{n} G_i \dot{y}_i  \Big) \Big) - \mathbb{E}_{\xi}U(\sigma_* \xi)  \Big)^{2k} \Big] \leq \frac{K}{n^k}. \nonumber 
    \end{align}
    The constant $K$ depends on $k$, $U(\cdot)$, $\beta$, but is independent of $n$. 
\end{thm}

\begin{proof}
    The proof is a direct adaptation of \cite[ Theorem 1.7.11]{talagrand2010mean}. We thus highlight only the differences, and refer the reader to \cite{talagrand2010mean} for proof details. 

    Let $\mathbf{y}^1, \cdots, \mathbf{y}^{2k}$ denote replicas from $\nu(\cdot)$. Set $\dot{S}^{\ell} = \frac{1}{\sqrt{n}} \sum_{i=1}^{n} G_i \dot{y}_i^{\ell}$. 
    The proof introduces 
    \begin{align}
        T_{\ell,\ell'} = \mathbb{E}_0[\dot{S}_{\ell} \dot{S}_{\ell'}] - \sigma_*^2 \mathbf{1}(\ell= \ell') = \frac{1}{n} \sum_{i=1}^{n} \dot{y}^{\ell}_i \dot{y}^{\ell'}_i - \sigma_*^2 \mathbf{1}(\ell= \ell'),   \nonumber  
    \end{align}
    where $\mathbb{E}_0[\cdot]$ introduces expectation with respect to $\{G_i: 1\leq i \leq n\}$. 
    For $\ell \neq \ell'$, 
    \begin{align*}
        T_{\ell, \ell'} = \Big(\frac{1}{n} \sum_{i=1}^{\ell} y_i^{\ell} y_i^{\ell'} - q_* \Big) - \Big( \frac{1}{n} \sum_{i=1}^{n} y_i^{\ell} \nu(y_i^{\ell'}) - q_* \Big) - \Big( \frac{1}{n} \sum_{i=1}^{n} y_i^{\ell'} \nu(y_i^{\ell}) - q_* \Big) + \Big( \frac{1}{n} \sum_{i=1}^{n}(\nu(y_i))^2 - q_*\Big). \nonumber  
    \end{align*}
    Using Jensen's inequality and Theorem~\ref{thm:overlap_higher_moment}, we can establish that $\mathbb{E}[\nu(T_{\ell,\ell'}^{2k})] \leq K/n^{k}$. For $\ell=\ell'$, 
    \begin{align*}
        T_{\ell,\ell} = 
         \Big( \frac{1}{n} \sum_{i=1}^{n} (y_i^{\ell})^2 - (\sigma_*^2 + q_*)\Big) - 2 \Big( \frac{1}{n} \sum_{i=1}^{n} y_i^{\ell} \nu(y_i^{\ell}) - q_* \Big)  + \Big( \frac{1}{n} \sum_{i=1}^{n}(\nu(y_i))^2 - q_*\Big)
    \end{align*}
    This term can be controlled similarly to conclude $\mathbb{E}[\nu(T_{\ell,\ell}^{2k}) ] \leq K/n^{k}$. 

    The rest of the proof goes through unchanged, and is thus omitted. 
\end{proof}

For the next result, let $\nu(\cdot)_{-}$ denote the leave-one-out measure, i.e. the distribution $\nu$ on the $n-1$ variables $(y_1, \cdots, y_{n-1})$.  

\begin{lemma}
    There exists $\beta_0>0$ such that for all $\beta < \beta_0$, 
    \begin{align*}
        \mathbb{E}\Big[\Big( \nu(y_n) - \alpha'\Big(\frac{\beta}{\sqrt{n}} \sum_{i=1}^{n-1} g_{in} \nu(y_i)_{-} + h_n , \beta^2 \sigma_*^2 \Big)\Big)^2 \Big] \lesssim \frac{1}{n},\,\,\,\, 
        \mathbb{E}[(\nu(y_1) - \nu(y_1)_{-})^2] \lesssim \frac{1}{n}. \nonumber 
    \end{align*}
\end{lemma}

\begin{proof}
    By direct computation, we obtain that 
    \begin{align}
        \nu(y_n) = \frac{\int y \nu\Big( \exp \Big( \frac{\beta y_n}{\sqrt{n}} \sum_{i=1}^{n-1} g_{in} y_i + h_n y_n  \Big) \Big| y_n = y \Big) \mathrm{d}\mu(y) }{\int \nu\Big( \exp \Big( \frac{\beta y_n}{\sqrt{n}} \sum_{i=1}^{n-1} g_{in} y_i + h_n y_n  \Big) \Big| y_n  = y\Big) \mathrm{d}\mu(y)}.\label{eq:marginal_mean_exp} 
    \end{align}
    Set $\beta_n = \beta \sqrt{ \frac{n-1}{n}}$. We start with the numerator. 

    \begin{align}
        &\int y \nu\Big( \exp \Big( \frac{\beta y_n}{\sqrt{n}} \sum_{i=1}^{n-1} g_{in} y_i + h_n y_n  \Big) \Big| y_n = y \Big) \mathrm{d}\mu(y)  \nonumber \\
        =& \int y \exp\Big(\frac{\beta y}{\sqrt{n}} \sum_{i=1}^{n-1} g_{in} \nu(y_i)_{-} + h_n y \Big) \nu \Big( \exp\Big( \frac{\beta_n y}{\sqrt{n-1}} \sum_{i=1}^{n-1} g_{in} \dot{y}_i \Big)  \Big) \mathrm{d}\mu(y).  
    \end{align}
    Using Theorem~\ref{thm:cavity_field_normality}, we have, 
    \begin{align*}
        \mathbb{E}\Big[ \Big(  \nu \Big( \exp\Big( \frac{\beta_n y}{\sqrt{n-1}} \sum_{i=1}^{n-1} g_{in} \dot{y}_i \Big) - \exp\Big( \frac{\beta_n^2 \sigma_*^2 y^2}{2}  \Big) \Big)^2  \Big] \lesssim \frac{1}{n}. 
    \end{align*}
    In turn, this implies 
    \begin{align*}
        &\mathbb{E}\Big[ \Big( \int y \nu\Big( \exp \Big( \frac{\beta y_n}{\sqrt{n}} \sum_{i=1}^{n-1} g_{in} y_i + h_n y_n  \Big) \Big| y_n = y \Big) \mathrm{d}\mu(y) \Big. \Big.  \nonumber \\
       & \Big. \Big. - \int y \exp\Big( \frac{\beta y }{\sqrt{n}} \sum_{i=1}^{n-1} g_{in} \nu(y_i)_{-} + h_n y + \frac{\beta^2 \sigma_*^2 y^2}{2}  \Big) \mathrm{d}\mu(y)   \Big)^2 \Big] \lesssim \frac{1}{n}. 
    \end{align*}
    The denominator in \eqref{eq:marginal_mean_exp} maybe analyzed the same way. Putting the two approximations together yields the first assertion.

    To prove the second assertion, we adopt a coupling based approach. Let $\nu_{-\{1,n\}}$ denote the leave-two-out cavity system, where the $\{1,n\}$-spins have been removed from the system. Let $\mathbf{y}_{2:(n-1)} = (y_2, \cdots, y_{n-1}) \sim \nu_{-\{1,n\}}$. Define $Z_1 = \frac{1}{\sqrt{n}} \sum_{i=2}^{n-1} g_{1i} y_i$ and $Z_n = \frac{1}{\sqrt{n}} \sum_{i=2}^{n-1} g_{ni} y_i$. Now consider two probability distributions: under the first measure, $(Y_1,Y_n) \in [-1,1]^2$ are conditionally sampled given $\mathbf{y}_{2:(n-1)}$ from the distribution 
    \begin{align}
        \frac{\mathrm{d}\mathcal{P}}{\mathrm{d}\mu^{\otimes 2}} (y_1, y_n) \propto \exp\Big( \frac{\beta}{\sqrt{n}} g_{1n} y_1 y_n  + \beta y_1 Z_1 + \beta y_n Z_n + h_1 y_1 + h_n y_n\Big). \nonumber  
    \end{align}
    Under the second distribution, $(Y_1, Y_n) \in [-1,1]^2$ are conditionally sampled given $\mathbf{y}_{2: (n-1)}$ from the distribution 
    \begin{align*}
        \frac{\mathrm{d} \mathcal{Q}}{\mathrm{d}\mu^{\otimes 2}} (y_1, y_n) \propto \exp\Big( \beta y_1 Z_1 + \beta y_n Z_n + h_1 y_1 + h_n y_n \Big).  
    \end{align*}
    To relate this back to the original distributions, note that if $(Y_1, Y_n) \sim \mathcal{P}$ given $\mathbf{y}_{2: (n-1)}$, then $(Y_1, \mathbf{y}_{2:(n-1)} , Y_n) \sim \nu$. Similarly, if $(Y_1, Y_n)$ is sampled from $\mathcal{Q}$ given $\mathbf{y}_{2:(n-1)}$, then $(Y_1, \mathbf{y}_{2:(n-1)}) \sim \nu_{-}$. Thus we can naturally couple the two distributions by an exact coupling of $\mathbf{y}_{2:(n-1)}$. For notational convenience, we also use $\mathcal{P}$ and $\mathcal{Q}$ to denote the joint distributions of $(Y_1, \mathbf{y}_{2:(n-1)},Y_n)$ in the respective cases. In particular, $\nu(y_1) - \nu(y_1)_{-} = \mathbb{E}_\mathcal{P}(Y_1) - \mathbb{E}_\mathcal{Q}(Y_1)$. 

    Define the cumulant generating function 
    \begin{align}
        G(\theta_0, \theta_1, \theta_n) = \log \int \int \exp\Big( \theta_0 y_1 y_n + \theta_1 y_1 + \theta_n y_n \Big) \mathrm{d}\mu(y_1) \mathrm{d}\mu (y_n). \nonumber 
    \end{align}
    We have, 
    \begin{align}
        \mathbb{E}_{\mathcal{P}}(Y_1 | \mathbf{y}_{2:(n-1)}) &= \frac{\partial G}{\partial \theta_1} \Big(\frac{\beta}{\sqrt{n}}g_{1n}, \beta Z_1 + h_1, \beta Z_n + h_n \Big), \nonumber \\
        \mathbb{E}_{\mathcal{Q}}(Y_1 | \mathbf{y}_{2:(n-1)}) &= \frac{\partial G}{\partial \theta_1} \Big(0, \beta Z_1 + h_1, \beta Z_n + h_n \Big). \nonumber  
    \end{align}

    In turn, this implies 
    \begin{align}
        &|\mathbb{E}_{\mathcal{P}}(Y_1) - \mathbb{E}_{\mathcal{Q}}(Y_1)| \leq \mathbb{E}[|\mathbb{E}_{\mathcal{P}}(Y_1| \mathbf{y}_{2:(n-1)}) -\mathbb{E}_{\mathcal{Q}}(Y_1| \mathbf{y}_{2:(n-1)}) |  ] \nonumber \\
        &\leq \mathbb{E}\Big[  \Big| \frac{\partial G}{\partial \theta_1} \Big(\frac{\beta}{\sqrt{n}}g_{1n}, \beta Z_1 + h_1, \beta Z_n + h_n \Big) - \frac{\partial G}{\partial \theta_1} \Big(0, \beta Z_1 + h_1, \beta Z_n + h_n \Big)  \Big| \Big] \nonumber \\
        &\leq \frac{\beta}{\sqrt{n}} |g_{1n}| \mathbb{E}\Big[ \Big| \frac{\partial^2 G}{\partial \theta_0 \theta_1} \Big(\theta_0^*, \beta Z_1 + h_1, \beta Z_n + h_n \Big)  \Big|\Big] . \label{eq:exp_bound_cavity}
    \end{align}

Consider the three-parameter exponential family, 
\begin{align}
    \frac{\mathrm{d}F_{\theta_0, \theta_1, \theta_n}}{\mathrm{d}\mu} (y_1, y_n) = \exp\Big( \theta_0 y_1 y_n + \theta_1 y_1 + \theta_n y_n  - G(\theta_0, \theta_1, \theta_n )\Big). \nonumber 
\end{align}

    We note that using well-known properties of exponential families \cite{wainwright2008graphical}, 
    \begin{align}
        \frac{\partial^2 G}{\partial \theta_0 \partial \theta_1} \Big(\theta_0, \theta_1, \theta_n \Big)  = \mathrm{cov}_{F_{\theta_0, \theta_1, \theta_n}} (Y_1, Y_1 Y_n) 
    \end{align}

    As $Y_1, Y_n \in [-1,1]$, $|\mathrm{cov}_{F_{\theta_0, \theta_1, \theta_n}} (Y_1, Y_1 Y_n)| \leq 2$. Plugging this bound back into \eqref{eq:exp_bound_cavity}, we have, 
    \begin{align}
        |\nu(y_1) - \nu(y_1)_{-}| \leq \frac{2 \beta}{\sqrt{n}} |g_{1n}|. \nonumber 
    \end{align}
We have the desired conclusion upon squaring and evaluating the expectation with respect to $g_{1n}$. This completes the proof.

\end{proof}

We now introduce the cavity iteration. Their convergence is a important step in analyzing Algorithm \ref{alg:amp}. For every $n \ge 1$, and $0 \le k \le n-1$, we define $[n]_k= \{S \subseteq [n] \,\text{s.t.}\, |S| \le n-(k+1)\}$. 
\begin{itemize}
    \item[(i)] \textbf{Initialization:} For $n \ge 1$ and any $S \in [n]_0$, define $w^{[0]}_S \in \mathbb R^{n \setminus S}$ by $w^{[0]}_{S,i}= 0$, $i \in [n]\setminus S$. Also, for any $S \in [n]_1$, define $w^{[0]}_S \in \mathbb R^{n \setminus S}$ by $w^{[1]}_{S,i}= 0$, $i \in [n]\setminus S$.

    \item[(ii)] \textbf{Iteration $k = 2$:} For $S \in [n]_2$, define $w^{[2]}_S \in \mathbb R^{n \setminus S}$ by 

\begin{equation*}
    w^{[2]}_{S,i}= \frac{1}{\sqrt{n}} \sum_{j \notin S \cup \{i\}} g_{ij} \sqrt{q_*}, \quad i \in [n] \setminus S. 
\end{equation*}

    \item[(iii)] \textbf{Iteration} $k \geq 2$: for any $2 \le k \le n-2$ and $S \in [n]_{k+1}$, define $w^{[k+1]}_S \in \mathbb R^{n \setminus S}$ by 

\begin{equation}
    w^{[k+1]}_{S,i}= \frac{1}{\sqrt{n}} \sum_{j \notin S \cup \{i\}} g_{ij} \alpha'(\beta w^{[k]}_{S \cup \{i\},j}+h_j, \beta^2 \sigma^2_{*}), \quad i \in [n] \setminus S. \label{eq:cavity_iteration}
\end{equation}
\item[(iv)]  We define $w^{[k]}_i= w^{[k]}_{\emptyset, i}$. 
\end{itemize}

Note that, this is an algorithm based on self-avoiding walk which has been proposed to analyze AMP algorithms \cite{chen2021convergence}, community detection of sparse stochastic block models \cite{hopkins2017efficient} among others.

We require the following $\ell_p$ moment estimate for $w^{[k]}_{S,i}$.

\begin{prop}\label{prop:tang_prop_2}
    For any $k \ge 0$ and $p \ge 1$, there exists a constant $C= C(k,p)>0$ such that for any $n \ge k+3$,
    \begin{align*}
        \sup \left\{ \Big(\mathbb E |w^{[k]}_{S,i}-w^{[k]}_{S\cup \{i'\},i}|^p\Big)^{1/p}\right\} \le \frac{C}{\sqrt{n}},
    \end{align*}
    where the supremum is taken over all $i,i' \in [n]$, $S \subset [n]$ with $ i \neq i'$, $i,i' \notin S$, $|S|\le n-k+2$.
\end{prop}
The proof follows exact same as \cite[Proposition 2]{chen2021convergence} once we note $f(x)= \alpha'(\beta x +h, \beta^2 \sigma^2)$ is bounded Lipschitz function for any $\beta >0$, $h \in \mathbb R$, $\sigma>0$ with Lipschitz constant independent of $h$ since $\mu$ is compactly supported. The details of the proof is omitted.

A weak law of large number for the cavity vectors $(w^{[k]}, w^{[k-1]},\ldots, w^{[0]})$ follows from Proposition \ref{prop:tang_prop_2} and tabulated below. For the proof, we will use the following notation: Two sequence of random variables $X_n$ and $Y_n$ are called $X_n \asymp_p Y_n$, if $\bE |X_n-Y_n|^p \rightarrow 0$ as $n \rightarrow \infty$.

\begin{lemma}\label{prop:tang_thm_1}
    For $k \ge 0$ and any bounded Lipschitz function $\phi: \mathbb R^{k+1} \mapsto \mathbb R$, we have, in probability
    \[
    \lim\limits_{n\rightarrow \infty} \frac{1}{n} \sum_{i \in [n]} \phi(w^{[k]}_i,w^{[k-1]}_i, \ldots ,w^{[0]}_i)= \bE \phi(W_k, W_{k-1},\ldots, 0),
    \]
    where $(W_k, W_{k-1}, \ldots, W_1)$ are jointly centered gaussian with 
    \begin{equation*}
        \bE W_{i+1} W_{j+1}= \bE \alpha'(\beta W_{i}+h, \beta^2 \sigma_*^2)\alpha'(\beta W_{j}+h, \beta^2 \sigma_*^2),
    \end{equation*}
where $0 \le i,j \le k-1$ and the expectation on the RHS is taken w.r.to  the joint  distribution of $(W_i,W_j)$ and $h$.
\end{lemma}

\begin{proof}
We complete the proof by induction on $k$. We start with a stronger induction hypothesis: we assume that for $k \geq 0$, any bounded Lipschitz function $\phi:\mathbb{R}^{k+1} \to \mathbb{R}$, $a, b \in \mathbb{R}$, we have, in probability,  
\begin{align*}
   \lim_{n \to \infty } \frac{1}{n} \sum_{i=1}^{n} \phi( a w_i^{[k]} + b h_i, \cdots, a w_i^{[0]} + b h_i) = \mathbb{E}[\phi(aW_k + b h, \cdots, a W_0 + b h)],  
\end{align*}
where $\mathbb{E}[\cdot]$ in the RHS refers to expectation w.r.to $(W_0, \cdots, W_k)$ and $h$. In addition, we assume that $(W_0, \cdots, W_k)$ have the desired joint multivariate Gaussian distribution. 

The base case can be verified directly using the definition of the iterates $w_i^{[0]}$ and $w_i^{[1]}$. Recall now from Proposition~\ref{prop:tang_prop_2} that for $1 \leq \ell \leq k$, 
\begin{align*}
    w_1^{[\ell+1]} &\asymp_1 w^{[\ell+1]}_{\{2\},1} = \frac{1}{\sqrt{n}} \sum_{j \neq 1,2} g_{1j} \alpha'(\beta w^{[\ell]}_{\{1,2\},j} + h_j, \beta^2 \sigma_*^2), \\
    w_2^{[\ell+1]} &\asymp_1 w^{[\ell+1]}_{\{1\},2} = \frac{1}{\sqrt{n}} \sum_{j \neq 1,2} g_{2j} \alpha'(\beta w^{[\ell]}_{\{1,2\},j} + h_j, \beta^2 \sigma_*^2). 
\end{align*}
Conditioning on $\{h_i : 1\leq i \leq n\}$, note that $\{w^{[\ell]}_{\{1,2\},j}: j \neq 1,2\}$ are functions of $\{g_{ij}: i,j \neq 1,2\}$. Thus conditioning on $\{h_i: 1\leq i \leq n\}$ and $\{g_{ij}: i,j \neq 1,2\}$, we see that 
\begin{align*}
    (w_{\{2\},1}^{[k+1]}, \cdots, w_{\{2\},1}^{[0]}) \,\,\, \textrm{and} \,\,\, (w_{\{1\},2}^{[k+1]}, \cdots, w_{\{1\},2}^{[0]})
\end{align*}
are independent mean-zero multivariate gaussian. To specify the covariance structure, note that for $1\leq a,b \leq k$ 
\begin{align}
    \mathbb{E}_{g_{1j}}[w^{[a+1]}_{\{2\},1} w^{[b+1]}_{\{2\},1}  ] &= \frac{1}{n} \sum_{j \neq 1,2} \alpha'(\beta w^{[a]}_{\{1,2\},j} + h_j , \beta^2 \sigma_*^2) \alpha'(\beta w^{[b]}_{\{1,2\},j} + h_j , \beta^2 \sigma_*^2) \nonumber \\
    &\asymp_1 \frac{1}{n} \sum_{j=1}^{n}  \alpha'(\beta w^{[a]}_{j} + h_j , \beta^2 \sigma_*^2) \alpha'(\beta w^{[b]}_{j} + h_j , \beta^2 \sigma_*^2)\nonumber \\
    &\stackrel{P}{\to} \mathbb{E}[\alpha'(\beta W_a +h, \beta^2 \sigma_*^2) \alpha'(\beta W_b + h, \beta^2 \sigma_*^2)], \nonumber 
\end{align}
where the last step follows by the induction hypothesis. Similarly, we have, 
\begin{align}
    \mathbb{E}_{g_{2j}}[w^{[a+1]}_{\{1\},2} w^{[b+1]}_{\{1\},2}  ] \stackrel{P}{\to} \mathbb{E}[\alpha'(\beta W_a +h, \beta^2 \sigma_*^2) \alpha'(\beta W_b + h, \beta^2 \sigma_*^2)]. \nonumber 
\end{align}

Now, we fix two bounded Lipschitz functions $\phi_1, \phi_2 : \mathbb{R}^{k+2} \to \mathbb{R}$ and constants $a, b \in \mathbb{R}$. Conditioning on $\{h_i : 1\leq i \leq n\}$ and evaluating the expectation with respect to $\{g_{ij}: i<j\}$, we have, 
\begin{align}
    &\mathbb{E}[\phi_1(a w_1^{[k+1]}+ b h_1, \cdots, a w_1^{[0]} + b h_1) \phi_2(a w_2^{[k+1]}+ b h_2, \cdots, a w_2^{[0]} + b h_2)] \nonumber \\
    =& \mathbb{E}[\phi_1(a w_{\{2\},1}^{[k+1]}+ b h_1, \cdots, a w_{\{2\}, 1}^{[0]} + b h_1) \phi_2(a w_{\{1\},2}^{[k+1]}+ b h_2, \cdots, a w_{\{1\},2}^{[0]} + b h_2)]  + o_P(1) \nonumber \\
    =& \mathbb{E}\Big[ \mathbb{E}_{g_{1j}}[\phi_1(a w_{\{2\},1}^{[k+1]}+ b h_1, \cdots, a w_{\{2\}, 1}^{[0]} + b h_1)] \mathbb{E}_{g_{2j}}[\phi_2(a w_{\{1\},2}^{[k+1]}+ b h_2, \cdots, a w_{\{1\},2}^{[0]} + b h_2)]\Big] + o_P(1) \nonumber \\
    &\stackrel{P}{\to} \mathbb{E}[\phi_1(aW_{k+1}+bh_1, \cdots, a W_{0}+b h_1)] \mathbb{E}[\phi_2(a W_{k+1} + bh_2, \cdots, a W_{0} + b h_2)],  \label{eq:conditional_limit} 
\end{align}
where the $\mathbb{E}[\cdot]$ in the final display denotes expectation with respect to $(W_{k+1}, \cdots, W_{0})$. Additionally, the entire argument above is conditional on $\{h_i: 1\leq i \leq n\}$, and thus the final convergence is in probability with respect to the randomness in $\{h_i : 1\leq i \leq n\}$. 

Now, evaluating the expectation with respect to both $\{g_{ij}: i<j\}$ and $\{h_i : 1\leq i \leq n\}$, we have, 
\begin{align}
    &\mathbb{E}\Big[ \Big( \frac{1}{n} \sum_{i=1}^{n} \phi_1(a w_i^{[k+1]} + b h_i, \cdots, a w_i^{[0]} + b h_i) \Big) \Big(\frac{1}{n} \sum_{j=1}^{n} \phi_1(a w_j^{[k+1]} + b h_j, \cdots, a w_j^{[0]} + b h_j) \Big) \Big)  \Big] \nonumber \\
    =& \mathbb{E}[\phi_1(a w_1^{[k+1]} + b h_1, \cdots, a w_1^{[0]} + b h_1) \phi_2(a w_2^{[k+1]} + b h_2, \cdots, a w_1^{[0]} + b h_2)] + o(1), \nonumber 
\end{align}
where the last step uses symmetry among the sites and that $\phi_1$, $\phi_2$ are bounded. Now, we use \eqref{eq:conditional_limit}, DCT and the independence of $h_1$ and $h_2$ to conclude 
\begin{align}
    &\lim_{n \to \infty} \mathbb{E}\Big[ \Big( \frac{1}{n} \sum_{i=1}^{n} \phi_1(a w_i^{[k+1]} + b h_i, \cdots, a w_i^{[0]} + b h_i) \Big) \Big(\frac{1}{n} \sum_{j=1}^{n} \phi_2(a w_j^{[k+1]} + b h_j, \cdots, a w_j^{[0]} + b h_j) \Big) \Big)  \Big] \nonumber \\
    &= \mathbb{E}[\phi_1(a W_{k+1} + b h, \cdots, a W_{0} + b h) ] \mathbb{E}[\phi_2(a W_{k+1} + b h, \cdots, a W_{0} + b h) ]. \label{eq:limit_expectation} 
\end{align}
To complete the induction step for $k+1$, we can now check that 
\begin{align}
    \lim_{n \to \infty} \mathbb{E}\Big[ \frac{1}{n} \sum_{i=1}^{n} \psi(a w_i^{[k+1]} + b h_i, \cdots, a w_i^{[0]} + b h_i) - \mathbb{E}[\psi(a W_{k+1} + b h, \cdots, a W_0 + b h)] \Big]^2=0. \nonumber 
\end{align}
In the display above, $\psi$ is a bounded Lipschitz function and $a,b \in \mathbb{R}$ are arbitrary constants. The final assertion follows by squaring and using \eqref{eq:limit_expectation}.

\end{proof}

\section{Approximate Message Passing tracks the cavity iteration}\label{appendix: amp+cavity}

Consider the SK model with random external field \eqref{eq:SK_model}. To approximate the mean-vector $\nu(\mathbf{y})$, we consider the following Approximate Message Passing (AMP) algorithm: 

\begin{defn}[AMP algorithm for soft spins]\label{app_B: AMP_description}
    \begin{itemize}
    \item[(i)] \textbf{Initialization:} Set $\mathbf{u}^{[0]}=0 \in \mathbb{R}^n$, $\mathbf{u}^{[1]}=0$. 
    \item[(ii)] \textbf{Iteration $k = 2$:} Set $\mathbf{u}^{[2]}= (u^{[2]}_i)$, $u^{[2]}_i = \frac{1}{\sqrt{n}}\sum_{l=1}^{n} g_{il} \sqrt{q_*}$. 
    \item[(iii)] \textbf{Iteration} $k \geq 2$: Set $\mathbf{u}^{[k]} = (u^{[k]}_i)$, and 
    \begin{align*}
        u^{[k+1]}_i = \frac{1}{\sqrt{n}} \sum_{l=1}^{n} g_{il} \alpha'(\beta u_l^{[k]} + h_l, \beta^2 \sigma_*^2) - \Big( \frac{\beta}{n} \sum_{j=1}^{n} \alpha''(\beta u_j^{[k]} + h_j, \beta^2 \sigma_*^2) \Big) \alpha'(\beta u^{[k-1]}_i + h_i , \beta^2 \sigma_*^2). \nonumber  
    \end{align*}
\end{itemize}
\end{defn}

Recall now the cavity iteration \eqref{eq:cavity_iteration}. The main result in this section states that the AMP algorithm tracks the cavity iteration \eqref{eq:cavity_iteration}. 

\begin{thm}
    For any $k \geq 0$, there exists a constant $C_k \geq 0$ (independent of $n$) such that 
    \begin{align*}
        \mathbb{E}[\| \mathbf{u}^{[k]} - \mathbf{w}^{[k]} \|^2 ] \leq C_k. \nonumber 
    \end{align*}
\end{thm}

\begin{proof}


    Our proof proceeds by induction on $k$. For two sequences $\{u_i: 1\leq i \leq n\}$ and $\{v_i: 1\leq i \leq n\}$ and $S \subseteq [n]$, we say that $u_i \asymp_2 v_i$ for all $i \in S$ if 
    \begin{align*}
        \max_{i \in S} \mathbb{E}(u_i - v_i)^2 \leq \frac{C}{n}
    \end{align*}
    for some constant $C>0$ independent of $n$. We say that $\mathbf{u} = (u_i) \asymp_2 \mathbf{v}= (v_i)$ if $u_i \asymp_2 v_i$ for all $i \in [n]$. 
    
    We assume that for all $k' \leq k$
      $\{u_i^{[k']} : i \in [n]\} \asymp_2 \{w_i^{[k']}: i \in [n]\}$. The base cases $k=0,1$ can be verified directly.
      Construct first the surrogate 
      \begin{align}
          \bar{\mathbf{u}}^{[k+1]} = \frac{1}{\sqrt{n}}\mathbf{G} \alpha'(\beta \mathbf{w}^{[k]} + \mathbf{h}, \beta^2 \sigma_*^2 ) - \Big( \frac{\beta}{n} \sum_{j=1}^{n} \alpha'' (\beta w_j^{[k]} + h_j , \beta^2 \sigma_*^2)\Big)  \alpha' ( \beta \mathbf{w}^{[k-1]} + \mathbf{h} , \beta^2 \sigma_*^2).  \label{eq:surrogate}
      \end{align}
 Since $\mathbf{G}/\sqrt{n}$ is square integrable, $\alpha'$ is bounded Lipschitz and $\alpha''$ is bounded, we have, 
      \begin{align*}
          \mathbb{E}[\| \mathbf{u}^{[k+1]} - \bar{\mathbf{u}}^{[k+1]} \|^2] \lesssim \mathbb{E}[\| \mathbf{u}^{[k]} - \mathbf{w}^{[k]}\|^2 ] + \mathbb{E}[\| \mathbf{u}^{[k-1]} - \mathbf{w}^{[k-1]}\|^2 ].  
      \end{align*}
      Using site symmetry, we obtain that $\mathbf{u}^{[k]} \asymp_2 \bar{\mathbf{u}}^{[k]}$. 
  As a second approximation, we check that 
  \begin{align}
      \bar{u}^{[k+1]}_i  \asymp_2  \frac{1}{\sqrt{n}} \sum_{j\neq i} g_{ij} \alpha'\Big( \frac{\beta}{\sqrt{n}} \sum_{\ell \neq i,j} g_{j\ell} \alpha'(\beta w^{[k-1]}_{\{j\},\ell} + h_{\ell}, \beta^2 \sigma_*^2) + h_j , \beta^2 \sigma_*^2 \Big).  \label{eq:approx_int1}
  \end{align}
  To this end, note that 
  \begin{align*}
      &\frac{1}{\sqrt{n}} \sum_{j \neq i} g_{ij} \alpha'(\beta w^{[k]}_j + h_j , \beta^2 \sigma_*^2) \nonumber \\
      &= \frac{1}{\sqrt{n}} \sum_{j \neq i} g_{ij} \alpha' \Big( \frac{\beta}{\sqrt{n}} \sum_{\ell \neq i,j} g_{j \ell} \alpha' ( \beta w^{[k-1]}_{\{j\},\ell} + h_{\ell}, \beta^2 \sigma_*^2) + \frac{\beta}{\sqrt{n}} g_{ij} \alpha'(\beta w^{[k-1]}_{\{j\},i} + h_i, \beta^2 \sigma_*^2) + h_j , \beta^2 \sigma_*^2 \Big) \\
      &= \frac{1}{\sqrt{n}} \sum_{j \neq i} g_{ij} \alpha' \Big( \frac{\beta}{\sqrt{n}} \sum_{\ell \neq i,j} g_{j \ell} \alpha' ( \beta w^{[k-1]}_{\{j\},\ell} + h_{\ell}, \beta^2 \sigma_*^2) + h_j , \beta^2 \sigma_*^2\Big) \nonumber \\
      &+ \frac{\beta}{n} \sum_{j \neq i} g_{ij}^2 \alpha'(\beta w^{[k-1]}_{\{j\},i} + h_i, \beta^2 \sigma_*^2)  \alpha'' \Big( \frac{\beta}{\sqrt{n}} \sum_{\ell \neq i,j} g_{j \ell} \alpha' ( \beta w^{[k-1]}_{\{j\},\ell} + h_{\ell}, \beta^2 \sigma_*^2) + h_j  , \beta^2 \sigma_*^2\Big) + O_P\Big(\frac{1}{\sqrt{n}} \Big). 
  \end{align*}

  Using \eqref{eq:surrogate}, we have, 
  \begin{align*}
      &\bar{u}^{[k+1]}_i = \frac{1}{\sqrt{n}} \sum_{j \neq i} g_{ij} \alpha'\Big( \frac{\beta}{\sqrt{n}} \sum_{\ell \neq i,j} g_{j\ell} \alpha'(\beta w^{[k-1]}_{\{j\},\ell} + h_\ell , \beta^2 \sigma_*^2) +h_j , \beta^2 \sigma_*^2 \Big)  \\
      & + \frac{\beta}{n} \sum_{j \neq i} g_{ij}^2 B_{ij} D_{ij} - \frac{\beta}{n} \sum_{j=1}^{n} B_j D_i + O_P\Big(\frac{1}{\sqrt{n}} \Big), 
  \end{align*}
  where 
  \begin{align*}
      B_{ij} &= \alpha''\Big(  \frac{\beta}{\sqrt{n}} \sum_{\ell \neq i,j} g_{j, \ell} \alpha'(\beta w^{[k-1]}_{\{j\},\ell} + h_{\ell} , \beta^2 \sigma_*^2) + h_j , \beta^2 \sigma_*^2\Big), \\
      D_{ij} &= \alpha'(\beta w^{[k-1]}_{\{j\},i} + h_i , \beta^2 \sigma_*^2), \\
      B_j &= \alpha''(\beta w_j^{[k]} + h_j , \beta^2 \sigma_*^2)\\
      D_i &= \alpha'(\beta w^{[k-1]}_i + h_i , \beta^2 \sigma_*^2). \nonumber 
  \end{align*}
      
   The approximation \eqref{eq:approx_int1} now follows once we check that 
   \begin{align}
       \max_{i \in [n]} \mathbb{E}\Big[\Big(  \frac{1}{n} \sum_{j \neq i} g_{ij}^2 B_{ij} D_{ij} - \frac{1}{n} \sum_{j=1}^{n} B_j D_i  \Big)^2\Big] \lesssim \frac{1}{n}. \nonumber 
   \end{align}
   To bound follows exactly from the proof of \cite[Lemma 6]{chen2021convergence}---we omit the details, and refer the reader to \cite{chen2021convergence} for a detailed proof. 

   Since $\mathbf{u}^{[k]} \asymp_2 \bar{\mathbf{u}}^{[k]}$, using \eqref{eq:approx_int1}, we note that it suffices to prove 
   \begin{align}
       &\frac{1}{\sqrt{n}} \sum_{j \neq i} g_{ij} \alpha' \Big( \frac{\beta}{\sqrt{n}} \sum_{\ell \neq i,j} g_{j\ell} \alpha'(\beta w^{[k-1]}_{\{j\},\ell} + h_{\ell} , \beta^2 \sigma_*^2) + h_j, \beta^2 \sigma_*^2  \Big) \nonumber \\
       &\asymp_2 \frac{1}{\sqrt{n}} \sum_{j \neq i} g_{ij} \alpha' \Big( \frac{\beta}{\sqrt{n}} \sum_{\ell \neq i,j} g_{j\ell} \alpha'(\beta w^{[k-1]}_{\{i,j\},\ell} + h_{\ell} , \beta^2 \sigma_*^2) + h_j, \beta^2 \sigma_*^2  \Big):= w_i^{[k+1]}. 
   \end{align}

   We fix $i \in [n]$ and for $j \neq i$, set 
   \begin{align*}
       L_j &= \alpha'\Big( \frac{\beta}{\sqrt{n}} \sum_{\ell \neq i,j} g_{j\ell} \alpha'(\beta w^{[k-1]}_{\{j\},\ell} + h_{\ell}, \beta^2 \sigma_*^2) + h_j, \beta^2 \sigma_*^2  \Big), \\
       K_j &= \alpha'\Big( \frac{\beta}{\sqrt{n}} \sum_{\ell \neq i,j} g_{j\ell} \alpha'(\beta w^{[k-1]}_{\{i,j\},\ell} + h_{\ell}, \beta^2 \sigma_*^2) + h_j, \beta^2 \sigma_*^2  \Big)
   \end{align*}
      
    Proceeding as in the proof of \cite[Section 6]{chen2021convergence}, one can show that 
    \begin{align}
        &\mathbb{E}\Big[ \Big( \frac{1}{\sqrt{n}} \sum_{j \neq i} g_{ij} (L_j - K_j) \Big)^2\Big] \\
        &= \frac{1}{n} \sum_{\tau, \iota \neq i, \tau \neq \iota} \mathbb{E} \Big[ \mathbb{E}_{g_{i\tau}}[g_{i\tau} L_\iota] \mathbb{E}_{g_{i \iota}}[g_{i \iota} L_{\tau} ] \Big] + \frac{1}{n} \sum_{j \neq i} \mathbb{E}[g_{ij}^2  (L_j - K_j)^2]. \label{eq:squared_moment}
    \end{align}
  Following the proof of \cite[Section 6]{chen2021convergence}, one can show that the second term is of $O(1/n)$. We omit the details to avoid repetition. It remains to control the first term. We follow the second moment approach introduced in \cite{chen2021convergence}.

  Note that $L_{\tau} = \alpha'(\Delta_{\tau} + h_{\tau}, \beta^2 \sigma_*^2)$, where 
  \begin{align*}
      \Delta_{\tau} = \frac{\beta}{\sqrt{n}} \sum_{\tau_{k-1} \neq i, \tau} g_{ \tau \tau_{k-1}} \alpha' (\beta w^{[k-1]}_{\{\tau\}, \tau_{k-1}} + h_{\tau_{k-1}}, \beta^2 \sigma_*^2). \nonumber 
  \end{align*}
  Since $\iota \neq \tau, i$, $g_{i \iota} \neq g_{\tau \tau_{k-1}}$. Thus using gaussian integration by parts and chain rule, we have, 
  \begin{align}
      \mathbb{E}_{g_{i\iota}} [g_{i\iota} L_{\tau}] = \frac{\beta}{\sqrt{n}}\mathbb{E}_{g_{i\iota}}\Big[\alpha''(\Delta_{\tau} + h_{\tau},\beta^2 \sigma_*^2) \sum_{\tau_{k-1} \neq i, \tau} g_{\tau \tau_{k-1}} \frac{\partial}{\partial g_{i\iota}} \alpha'(\beta w^{[k-1]}_{\{\tau\}, \tau_{k-1}} + h_{\tau_{k-1}}, \beta^2 \sigma_*^2) \Big]. \nonumber  
  \end{align}
  We unroll the iterations till we encounter $g_{i \iota}$ or $g_{\iota i}$ for the first time. Assume that this occurs for $1\leq r \leq k-1$; note that if this occurs at iteration $r$, the subsequent iterations do not contain these terms due to the self-avoiding nature of the paths. Consequently, we obtain the following expression: 

  \begin{align*}
      \mathbb{E}_{g_{i\iota}}[g_{i\iota} L_{\tau}] = \sum_{r=1}^{k-1} \Big( \frac{\beta}{\sqrt{n}} \Big)^{(k-r+1)} \mathbb{E}_{g_{i\iota}}\Big[  \sum_{I_{\tau,r} \in \mathcal{I}_{\tau,r}} G_{I_{\tau,r}} F_{I_{\tau,r}}(G)  \Big] \mathbf{1}_{\{ \{\tau_r, \tau_{r-1}\} = \{i, \iota\} \}}
  \end{align*}
  where $\mathcal{I}_{\tau,r}$ is the set of self-avoiding paths $I_{\tau,r} = (\tau_k, \tau_{k-1}, \cdots, \tau_{r-1})$. In addition, 
  \begin{align}
      G_{I_{\tau,r}} &= \prod_{s=r}^{k-1} g_{\tau_{s+1} \tau_s} \nonumber \\
      F_{I_{\tau,r}} &= \alpha''(\Delta_{\tau} + h_{\tau}, \beta^2 \sigma_*^2) \Big( \prod_{s=r}^{k-1} \alpha''(\beta w^{[s]}_{\{\tau, \tau_{k-1}, \cdots, \tau_{s+1}\}, \tau_s} +h_{\tau_s} , \beta^2 \sigma_*^2)  \Big)  \times \nonumber \\
      & \alpha'(\beta w^{[r-1]}_{\{\tau, \tau_{k-1}, \cdots, \tau_r\}, \tau_{r-1}} + h_{\tau_{r-1}} , \beta^2 \sigma_*^2). \nonumber 
  \end{align}

  Thus we have, 
  \begin{align}
      &\mathbb{E}\Big[ \mathbb{E}_{g_{i \tau}} [g_{i \tau} L_{\iota}] \mathbb{E}_{g_{i \iota}}[g_{i \iota} L_{\tau}] \Big] = \sum_{r,r'=1}^{k-1} \Big(\frac{\beta}{\sqrt{n}} \Big)^{2k+2 - (r+r')} \cdot \nonumber \\
      & \sum_{I_{\tau,r} \in \mathcal{I}_{\tau,r}, I_{\iota,r'} \in \mathcal{I}_{\iota,r'}} \mathbb{E}\Big[ G_{I_{\tau,r}} G_{I_{\iota,r'}} F_{I_{\tau,r}} (G) F_{I_{\iota,r'}}(G)\Big] \mathbf{1}_{\{\tau_r, \tau_{r-1} \} = \{i, \iota\}, \{\iota_{r'}, \iota_{r'-1} \} = \{ i, \tau \} } \label{eq:cross_product_int1}
  \end{align}
   where the last equality uses  that $G_{I_{\tau,r}} F_{I_{\tau,r}}(G)$ is independent of $g_{i \tau}$ and vice versa. The rest of the proof proceeds along the same lines as the proof of  \cite[Theorem 2]{chen2021convergence}. For the sake of conciseness, we highlight the main steps in the remaining proof, and refer to \cite{chen2021convergence} for complete details. First, note that $I_{\tau,r}$, $I_{\iota, r'}$ are self avoiding paths which satisfy the constraints $(\tau_r, \tau_{r-1}) = (i, \iota)$ or $(\tau_r, \tau_{r-1}) = (\iota, i)$ and $(\iota_{r'}, \iota_{r'-1}) = (i, \tau)$ or $(\iota_{r'}, \iota_{r'-1}) = ( \tau, i)$. Consequently, the pairs of self-avoiding paths $(I_{\tau,r}, I_{\iota,r'})$ may be classified according to the number of common (undirected edges) $s$ and distinct vertices in the shared edges $t$. By definition, the last edge in $I_{\tau,r}$ i.e. $(\tau_r, \tau_{r-1})$ cannot be shared with $I_{\iota,r'}$. This leads to the a priori bounds on $s,t$: either $s=t=0$ or $1\leq s \leq \min\{ k -r, k-r'\}$ and $s+1 \leq t \leq \min\{ 2s, k-r + 1, k-r'+1\}$. For fixed values of $s$ and $t$ these self-avoiding path pairs may be further classified depending on how many of $\{\tau, \tau_r\}$ occur on the overlapping edges---note that this value $\ell \in \{0,1,2\}$. Denoting this collection of self-avoiding path pairs as $\mathcal{I}_{\tau,\iota, r,r'}(s,t,\ell)$, \cite[Lemma 7]{chen2021convergence} establishes that $|\mathcal{I}_{\tau,\iota, r,r'}(s,t,\ell)| \lesssim n^{2k - r -r' - t + \ell -2}$. 

   For a specific set of values of $(s,t,\ell)$, to evaluate $\mathbb{E}\Big[ G_{I_{\tau,r}} G_{I_{\iota,r'}} F_{I_{\tau,r}} (G) F_{I_{\iota,r'}}(G)\Big]$, there are $2k - r -r'-2s$ many non-overlapping edges in the pair $(I_{\tau,r}, I_{\iota,r'})$. The corresponding gaussian variables appear once in the product $G_{I_{\tau,r}} G_{I_{\iota,r'}}$---we can apply integration by parts to evaluate the derivative with respect to these gaussians. This implies 
   \begin{align*}
       \mathbb{E}\Big[ G_{I_{\tau,r}} G_{I_{\iota,r'}} F_{I_{\tau,r}} (G) F_{I_{\iota,r'}}(G)\Big] &= \mathbb{E}[S_{I_{\tau,r},I_{\iota,r'}} \partial_{P_{I_{\tau,r},I_{\iota,r'}}}F_{I_{\tau,r}} (G) F_{I_{\iota,r'}}(G)]\nonumber \\
       &= \sqrt{\mathbb{E}[(S_{I_{\tau,r},I_{\iota,r'}})^2] \mathbb{E}\Big[ \Big(\partial_{P_{I_{\tau,r},I_{\iota,r'}}}F_{I_{\tau,r}} (G) F_{I_{\iota,r'}}(G) \Big)^2 \Big]}
   \end{align*}
   In the display above, $S_{I_{\tau,r}, I_{\iota,r'}}$ denotes the gaussians corresponding to the shared edges, while $P_{I_{\tau,r}, I_{\iota,r'}}$ denotes the non-overlapping edges. Note that $\mathbb{E}[(S_{I_{\tau,r},I_{\iota,r'}})^2] \leq \mathbb{E}[|z|^{4s}]$ for $z \sim \mathcal{N}(0,1)$. Thus, it suffices to control the second term. We claim that 
   \begin{align}
       \mathbb{E}\Big[ \Big(\partial_{P_{I_{\tau,r},I_{\iota,r'}}}F_{I_{\tau,r}} (G) F_{I_{\iota,r'}}(G) \Big)^2 \Big] \lesssim \frac{1}{n^{2k - 2s - r-r'}}. 
   \end{align}
   The proof follows upon combining Holder's inequality with the boundedness of $\alpha'$ and $\alpha''$. We refer to \cite[Lemma 8, Proposition 4 and Proposition 5]{chen2021convergence} for additional details. 

Plugging these estimates back into \eqref{eq:cross_product_int1}, we obtain that 
\begin{align}
    \mathbb{E}\Big[ \mathbb{E}_{g_{i \tau}} [g_{i \tau} L_{\iota}] \mathbb{E}_{g_{i \iota}}[g_{i \iota} L_{\tau}] \Big]  \lesssim \frac{1}{n^{k+1-(r+r')/2}} n^{2k - r -r' - t + \ell - 2 } \frac{1}{n^{k-s- (r+r')/2}} = \Theta\Big( \frac{1}{n^{3+t-s - \ell}} \Big). 
\end{align}

Upon observing that if $s=t=0$, then $\ell=0$, otherwise $t \geq s+1$ and $\ell \leq 2$, it follows that LHS is $O(1/n^2)$. Plugging into \eqref{eq:squared_moment} completes the proof.   
 \end{proof}

Next, we control the covariance structure between the cavity iterates $w^{[k]}$. To this end, for $1 \le k \le n-1$ and $S \in [n]_k$, define $\nu^{[k]}_S= \nu^{[k]}_{S,i}$, $i \in [n]\setminus S$, by 
\begin{equation*}
    \nu^{[k]}_{S,i}= \alpha' (\beta w^{[k]}_{S,i}+h_i, \beta^2 \sigma^2_*),
\end{equation*}
and following our convention, we define $\nu^{[k]}=\nu^{[k]}_{\emptyset}$. Define the overlap between $\langle \by \rangle_S$ and $\nu^{[k]}_S$ by 
\begin{equation}\label{eq:define_rsk}
    R^k_S= \frac{1}{n} \sum_{j \notin S} \langle y_i \rangle_S \nu^{[k]}_{S,i}.
\end{equation}
We further define $D_S:= \frac{1}{n} \sum_{j \notin S} \langle y_i \rangle^2_S$ and $E^k_S= \frac{1}{n} \sum_{j \notin S} \big(\nu^{[k]}_{S,i}\big)^2$. Define the function 
\begin{align}\label{eq:aux_function}
    \Gamma(t, \gamma_1,\gamma_2)= \ee &\alpha' (\beta Z \sqrt{\gamma_1 |t|} + \beta Z_1 \sqrt{ \gamma_1 (1-|t|)}+h, \beta^2 \sigma^2_*) \nonumber \\ &\alpha' (\beta Z \text{ sign}(t) \sqrt{\gamma_2  |t|} + \beta Z_2 \sqrt{ \gamma_2 (1-|t|)}+h, \beta^2 \sigma^2_*),
\end{align}
where $Z,Z_1,Z_2$ i.i.d. standard Gaussian. The asymptotic behavior of the quantities defined above is given by the following proposition whose proof follows from \cite[Proposition 3]{chen2021convergence} and is thus omitted for brevity.

\begin{lemma}\label{lem:chen_prop_3}
    There exists $\beta_0>0$, such that for any $\beta<\beta_0$, we have,
    \begin{align*}
&\lim\limits_{n \rightarrow \infty} \sup_{|S|=l} \ee |D_S-q_*|^2=0\\
&\lim\limits_{n \rightarrow \infty} \sup_{|S|=l} \ee |E^{[k]}_S-q_*|^2\\
&\lim\limits_{n \rightarrow \infty} \sup_{|S|=l} \ee |R^{[k]}_S-\Delta^{\circ (k-1)}(Q_*)|^2,
    \end{align*}
    where $Q_*= \sqrt{q_*} \,\ee \alpha'(\beta \sqrt{q_*}Z+h, \beta^2 \sigma^2_*)$, $\Delta(x)=\Gamma(x/q_*,q_*,q_*)$ and $\Delta^{\circ (k-1)}$ denotes the composition of $\Delta$ for $(k-1)$ times.
\end{lemma}
Finally, we are in a position to state the result regarding convergence of AMP iterations to local magnetization. Recall that $q_*$ introduced in Definition~\ref{defn:fixed_points_sk} is a function of $\beta$. For the next definition, we will make this definition explicit, and refer to the corresponding fixed point as $q_*(\beta)$. 
\begin{thm}\label{thm:amp_good}
    Assume that $\beta>0$ satisfies that for some $\delta>0$, we have
    \begin{equation*}
\lim\limits_{n \rightarrow \infty} \sup_{\beta - \delta \leq \beta' \leq \beta}\mathbb{E}[\nu_{\beta'}((R_{1,2} - q_*(\beta'))^2)] =0.
    \end{equation*}
    Then we have 
    \begin{equation}
        \lim\limits_{k \rightarrow \infty} \lim\limits_{n \rightarrow \infty} \frac{1}{n}\ee \|\langle \by \rangle - \mathbf{m}^{[k]} \|^2=0,
    \end{equation}
where $\mathbf{m}^{[k]}$ is defined by Algorithm \ref{alg:amp}.
\end{thm}
\begin{proof}
    The proof follows the same way as \cite[Theorem 3]{chen2021convergence} once we can prove that 
    \begin{align}\label{eq:delta_fixed_point}
        \lim\limits_{k \rightarrow \infty} \Delta^{\circ k} (Q_*)= q_*
    \end{align}
Note that,
\begin{align*}
    \Delta(t)= \ee &\alpha' (\beta Z \sqrt{|t|} + \beta Z_1 \sqrt{(q_*-|t|)}+h, \beta^2 \sigma^2_*) \nonumber \\ &\alpha' (\beta Z \text{ sign}(t) \sqrt{  |t|} + \beta Z_2 \sqrt{ (q_*-|t|)}+h, \beta^2 \sigma^2_*),
\end{align*}
implying, by \eqref{eq:fixed_pt}, $\Delta(q_*)=q_*$. Note that, if $t \in [-q_*,q_*]$, we have, by Cauchy-Schwarz inequality,
\begin{align*}
    |\Delta(t)| \le \ee (\alpha' (\beta Z \sqrt{q_*}+h, \beta^2 \sigma^2_*))^2=q_*.
\end{align*}
Further, we have, $0<  |\Delta'(t)| \le \beta^2 \|\alpha ''\|^4_{\infty} < 1$ for $\beta<1$, since $\|\alpha ''\|_{\infty} \le 1$. Hence, if $\Delta(t)=t$ for any $t \in [-q,q)$,  by mean value theorem, there exists some $s \in (t,q)$ with $\Delta'(s)=1$ contradicting above. Since $\Delta(0)>0$, this yields $\Delta(t)>t$ for all $t \in [-q,q]$.

Therefore $\Delta(Q_*)> Q_*$ implying, $\Delta^{\circ (k)}(Q_*)> \Delta^{\circ (k-1)}(Q_*)$. Hence the limit of \eqref{eq:delta_fixed_point} exists and equals the fixed point of $\Delta$, which is $q_*$. This completes the proof of the result.
\end{proof}

Next, we show the stability of the AMP iterations under $h$.

\begin{lemma}
   Suppose $\mathbf{h}_0, \mathbf{h}_1$ be two independent random vectors such that $\|\mathbf{h}_0-\mathbf{h}_1\|_\infty \le \varepsilon$ deterministically. Suppose the corresponding AMP iterates are denoted by $\bu^{[k]}_i, \mathbf{m}^{[k]}_i$, $i=0,1$ respectively. Then there exists $\beta_0$ such that for any $\beta<\beta_0$, and any $k \in \mathbb N  \cup \{0\}$, almost surely,
    \begin{equation}\label{eq:b13}
        \lim\limits_{\varepsilon \rightarrow 0^+}\lim\limits_{n \rightarrow \infty} \Delta_k:= \frac{1}{n} \|\bu^{[k]}_1- \bu^{[k]}_0\|^2 =0.
    \end{equation}
Further, we have \begin{equation}\label{eq:b14}
        \lim\limits_{\varepsilon \rightarrow 0^+}\lim\limits_{n \rightarrow \infty} \frac{1}{n} \|\mathbf{m}^{[k]}_1- \mathbf{m}^{[k]}_0\|^2 =0.
    \end{equation}
\end{lemma}
\begin{proof}
    Recall that under the hypothesis of the Lemma, we have, by Lemma \ref{lem:fixed_point_h_stability}, the unique fixed points \eqref{eq:fixed_pt} $(\sigma^2_i, q_i)$, $i=0,1$ satisfy  $\|(\sigma^2_1, q_1)- (\sigma^2_0, q_0)\|^2 \lesssim \varepsilon^2$. We will prove \eqref{eq:b13} using induction on $k$. 

    For $k=0,1$, $\bu^{[k]}_i=0$, $i=0,1$, implying \eqref{eq:b13} holds. For $k=2$, we have $\bu^{[2]}_i= \sqrt{\frac{q_i}{n}} \mathbf{G} \mathbf{1}$, $i=0,1$ yielding
    \begin{equation*}
        \frac{1}{n}\|\bu^{[2]}_1-\bu^{[2]}_0\|^2 \le \Big(\frac{1}{\sqrt{n}} \| \mathbf{G} \|\Big)^2(\sqrt{q_1}- \sqrt{q_0})^2 \rightarrow 0,
    \end{equation*}
    almost surely as $n \rightarrow \infty$, $\varepsilon \rightarrow 0+$, by \cite[Theorem 5.1]{bai2010spectral}.

    For the general case, suppose the hypothesis is true for $k \ge 2$. Then by definition,
    \begin{equation}\label{eq:amp_form}
        \bu^{[k+1]}_i= \frac{1}{\sqrt{n}} \mathbf{G} \alpha'(\beta \bu^{[k]}_i+\mathbf{h}_i, \beta^2 \sigma^2_i) -(\frac{\beta}{n} \sum_{j=1}^n \alpha ''(\beta u^{[k]}_{i,j}+h_{i,j}, \beta^2 \sigma^2_i))\alpha'(\beta \bu^{[k-1]}_i+\mathbf{h}_i, \beta^2 \sigma^2_i) 
    \end{equation}
To bound $\Delta_{k+1}$, note that,
\begin{align*}
    &\frac{1}{n} \|\frac{1}{\sqrt{n}} \mathbf{G} \alpha'(\beta \bu^{[k]}_1+\mathbf{h}_1, \beta^2 \sigma^2_1)- \frac{1}{\sqrt{n}} \mathbf{G} \alpha'(\beta \bu^{[k]}_0+\mathbf{h}_0, \beta^2 \sigma^2_0)\|^2\\
    &\lesssim \frac{1}{n} \Big( \frac{1}{\sqrt{n}} \| \mathbf{G}\| \Big)^2 \left(\beta^2 \|\bu^{[k]}_1- \bu^{[k]}_0\|^2 + n \varepsilon ^2 + \beta^2 (\sigma^2_1-\sigma^2_0)^2\right) \\
    &\lesssim \beta^2 \Delta_k + \varepsilon^2 (1+ \beta^2),
\end{align*}
since $(\sigma_1^2-\sigma^2_0)^2 \lesssim \varepsilon^2$ and $\alpha'$ is a uniformly Lipschitz function. Moreover, we have,
\begin{align*}
    &\left(\frac{\beta}{n} \sum_{j=1}^n \alpha ''(\beta u^{[k]}_{1,j}+h_{1,j}, \beta^2 \sigma^2_1)- \frac{\beta}{n} \sum_{j=1}^n \alpha ''(\beta u^{[k]}_{0,j}+h_{0,j}, \beta^2 \sigma^2_0) \right)^2 \\ & 
    \lesssim \beta^2 \Delta_k+ \beta^2 \varepsilon^2+  \beta^6 \varepsilon^2,
\end{align*}
since $\alpha''$ is uniformly Lipschitz. Also,
\begin{align*}
   &\frac{1}{n} \left\|\alpha'(\beta \bu^{[k-1]}_1+\mathbf{h}_1, \beta^2 \sigma^2_1) -\alpha'(\beta \bu^{[k-1]}_0+\mathbf{h}_0, \beta^2 \sigma^2_0) \right \|^2 \\
   & \lesssim \beta^2 \Delta_{k-1} + (1+\beta^2)\varepsilon^2,
\end{align*}
which finally implies, for some $C>0$,
\begin{align*}
    \Delta_{k+1} \le C\beta^2 \Delta_k+ C\beta^2 \Delta_{k-1} + C(1+\beta^2) \varepsilon^2.
\end{align*}
By induction hypothesis, $\Delta_k,\Delta_{k-1} \rightarrow 0$ as $n \rightarrow \infty$, implying  $\lim\limits_{\varepsilon \rightarrow 0^+}\lim\limits_{n \rightarrow \infty} \Delta_{k+1}=0$. Finally, since $\alpha'$ is Lipschitz, the limit \eqref{eq:b14} follows, thus completing the proof.
\end{proof}




\end{document}